\documentclass[10pt]{amsart}

\usepackage{amsmath}
\usepackage{amsthm}
\usepackage{amssymb}
\usepackage{tikz-cd}
\usepackage[bbgreekl]{mathbbol}
\usepackage{relsize}
\usepackage{wasysym}

\usepackage{xcolor}
\usepackage[colorlinks=true]{hyperref}
\hypersetup{linkcolor=black, citecolor=black, urlcolor=black}
\usepackage[color = blue!20, bordercolor = black, textsize = tiny]{todonotes}

\newtheorem{thm}{Theorem}[section]
\newtheorem{prop}[thm]{Proposition}
\newtheorem{cor}[thm]{Corollary}

\newtheorem{lem}[thm]{Lemma}
\theoremstyle{definition}
\newtheorem{df}[thm]{Definition}
\newtheorem{rmk}[thm]{Remark}
\newtheorem{exm}[thm]{Example}
\newtheorem{const}[thm]{Construction}

\numberwithin{equation}{section}

\newcommand{\E}{\mathbb{E}}
\newcommand{\F}{\mathbb{F}}

\renewcommand{\L}{\mathbb{L}}

\renewcommand{\P}{\mathbb{P}}
\newcommand{\Q}{\mathbb{Q}}
\newcommand{\R}{\mathbb{R}}
\newcommand{\Sphere}{\mathbb{S}}

\newcommand{\W}{\mathbb{W}}

\newcommand{\Z}{\mathbb{Z}}

\newcommand{\cC}{\mathcal{C}}
\newcommand{\cD}{\mathcal{D}}

\newcommand{\cN}{\mathcal{N}}
\newcommand{\cO}{\mathcal{O}}
\newcommand{\cP}{\mathcal{P}}

\newcommand{\cS}{\mathcal{S}}

\newcommand{\rD}{\mathrm{D}}

\newcommand{\rH}{\mathrm{H}}

\newcommand{\ul}{\underline}
\newcommand{\ol}{\overline}

\DeclareSymbolFontAlphabet{\mathbb}{AMSb}
\DeclareSymbolFontAlphabet{\mathbbl}{bbold} 
\newcommand{\prism}{{\mathlarger{\mathbbl{\Delta}}}}
\newcommand{\pt}{\mathrm{pt}}

\DeclareMathOperator{\im}{im}
\DeclareMathOperator{\coker}{coker}

\newcommand{\CAlg}{\mathrm{CAlg}}
\newcommand{\CRing}{\mathrm{CRing}}
\newcommand{\Ani}{\mathrm{Ani}}
\newcommand{\Ring}{\mathrm{Ring}}
\newcommand{\Fun}{\mathrm{Fun}}
\newcommand{\Mack}{\mathrm{Mack}}

\newcommand{\NAlg}{\mathrm{NAlg}}
\newcommand{\Sp}{\mathrm{Sp}}

\newcommand{\Hom}{\mathrm{Hom}}
\newcommand{\Map}{\mathrm{Map}}
\newcommand{\Poly}{\mathrm{Poly}}

\newcommand{\fib}{\mathrm{fib}}
\newcommand{\cofib}{\mathrm{cofib}}
\newcommand{\DF}{\mathrm{DF}}
\newcommand{\gr}{\mathrm{gr}}
\newcommand{\Fil}{\mathrm{Fil}}

\newcommand{\id}{\mathrm{id}}
\newcommand{\can}{\mathrm{can}}

\newcommand{\comp}{\mathrm{comp}}
\newcommand{\colim}{\mathop{\mathrm{colim}}}

\newcommand{\Fin}{\mathrm{Fin}}
\newcommand{\Spc}{\mathrm{Spc}}
\newcommand{\Span}{\mathrm{Span}}
\DeclareMathOperator{\sotimes}{\square}
\DeclareMathOperator{\Tor}{\mathrm{Tor}}

\newcommand{\Cat}{\mathrm{Cat}}
\newcommand{\QSyn}{\mathrm{QSyn}}

\newcommand{\FinGpd}{\mathrm{FinGpd}}
\newcommand{\SH}{\mathrm{SH}}

\newcommand{\HH}{\mathrm{HH}}
\newcommand{\HR}{\mathrm{HR}}
\newcommand{\HCR}{\mathrm{HCR}}
\newcommand{\HPR}{\mathrm{HPR}}
\newcommand{\THH}{\mathrm{THH}}
\newcommand{\TPR}{\mathrm{TPR}}
\newcommand{\THR}{\mathrm{THR}}
\newcommand{\TC}{\mathrm{TC}}

\newcommand{\TP}{\mathrm{TP}}
\newcommand{\TCR}{\mathrm{TCR}}

\newcommand{\KR}{\mathrm{KR}}

\newcommand{\tr}{\mathrm{tr}}
\newcommand{\res}{\mathrm{res}}

\newcommand{\Nm}{\mathrm{Nm}}

\usepackage{pdftexcmds}
\DeclareFontFamily{U}{cbgreek}{}
\DeclareFontShape{U}{cbgreek}{m}{n}{
        <-6>    grmn0500
        <6-7>   grmn0600
        <7-8>   grmn0700
        <8-9>   grmn0800
        <9-10>  grmn0900
        <10-12> grmn1000
        <12-17> grmn1200
        <17->   grmn1728
      }{}
\DeclareFontShape{U}{cbgreek}{bx}{n}{
        <-6>    grxn0500
        <6-7>   grxn0600
        <7-8>   grxn0700
        <8-9>   grxn0800
        <9-10>  grxn0900
        <10-12> grxn1000
        <12-17> grxn1200
        <17->   grxn1728
      }{}

\DeclareRobustCommand{\Qoppa}{%
  \text{\usefont{U}{cbgreek}{\normalorbold}{n}\symbol{21}}%
}
\makeatletter
\newcommand{\normalorbold}{%
  \ifnum\pdf@strcmp{\math@version}{bold}=\z@ bx\else m\fi
}
\makeatother

\begin{document}

\author{Doosung Park}
\address{Bergische Universit{\"a}t Wuppertal,
Fakult{\"a}t f\"ur Mathematik und Naturwissenschaften
\\
Gau{\ss}strasse 20, 42119 Wuppertal, Germany}
\email{dpark@uni-wuppertal.de}

\title{Syntomic cohomology and real topological cyclic homology}

\begin{abstract}
We define the motivic filtrations on real topological Hochschild homology and its companions.
In particular,
we prove that real topological cyclic homology admits a natural complete filtration whose graded pieces are equivariant suspensions of syntomic cohomology.
As an application,
assuming a real refinement of the Dundas--Goodwillie--McCarthy theorem,
we compute the the $RO(\Z/2)$-graded homotopy groups of $\mathrm{KR}(\mathbb{F}_p[x]/x^e;\mathbb{Z}_p)$,
and we compute the equivariant slices of $\Sigma^2 \tau_{\geq 1}\mathrm{KR}(\mathbb{Z}/p^n;\mathbb{Z}_p)$.
\end{abstract}
\maketitle

\section{Introduction}

The cyclotomic trace $K\to \TC$ of B\"okstedt--Hsiang--Madsen \cite{BHM93} approximates algebraic $K$-theory via topological cyclic homology.
Together with the theorem of Dundas--Goodwillie--McCarthy \cite{DGM13},
trace methods become effective tools for computing algebraic K-theory,
see e.g.\ computations of Hesselholt--Madsen on algebraic $K$-theory of truncated polynomial algebras \cite{MR1471886} and local fields \cite{MR1998478}.

\medskip

Recent advances in syntomic cohomology and prisms offer further computational tools for algebraic $K$-theory.
Bhatt, Morrow, and Scholze \cite{BMS19} provided the motivic filtration on $\TC$ whose graded pieces are syntomic cohomology $\Z_p(i)$.
Bhatt and Scholze \cite{BS22} further studied the structure of $\Z_p(i)$ by introducing prisms.
As an application,
they proved in \cite[Theorem 1.18]{BS22} that $\Z_p(i)$ is quasisyntomic locally concentrated in degree $0$.
Liu and Wang \cite{MR4493328} employed this technique to revisit topological cyclic homology of local fields.
Mathew \cite[Theorem 10.4]{M21} computed $\Z_p(i)(k[x]/x^2)$ for a perfect $\F_p$-algebra $k$ and $p>2$, and Sulyma \cite{MR4632370} generalized this to $\Z_p(i)(k[x]/x^e)$ for all $p$ and integer $e\geq 1$.
Antieau, Krause, and Nikolaus \cite{AKN} established the computation of $\pi_i \TC(\Z/p^n;\Z_p)$ and hence $K_i(\Z/p^n)$ using the techniques of syntomic cohomology.

\medskip

In contrast, only few computations for hermitian and real $K$-theories are known especially if $2$ is not invertible.
However,
as written in \cite[\S 1.1]{QS22},
a work in progress of Harpaz--Nikolaus--Shah would construct the real cyclotomic trace
\[
\KR(\cC,\Qoppa)\to \TCR(\cC,\Qoppa)
\]
from real $K$-theory to real topological cyclic homology for every Poincar\'e $\infty$-category $(\cC,\Qoppa)$.
We refer to \cite{MR4514986} and \cite{2009.07224} for hermitian and real $K$-theories of Poincar\'e categories,
and see \cite[\S B.2]{2009.07224} for the comparison with Schlichting's hermitian K-theory of exact categories \cite{zbMATH05690542} when $2$ is invertible.
Furthermore,
Harpaz, Nikolaus, and Shah claimed that the real cyclotomic trace
satisfies a real refinement of the theorem of Dundas--Goodwillie--McCarthy.
This would imply that real trace methods are effective tools for computing hermitian and real $K$-theories.

\medskip

To utilize real trace methods,
the computations of real topological Hochschild homology are also desired.
We review the recent progress on this as follows.
Dotto, Moi, Patchkoria, and Reeh \cite[Theorem 5.15]{DMPR21} computed $\THR(\F_p)$,
which refines B\"okstedt's computation of $\THH(\F_p)$.
Dotto, Moi, and Patchkoria \cite[Theorem D, Proposition 4.2, Remark 5.14]{DMP} computed $\THR(k)$ and $\TCR(k)$ for every perfect field $k$ (as observed in \cite[Proposition 6.26]{QS22}, a similar argument can be used to compute $\THR(k)$ for every perfect $\F_p$-algebra $k$).
Quigley and Shah \cite{QS22} introduced the notion of real cyclotomic spectra and show that $\THR$ is a real cyclotomic spectrum,
which refines the results of Nikolaus--Scholze \cite[Theorem 1.4, Corollary 1.5]{NS}.
Using this method,
$\TCR(k)$ can be directly computed from the real cyclotomic spectrum $\THR(k)$ for every perfect algebra $k$, see \cite[Theorem F]{QS22}.

\medskip

The purpose of this article is to relate syntomic cohomology and real topological cyclic homology,
which aims to take advantage of the computational aspects of syntomic cohomology to real and hermitian $K$-theories.
More precisely,
we provide the motivic filtrations on $\THR$ and its companions $\TCR^-$, $\TPR$, and $\TCR$.
Our main result is as follows.

\begin{thm}[Theorem \ref{TCR.1}]
\label{intro.1}
Let $A$ be a quasisyntomic ring with trivial involution.
Then there exist natural complete exhaustive multiplicative filtrations on $\THR(A;\Z_p)$, $\TCR^-(A;\Z_p)$, $\TPR(A;\Z_p)$, and $\TCR(A;\Z_p)$ whose $n$th graded pieces are
\begin{gather*}
\gr^n \THR(A;\Z_p)
\simeq
\Sigma^{n+n\sigma}
\iota \cN^n  \widehat{\prism}_A\{n\},
\\
\gr^n \TCR^-(A;\Z_p)
\simeq
\Sigma^{n+n\sigma}
\iota \cN^{\geq n} \widehat{\prism}_A\{n\},
\\
\gr^n \TPR(A;\Z_p)
\simeq
\Sigma^{n+n\sigma}
\iota \widehat{\prism}_A\{n\},
\\
\gr^n \TCR(A;\Z_p)
\simeq
\Sigma^{n+n\sigma}
\iota \Z_p(n)(A).
\end{gather*}
\end{thm}

Here,
$\iota \colon \Sp\to \Sp^{\Z/2}$ is the functor of $\infty$-categories that is left adjoint to the fixed point functor $(-)^{\Z/2}\colon \Sp^{\Z/2}\to \Sp$ from the $\infty$-category of $\Z/2$-spectra to the $\infty$-category of spectra.
The notation $\Sigma^{m+n\sigma}$ means the equivariant suspension,
where $\Sigma^1$ corresponds to the circle $S^1$ with the trivial involution, and $\Sigma^\sigma$ corresponds to the circle $S^\sigma$ with the involution given by $(x,y)\in S^1\mapsto (x,-y)$.
We refer to \cite[\S A]{HP23} for a review of equivariant homotopy theory.
We also refer to \cite[Theorem 1.12]{BMS19} for the Nygaard completed prismatic cohomology $\widehat{\prism}_A\{n\}$, the Nygaard filtration $\cN^{\geq n}$, and the Breuil--Kisin twist $\{n\}$.

\medskip

By \cite[Theorem 7.1(1)]{BMS19},
$\THH(S;\Z_p)$ is even for every quasiregular semiperfectoid ring $S$, i.e., $\pi_{*}\THH(S;\Z_p)$ is concentrated in even degrees.
This is the key property of $\THH$ that implies many results in \cite[\S 7]{BMS19}.
For $\THR$,
we need the slice filtration
\[
\cdots \to P_n \to  P_{n-1}\to \cdots
\]
in $\Sp^{\Z/2}$,
which is an equivariant generalization of the Postnikov filtration in $\Sp$.
In this article,
we use the regular slice filtration introduced by Ullman \cite{Ull13},
which is a slight adjustment of the original slice filtration of Hill--Hopkins--Ravenel \cite{HHR}.
The adjusted version is further considered in \cite{HHR21}.
The $n$th slice is $P_n^n:=\cofib(P_{n+1}\to P_n)$.
The key property of $\THR$ is that $\THR(S;\Z_p)$ is \emph{strongly even} for quasiregular semiperfectoid ring $S$ in the following sense,
which we need for the proof of Theorem \ref{TCR.1}.

\begin{thm}[Theorem \ref{THR.6}]
Let $S$ be a quasiregular semiperfectoid ring with trivial involution.
Then $\THR(S;\Z_p)$ is strongly even.
In other words,
we have natural equivalences of $\Z/2$-spectra
\begin{gather*}
P_{2n}^{2n}\THR(S;\Z_p)
\simeq
\Sigma^{n+n\sigma} \rH \ul{\pi_{2n}\THH(S;\Z_p)},
\\
P_{2n+1}^{2n+1} \THR(S;\Z_p)\simeq 0
\end{gather*}
for every integer $n$.
\end{thm}

To show this,
we need to show that the natural filtration on the real Hochschild homology in the author's joint work with Hornbostel \cite[Theorem 4.32]{HP} is indeed complete with a certain assumption,
see Theorem \ref{geometric.6}.
Another crucial ingredient is the computation of $\THR$ of perfectoid rings in \cite[Theorem 5.16]{HP}.

\medskip

We also need the following for the proof Theorem \ref{TCR.1}.

\begin{thm}[Part of Theorem \ref{TPR.6}]
The presheaves
\[
\THR(-;\Z_p),
\;
\TCR^-(-;\Z_p),
\;
\TPR(-;\Z_p)
\]
on the opposite category of quasisyntomic rings with trivial involution are quasisyntomic sheaves.
\end{thm}

Mathew \cite[Theorem 10.4]{M21} for the cases of $p>2$ and $e=2$ and Sulyma \cite[Theorem 1.1]{MR4632370} for the general case computed $\Z_p(t)(k[x]/x^e)$,
where $k$ is a perfect field of characteristic $p$.
Combining this with Theorem \ref{intro.1},
we obtain the following computation.

\begin{thm}[Theorem \ref{TCR.6}]
Let $k$ be a perfect field of characteristic $2$.
Then we have an isomorphism of abelian groups
\[
\pi_{s,w}^{\Z/2}\TCR(k[x]/x^e;\Z_2) 
\cong
\bigoplus_{i=0}^s E_{i,s-i,w}^2,
\]
where
\begin{align*}
&E_{*,*,*}^2
\cong 
(C[\tau^2]
\oplus
\rho D[\tau^2,\rho]
\oplus
\tau D'[\tau^2,\rho]
\oplus
\tfrac{\gamma}{\tau} C[\tfrac{1}{\tau^2}]
\oplus
\tfrac{\gamma}{\tau^2}D [\tfrac{1}{\tau^2},\tfrac{1}{\rho}] \oplus \tfrac{\gamma}{\tau \rho} D'[\tfrac{1}{\tau^2},\tfrac{1}{\rho}])\{y\}
\\
&\oplus 
\Z_2[\tau^2] \oplus \rho \F_2[\tau^2,\rho]
\oplus
\tfrac{\gamma}{\tau} \Z_2[\tfrac{1}{\tau^2}] \oplus \tfrac{\gamma}{\tau^2}\F_2 [\tfrac{1}{\tau^2},\tfrac{1}{\rho}]
\\
\oplus
&
\bigoplus_{t=1}^\infty
(
A_t[\tau^2]
\oplus
\rho B_t[\tau^2,\rho]
\oplus
\tau B_t'[\tau^2,\rho]
\oplus
\tfrac{\gamma}{\tau} A_t[\tfrac{1}{\tau^2}]
\oplus
\tfrac{\gamma}{\tau^2}B_t [\tfrac{1}{\tau^2},\tfrac{1}{\rho}] \oplus \tfrac{\gamma}{\tau \rho} B_t'[\tfrac{1}{\tau^2},\tfrac{1}{\rho}]
)
\{x_t\},
\end{align*}
$\lvert \tau\rvert =(0,0,-1)$,
$\lvert \rho \rvert = (-1,0,-1)$,
$\lvert \gamma \rvert = (0,0,1)$,
$\lvert x_t\rvert = (t-1,t,t)$,
$\lvert y \rvert = (-1,0,0)$,
$A_t:=\W_{et}(k)/V_e\W_{t}(k)$, $\W_{m}$ denotes the ring of $m$-truncated big Witt vectors for an integer $m$.
$V_e$ is the $e$th Verschiebung operator,
$B_t:=\coker(2\colon A_t\to A_t)$,
$B_t':=\ker(2\colon A_t\to A_t)$,
$C:=\coker(1-F)$ with the Frobenius $F$ on $W(k)$,
$D:=\coker(2\colon C\to C)$,
and $D':=\ker(2\colon C\to C)$.
\end{thm}

We also compute the equivariant slices of $\Sigma^2 \tau_{\geq 1}\mathrm{KR}(\mathbb{Z}/p^n;\mathbb{Z}_p)$ using the computation of Antieau--Krause--Nikolaus \cite{AKN},
see Example \ref{TCR.3}.

\medskip

The claimed real refinement of the Dundas--Goodwillie--McCarthy theorem would yield an equivalence between $\mathrm{KR}(A;\Z_p)$ and $\tau_{\geq 0}\TCR(A;\Z_p)$ for several classes of commutative rings $A$ with trivial involution,
where $(-;\Z_p)$ means the $p$-completion.
In this case,
the aforementioned computations valid for $\mathrm{KR}(A;\Z_p)$ after taking $\tau_{\geq 0}$.

\medskip

Recently,
Angelini-Knoll, Kong, and Quigley uploaded their preprint \cite{AKQ} on real syntomic cohomology in arXiv.
See \cite[Theorem 6.7]{AKQ} for the comparison between their filtration and ours.

\medskip

\subsection*{Convention} The involution on a commutative ring is regarded as the trivial one if we do not specify the involution.

\medskip

\subsection*{Acknowledgement} We thank Jens Hornbostel for many helpful conversations on this topic and comments on the draft.
We also thank Gabriel Angelini-Knoll and the referee for encouraging us to carry out the computations of $\pi_{*,*}^{\Z/2}\TCR(A;\Z_p)$ for quasisyntomic rings $A$.
Finally,
we are grateful to the referee for the helpful comments leading to improvements of this text.
This research was conducted in the framework of the DFG-funded research training group GRK 2240: \emph{Algebro-Geometric Methods in Algebra,
Arithmetic and Topology}.

\section{Reminders on equivariant homotopy theory}

In this section,
we recall equivariant stable homotopy theory.
We refer to \cite{HHR21} for a comprehensive book and \cite[\S 9]{BH21} for the $\infty$-categorical construction.
See  \cite[\S A]{HP23} for a review.

Let $\FinGpd$ denote the $2$-category of finite groupoids.
For $X\in \FinGpd$,
let $\Fin_X$ denote the $1$-category of finite coverings of $X$ whose morphisms are also finite coverings between finite coverings.
For an $\infty$-category $\cC$,
let $\Span(\cC)$ denote the $\infty$-category of spans \cite[\S 5]{Bar17},
whose morphism from $X$ to $Z$ has the form
\[
X\xleftarrow{f} Y\xrightarrow{p}Z.
\]
In this case,
$f$ is called a \emph{backward morphism},
and $p$ is called a \emph{forward morphism}.
According to \cite[\S 9.2]{BH21},
we have the functor
\[
\SH \colon \Span(\FinGpd)\to \Cat_\infty,
\;
(X\xleftarrow{f}Y\xrightarrow{p}Z)\mapsto p_\otimes f^*
\]
such that $\SH (B G)$ is the stable $\infty$-category $\Sp^G$ of $G$-spectra for every finite group $G$.
We write $\wedge$ for the monoidal product and $\Sphere$ for the monoidal unit in $\Sp^G$.
For a morphism $f$ in $\FinGpd$,
$f^*$ admits a right adjoint $f_*$.
If $f$ is a finite covering,
then $f^*$ admits a left adjoint $f_!$.

If $i\colon *\to B (\Z/2)$ is the unique finite covering,
then $i$ induces the adjunction pair
\[
i_! : \Sp \rightleftarrows \Sp^{\Z/2} : i^*.
\]
Our convention for such diagrams is that $i_!$ is left adjoint to $i^*$.
The restriction functor $i^*$ is the forgetful functor,
and the functor $i_\otimes$ is the norm functor $N_e^{\Z/2}\colon \Sp\to \Sp^{\Z/2}$.

If $p\colon B(\Z/2)\to *$ is the unique morphism,
then $p$ induces the adjunction pair
\[
\iota : \Sp\rightleftarrows \Sp^{\Z/2} : (-)^{\Z/2},
\]
where $(-)^{\Z/2}:=p^*$ is the fixed point functor.
The functor $p_\otimes$ is the geometric fixed point functor $\Phi^{\Z/2}\colon \Sp^{\Z/2}\to \Sp$.
Since $pi=\id$,
we have $\id \simeq p_\otimes i_\otimes =\Phi^{\Z/2}N_e^{\Z/2}$.

The pair of functors $(i^*,(-)^{\Z/2})$ is conservative,
and the pair of functors $(i^*,\Phi^{\Z/2})$ is conservative too.

We have the infinite suspension functor $\Sigma^\infty$ from the $\infty$-category of pointed $\Z/2$-spaces to $\Sp^{\Z/2}$.
For a $\Z/2$-spectrum $X$ and an integer $n$,
we set $\Sigma^\sigma X:=\Sigma^\infty S^\sigma\wedge X$.
The functor $\Sigma^\sigma\colon \Sp^{\Z/2}\to \Sp^{\Z/2}$ is an equivalence of $\infty$-categories,
and let $\Sigma^{-\sigma}$ denote its quasi-inverse functor.
For integers $m$ and $n$,
we have the equivariant suspension functor $\Sigma^{m+n\sigma}\colon \Sp^{\Z/2}\to \Sp^{\Z/2}$ built  from $\Sigma^m$, $\Sigma^\sigma$, and $\Sigma^{-\sigma}$.

A \emph{Mackey functor $M$ for $\Z/2$} is a diagram of abelian groups
\[
\begin{tikzcd}
M(\Z/2) \ar[r,shift left=0.75ex,"\tr"]\ar[r,shift right=0.75ex,"\res"',leftarrow]\ar[loop left, out=190, in=170,looseness=7,"w"]&
M(\pt)
\end{tikzcd}
\]
satisfying
\[
w\circ w=\id,
\;
\res\circ\tr=1+w,
\;
\tr\circ w=\tr,
\;
w\circ \res=\res.
\]
Let $\Mack_{\Z/2}$ denote the category of Mackey functors for $\Z/2$.

Let $\ul{\pi}_n(X)$ be the equivariant $n$th homotopy group of $X$,
which is a Mackey functor.
We set $\pi_n^{\Z/2}(X):=\ul{\pi}_n(X)(\pt)$.
There is a natural isomorphism
\[
\pi_n^{\Z/2}(X)\cong \Hom_{\Sp^{\Z/2}}(\Sigma^n \Sphere,X).
\]

For a Mackey functor $M$, the \emph{equivariant Eilenberg--MacLane spectrum $\rH M$} is a unique $\Z/2$-spectrum satisfying
\[
\ul{\pi}_i(\rH M)
\cong
\left\{
\begin{array}{ll}
M & \text{if $i=0$},
\\
0 & \text{otherwise}.
\end{array}
\right.
\]
For Mackey functors $M$ and $M'$, we have the \emph{box product} given by
\[
M\sotimes M':=\ul{\pi}_0(\rH M\wedge \rH M').
\]
This gives a monoidal structure on $\Mack_{\Z/2}$.
A \emph{Green functor} is a monoid object of $\Mack_{\Z/2}$.

For $X\in \FinGpd$, a \emph{normed $X$-spectrum} is a section of $\SH$ over $\Span(\Fin_X)$ that is cocartesian over backward morphisms, see \cite[Definition 9.14]{BH21}.
Let $\NAlg(\SH(X))$ denote the $\infty$-category of normed $X$-spectra.
As observed in \cite[\S A.1]{HP23},
the forgetful functor $\NAlg(\SH(X))\to \SH(X)$ is conservative.
If $f\colon X\to S$ is a finite covering in $\FinGpd$,
then the functor $f^*\colon \NAlg(\SH(S))\to \NAlg(\SH(X))$ admits a left adjoint $f_\otimes$ by \cite[Proposition A.1.11]{HP23} and a right adjoint $f_*$ by \cite[Proposition A.1.12]{HP23}.
For abbreviation,
we set
\[
\NAlg^{\Z/2}:=\NAlg(\SH(B (\Z/2))).
\]
A \emph{normed $\Z/2$-spectrum} is an object of this $\infty$-category.
On the other hand,
we have $\CAlg(\Sp)\simeq \NAlg(\SH(\pt))$ by \cite[Example 9.15]{BH21},
where $\CAlg(\Sp)$ is the $\infty$-category of $\E_\infty$-rings.

\section{Reminders on \texorpdfstring{$\Z/2$-$\infty$-}{Z/2-infinity }categories and real topological cyclic homology}

Let us review the theory of $\Z/2$-$\infty$-categories following \cite{MR4587313}, \cite{2109.11954}, and \cite{QS21}.

Let $\cO_{\Z/2}^{op}$ denote the opposite category of the orbit category of $\Z/2$,
which can be described as the diagram
\[
\begin{tikzcd}
\Z/2 \ar[r,leftarrow,"\res"]\ar[loop left,"w"]&
\pt
\end{tikzcd}
\]
with the relations $w\circ \res=\res$ and $w\circ w=\id$.
We have the functor $\cO_{\Z/2}^{op}\to \Fin_{B(\Z/2)}$ sending $\pt$ to $\pt$, $\Z/2$ to $B(\Z/2)$, and $\res$ to the finite covering $\pt\to B(\Z/2)$.

To avoid a set-theoretical issue,
we fix two Grothendieck universes $\mathbb{U}\in \mathbb{V}$.
A \emph{$\Z/2$-space} is a functor,
\[
X\colon \cO_{\Z/2}^{op}\to \Spc,
\]
where $\Spc$ denotes the $\infty$-category of $\mathbb{U}$-small spaces.
A \emph{$\Z/2$-$\infty$-category} is a functor
\[
\cC\colon \cO_{\Z/2}^{op}\to \Cat_\infty,
\]
where $\Cat_\infty$ denotes the $\infty$-category of $\mathbb{V}$-small $\infty$-categories.
For example, the $\infty$-category $\Sp$ of $\mathbb{U}$-small spectra is not $\mathbb{U}$-small but $\mathbb{V}$-small.

We can view a $\Z/2$-space as a $\Z/2$-$\infty$-category.
For $\Z/2$-$\infty$-categories $\cC$ and $\cD$,
a \emph{$\Z/2$-functor $\cC\to \cD$} is a natural transformation.
Let $\Fun_{\Z/2}(\cC,\cD)$ be the $\infty$-category of $\Z/2$-functors $\cC\to \cD$.

\begin{exm}
Let $X$ be a topological $\Z/2$-space, i.e., topological space with $\Z/2$-action.
Then we have the natural functor $X\colon \cO_{\Z/2}^{op}\to \Spc$ such that $X(\Z/2)$ is $X$, $X(\pt)$ is the $\Z/2$-fixed point $X^{\Z/2}$,
$X(\res)$ is the inclusion $X^{\Z/2}\to X$,
and $X(w)$ is the $\Z/2$-action on $X$.
We can also recover a topological $\Z/2$-space from a functor $\cO_{\Z/2}^{op}\to \Spc$ using Elmendorf's equivalence \cite[Theorem 8.8.1]{HHR21}.

Let $*$ denote the $\Z/2$-space with trivial $\Z/2$-action.
\end{exm}

\begin{exm}
By restricting the functor $\mathrm{SH}$ to $\cO_{\Z/2}^{op}$,
we obtain the $\Z/2$-$\infty$-category
\[
\ul{\Sp}^{\Z/2}
\colon
\cO_{\Z/2}^{op}
\to
\Cat_\infty.
\]
If we regard $\ul{\Sp}^{\Z/2}$ as a cocartesian fibration over $\cO_{\Z/2}^{op}$,
then the fiber at $\Z/2\in \cO_{\Z/2}^{op}$ is $\Sp$,
and the fiber at $\pt\in \cO_{\Z/2}^{op}$ is $\Sp^{\Z/2}$.
Furthermore, $\ul{\Sp}^{\Z/2}(\res)$ corresponds to $i^*\colon \Sp^{\Z/2}\to \Sp$.

Adapt \cite[Example 8.7]{BH21} to the non-motivic setting to similarly obtain the $\Z/2$-$\infty$-category
\[
\ul{\NAlg}^{\Z/2}
\colon
\cO_{\Z/2}^{op}
\to
\Cat_\infty.
\]
If we regard $\ul{\NAlg}^{\Z/2}$ as a cocartesian fibration over $\cO_{\Z/2}^{op}$,
then the fiber at $\Z/2\in \cO_{\Z/2}^{op}$ is $\CAlg(\Sp)$,
and the fiber at $\pt\in \cO_{\Z/2}^{op}$ is $\NAlg^{\Z/2}$.
Furthermore, $\ul{\NAlg}^{\Z/2}(\res)$ corresponds to $i^*\colon \NAlg^{\Z/2}\to \CAlg(\Sp)$.
\end{exm}

\begin{prop}
\label{G-cat.2}
Let $\cC$ be a $\Z/2$-$\infty$-category,
which we regard as a cocartesian fibration $p\colon \cC\to \cO_{Z/2}^{op}$.
Then $\Fun_{\Z/2}(*,\cC)$ is equivalent to the fiber of $\cC$ at $\pt$.
\end{prop}
\begin{proof}
This is a special case of \cite[Lemma 2.12]{1608.03657}. We include the proof of this special case as follows.
If we regard the $\Z/2$-space $*$ as a cartesian fibration over $\cO_{\Z/2}$,
then this is the identity functor $\cO_{\Z/2}\to \cO_{\Z/2}$.
Hence we can identify $\Fun_{\Z/2}(*,\cC)$ with the $\infty$-category of cartesian sections of $p^{op}\colon \cC^{op}\to \cO_{\Z/2}$.
Together with \cite[Corollary 3.3.3.2]{HTT},
we have an equivalence of $\infty$-categories $\lim(\cC \colon \cO_{\Z/2}^{op}\to \Cat_\infty)\simeq \Fun_{\Z/2}(*,\cC)$.
To conclude,
observe that $\pt$ is an initial object of $\cO_{\Z/2}^{op}$.
\end{proof}

\begin{exm}
The $\Z/2$-space $S^\sigma$ has the monoid structure induced by the multiplication $S^1\times S^1\to S^1$.
Let $B S^\sigma$ be the classifying $\Z/2$-space of $S^\sigma$,
which is given by the colimit of the bar construction
\[
\colim
\big(
\cdots
\,
\substack{\rightarrow\\[-1em] \rightarrow \\[-1em] \rightarrow}
\,
S^\sigma \times S^\sigma
\,
\substack{\rightarrow\\[-1em] \rightarrow}
\,
S^\sigma
\big)
\]
This is written as $B_{C_2}^tS^1$ in \cite{QS22}.
\end{exm}

\begin{exm}
For an integer $m\geq 1$,
Let $C_m^\sigma$ be the $\Z/2$-subspace of $S^\sigma$ consisting of $e^{2\pi n i/m}$ for $0\leq n\leq m-1$.
The monoid structure on $S^\sigma$ induces a monoid structure on $C_m^\sigma$.
Let $B C_m^\sigma$ be the classifying space of $C_m^\sigma$.
\end{exm}

\begin{df}
The \emph{$\infty$-category of $\Z/2$-spectra with $S^\sigma$-action} is 
\[
(\Sp^{\Z/2})^{B S^\sigma}
:=
\Fun_{\Z/2}(B S^\sigma,\ul{\Sp}^{\Z/2}).
\]
This is written as $\mathbf{Sp}^{h_{C_2}S^1}$ in \cite{QS22}.
An \emph{$S^\sigma$-equivariant map of $\Z/2$-spectra with $S^\sigma$-action} is a map in this $\infty$-category.

The \emph{$\infty$-category of normed $\Z/2$-spectra with $S^\sigma$-action} is
\[
(\NAlg^{\Z/2})^{B S^\sigma}
:=
\Fun_{\Z/2}(BS^\sigma,\ul{\NAlg}^{\Z/2}).
\]
An \emph{$S^\sigma$-equivariant map of normed $\Z/2$-spectra with $S^\sigma$-action} is a map in this $\infty$-category.
\end{df}

\begin{prop}
\label{G-cat.1}
Let $\phi\colon I\to J$ be a map of $\Z/2$-spaces.
Then the restriction functor
\[
\phi^*\colon \Fun_{\Z/2}(J,\ul{\Sp}^{\Z/2})
\to
\Fun_{\Z/2}(I,\ul{\Sp}^{\Z/2})
\]
admits a left adjoint $\phi_!$ and a right adjoint $\phi_*$,
and the restriction functor
\[
\phi^*\colon \Fun_{\Z/2}(J,\ul{\NAlg}^{\Z/2})
\to
\Fun_{\Z/2}(I,\ul{\NAlg}^{\Z/2})
\]
admits a left adjoint $\phi_\otimes$ and a right adjoint $\phi_*$.
\end{prop}
\begin{proof}
By \cite[Observation 4.2]{QS21},
it suffices to check the conditions in \cite[Theorem B]{2109.11954} and its dual statement.

The conditions (1) and (2) for $\ul{\Sp}^{\Z/2}$ are satisfied since for every finite covering $f\colon X\to S$ in $\FinGpd$,
$f^*\colon \SH(S)\to \SH(X)$ admits a left adjoint $f_!$ and a right adjoint $f_*$.
The condition (3) for $\ul{\Sp}^{\Z/2}$ is satisfied by \cite[Proposition A.1.9]{HP23}.

The conditions (1) and (2) for $\ul{\NAlg}^{\Z/2}$ are satisfied since for every finite covering $f\colon X\to S$ in $\FinGpd$, $f^*\colon \NAlg(\SH(S))\to \NAlg(\SH(X))$ admits a left adjoint $f_\otimes$ and a right adjoint $f_*$ by \cite[Propositions A.1.11, A.1.12]{HP23}.
The forgetful functor $\NAlg(\SH(X))\to \SH(X)$ is conservative for every $X\in \FinGpd$ as observed in \cite[\S A.1.1]{HP23}.
To check the condition (3)  for $\ul{\NAlg}^{\Z/2}$,
using the diagrams in  \cite[Propositions A.1.11, A.1.12]{HP23},
it suffices to check $(-)_\otimes \circ (-)^* \xrightarrow{\simeq} (-)^* \circ (-)_\otimes$ for $\SH$.
This is encoded in the construction of $\SH\colon \Span(\FinGpd)\to  \Cat_\infty$.\footnote{This functor $\SH$ is denoted by $\SH^\otimes$ in \cite{HP23}.}
\end{proof}

Let $j\colon B S^\sigma\to *$ be the trivial map of $\Z/2$-spaces.
Then by Propositions \ref{G-cat.2} and \ref{G-cat.1}, we have the adjunctions
\[
\begin{tikzcd}
(\Sp^{\Z/2})^{B S^\sigma}
\ar[r,shift left=1.5ex,"j_!"]
\ar[r,"j^*" description,leftarrow]
\ar[r,"j_*"',shift right=1.5ex]
&
\Sp^{\Z/2}.
\end{tikzcd}
\]
The \emph{$S^\sigma$-homotopy orbit functor} is $(-)_{hS^\sigma}:=j_!$.
The \emph{$S^\sigma$-homotopy fixed point functor} is $(-)^{hS^\sigma}:=j_*$.
On the other hand,
$j^*$ imposes the trivial $S^\sigma$-action.

{For an integer $m\geq 1$, let $w_m\colon BS^\sigma\to B(S^\sigma/C_m^\sigma)\xrightarrow{\simeq} BS^\sigma$ be the composite map of $\Z/2$-spaces,
where the first map is induced by the quotient map $S^\sigma\to S^\sigma/C_m^\sigma$.
Then by Propositions \ref{G-cat.2} and \ref{G-cat.1},
we have the adjunctions
\[
\begin{tikzcd}
(\Sp^{\Z/2})^{B S^\sigma}
\ar[r,shift left=1.5ex,"w_{m!}"]
\ar[r,"w_m^*" description,leftarrow]
\ar[r,"w_{m*}"',shift right=1.5ex]
&
(\Sp^{\Z/2})^{B S^\sigma}
\end{tikzcd}
\]
The \emph{$C_m^\sigma$-homotopy orbit functor} is $(-)_{hC_m^\sigma}:=w_{m!}$.
The \emph{$C_m^\sigma$-homotopy fixed point functor} is $(-)^{hC_m^\sigma}:=w_{m*}$.

The map of $\Z/2$-spaces $q\colon *\to B S^\sigma$ to the base point induces the forgetful functors
\begin{gather*}
q^*\colon (\Sp^{\Z/2})^{BS^\sigma}\to \Sp^{\Z/2},
\\
q^*\colon (\NAlg^{\Z/2})^{BS^\sigma}\to \NAlg^{\Z/2}.
\end{gather*}

\begin{prop}
\label{G-cat.3}
The above forgetful functors $q^*$ preserve colimits and are conservative.
\end{prop}
\begin{proof}
We focus on $\Sp^{\Z/2}$ since the proofs are similar.
If we regard $BS^\sigma$ as a cocartesian fibration over $\cO_{\Z/2}^{op}$,
then $BS^\sigma$ have exactly two connected components since $BS^1$ and $B(\Z/2)$ are connected.
We also regard $\ul{\Sp}^{\Z/2}$ as a cocartesian fibration over $\cO_{\Z/2}^{op}$.
By \cite[Corollary 5.1.2.3]{HTT},
the evaluation functors
\begin{gather*}
ev_1\colon 
(\Sp^{\Z/2})^{BS^\sigma}
\to
\Fun_{\cO_{\Z/2}^{op}}({\{v_1\},\ul{\Sp}^{\Z/2}})
\simeq
\Sp,
\\
ev_2\colon 
(\Sp^{\Z/2})^{BS^\sigma}
\to
\Fun_{\cO_{\Z/2}^{op}}({\{v_2\},\ul{\Sp}^{\Z/2}})
\simeq
\Sp^{\Z/2}
\end{gather*}
corresponding to two vertices $v_1$ and $v_2$ in the two connected components of $BS^\sigma$
preserve colimits,
and they form a conservative family of functors.
To conclude, observe that $ev_2\simeq q^*$ and $ev_1$ factors through $ev_2$.
\end{proof}

\begin{df}
\label{remind.1}
Let $X$ be a $\Z/2$-spectrum with $S^\sigma$-action.
The \emph{Tate construction of $X$} is
\[
X^{tS^\sigma}
:=
\cofib(
\Sigma^\sigma X_{hS^\sigma}\xrightarrow{\Nm} X^{hS^\sigma}),
\]
where $\Sigma^\sigma X_{hS^\sigma}\xrightarrow{\Nm}X^{hS^\sigma}$ is the norm map,
see \cite[Example 5.58]{QS21} for the details.
According to the proof of \cite[Lemma 4.6]{QS22},
taking $\Phi^{\Z/2}$ to the \emph{norm cofiber sequence}
\[
\Sigma^\sigma X_{hS^\sigma}\xrightarrow{\Nm} X^{hS^\sigma}\to X^{tS^\sigma}
\]
yields the non-equivariant norm cofiber sequence for $C_2$:
\[
(\Phi^{\Z/2} X)_{hC_2}
\to
(\Phi^{\Z/2} X)^{hC_2}
\to
(\Phi^{\Z/2} X)^{tC_2}.
\]

We similarly have the norm cofiber sequence
\[
X_{hC_m^\sigma}\xrightarrow{\Nm} X^{hC_m^\sigma}\to X^{tC_m^\sigma}
\]
for every integer $m\geq 1$,
see \cite[Observation 1.9]{QS21}.
\end{df}

\begin{prop}
\label{remind.2}
Let $X$ be a $\Z/2$-spectrum.
If we impose the trivial $S^\sigma$-action on $X$,
then we have equivalences of $\Z/2$-spectra
\begin{gather*}
X_{hS^\sigma}
\simeq
\Sigma^\infty (B S^\sigma)_+ \wedge X,
\\
X^{hS^\sigma}
\simeq
F(\Sigma^\infty (B S^\sigma)_+, X),
\end{gather*}
where $F(-,-)$ denotes the internal hom in $\Sp^{\Z/2}$.
\end{prop}
\begin{proof}
The first one is a consequence of the projection formula \cite[Lemma 5.45]{QS21}.
To use this for $\Z/2$-spectra,
we also need \cite[Example 5.19]{QS21}.
We obtain the second one by adjunction.
\end{proof}

Next,
we recall the definition of real topological Hochschild homology following \cite[\S 5]{QS22}.

\begin{df}
Let $A$ be a normed $\Z/2$-spectrum.
The \emph{real topological Hochschild homology of $A$} is
\[
\THR(A)
:=
q_\otimes A
\in
(\NAlg^{\Z/2})^{BS^\sigma},
\]
where $q\colon *\to BS^\sigma$ is the inclusion functor to the base point,
$q^*\colon (\NAlg^{\Z/2})^{B S^\sigma}\to \NAlg^{\Z/2}$ is the restriction functor,
and $q_\otimes$ is the left adjoint of $q^*$ that exists by Proposition \ref{G-cat.1}.

\begin{prop}
Let $A$ be a normed $\Z/2$-spectrum.
After forgetting the $S^\sigma$-action,
$\THR(A)$ is the $\Z/2$-tensor of $A$ with the $\Z/2$-space $S^\sigma$.
\end{prop}
\begin{proof}
Consider the cartesian square
\[
\begin{tikzcd}
S^\sigma\ar[r,"r"]\ar[d,"r"']&
*\ar[d,"q"]&
\\
*\ar[r,"q"]&
B S^\sigma.
\end{tikzcd}
\]
What we mean by the statement is $q^*q_\otimes A\simeq r_\otimes r^*A$,
where
\[
r^*\colon \NAlg^{\Z/2}\to \Fun_{\Z/2}(S^\sigma,\ul{\NAlg}^{\Z/2})
\]
is the restriction functor,
and $r_\otimes$ is the left adjoint of $r^*$.
This is a consequence of \cite[Lemma 4.4]{QS21},
which is applicable to $\ul{\NAlg}^{\Z/2}$ by the proof of Proposition \ref{G-cat.1}.
\end{proof}

Let $A$ be a normed $\Z/2$-spectrum.
If we forget the $S^\sigma$-action on $\THR(A)$,
then we have an equivalence of normed $\Z/2$-spectra
\[
\THR(A)
\simeq
A\otimes_{ N_e^{\Z/2}i^* A} A
\]
by \cite[Remark 5.3]{QS22}.
This is the definition that is employed in \cite{HP23} and \cite{HP}.
The inclusion $*\to S^\sigma$ to the base point and the collapsing map $S^\sigma\to *$ induce $S^\sigma$-equivariant maps of normed $\Z/2$-spectra
\[
A\to \THR(A)\to A.
\]
According to \cite[Construction 5.5]{QS22},
we have the \emph{real $p$-cyclotomic Frobenius}
\[
\varphi_p\colon \THR(A)\to \THR(A)^{t C_p^\sigma}.
\]

We also have the companions of $\THH(A)$:
\begin{gather*}
\TCR^-(A)
:=
\THR(A)^{hS^\sigma}\in \Sp^{\Z/2},
\\
\TPR(A)
:=
\THR(A)^{tS^\sigma}\in \Sp^{\Z/2}.
\end{gather*}
\end{df}

\begin{df}
If $R$ is a commutative ring,
then we set $\THR(R):=\THR(\rH \ul{R})$,
where $\ul{R}$ denotes the constant Tambara functor associated with $R$.
We will follow a similar convention for $\TCR^-$, $\TPR$, and their $p$-complete variants too.
\end{df}

\section{\texorpdfstring{$p$}{p}-completions}

In this section,
we review $p$-completions in $\Sp^{\Z/2}$.
We refer to \cite[\S 2]{HP} for a similar one in the case of equivariant derived $\infty$-categories.
We will apply $p$-completions to $\THR$ and its companions.

\begin{df}
\label{comp.1}
Let $X$ be a $\Z/2$-spectrum.
Following \cite[Definitions 7.1.1.1, 7.2.4.1, 7.3.1.1]{SAG},
we say that $X$ is
\begin{enumerate}
\item[(1)] $p$-nilpotent if
\(
X[1/p]
:=
\colim(X\xrightarrow{p} X\xrightarrow{p} \cdots)
\simeq 0,
\)
\item[(2)] $p$-local if $\Map_{\Sp^{\Z/2}}(Y,X)\simeq 0$ for every $p$-nilpotent $\Z/2$-spectrum $Y$,
\item[(3)] $p$-complete  if $\Map_{\Sp^{\Z/2}}(Y,X)\simeq 0$ for every $p$-local $\Z/2$-spectrum $Y$.
\end{enumerate}
We have similar definitions for spectra too.
Let $\Sp_{p-\comp}$ (resp.\ $\Sp_{p-\comp}^{\Z/2}$) be the full subcategory of $\Sp$ (resp.\ $\Sp^{\Z/2}$) spanned by the $p$-complete spectra (resp.\ $p$-complete $\Z/2$-spectra).
The inclusion functor $\Sp_{p-\comp}^{\Z/2}\to \Sp^{\Z/2}$ admits a left adjoint
\[
(-)_p^\wedge \colon \Sp^{\Z/2}\to \Sp_{p-\comp}^{\Z/2},
\]
which we call the \emph{$p$-completion}.
\end{df}

\begin{prop}
\label{comp.2}
Let $X$ be a $\Z/2$-spectrum.
\begin{enumerate}
\item[\textup{(1)}]
We have
$X_p^\wedge\simeq \cofib(\lim (\cdots \xrightarrow{p} X\xrightarrow{p} X)\to X)$.
\item[\textup{(2)}]
The fiber of the canonical map $X\to X_p^\wedge$ is $p$-local.
\item[\textup{(3)}] $X$ is $p$-local if and only if $X_p^\wedge \simeq 0$.
\item[\textup{(4)}] $X$ is $p$-local if and only if $i^*X$ and $X^{\Z/2}$ are $p$-local.
\item[\textup{(5)}] $X$ is $p$-local if and only if the multiplication $p\colon X\to X$ is an equivalence.
\item[\textup{(6)}] $X$ is $p$-complete if and only if $i^*X$ and $X^{\Z/2}$ are $p$-complete.
\end{enumerate}
\end{prop}
\begin{proof}
(1) is a consequence of \cite[Proposition 7.3.2.1]{SAG},
and (2) and (3) are consequences of \cite[Proposition 7.3.1.4]{SAG}. (4) and (6) are consequences of (1) and (3) since the pair of functors $i^*$ and $(-)^{\Z/2}$ is conservative and they preserve limits and cofibers.

(5) If $p\colon X\to X$ is an equivalence,
then $X$ is $p$-local by (1) and (3).
Conversely, assume that $X$ is $p$-local.
Then $i^*X$ and $X^{\Z/2}$ are $p$-local by (4).
Since $\Sigma^n\Sphere/p:=\Sigma^n\cofib(\Sphere\xrightarrow{p}\Sphere)\in \Sp$ is $p$-nilpotent for $n\in \Z$,
$i^*X\xrightarrow{p} i^*X$ and $X^{\Z/2}\xrightarrow{p} X^{\Z/2}$ are equivalences.
It follows that $X\xrightarrow{p} X$ is an equivalence.
\end{proof}

\begin{const}
\label{comp.7}
Let $\ul{\Sp}_{p-\comp}^{\Z/2}$ be the full $\Z/2$-subcategory of $\ul{\Sp}^{\Z/2}$ such that the fiber of $\ul{\Sp}_{p-\comp}^{\Z/2}$ at $\Z/2$ is $\Sp_{p-\comp}$ and at $\pt$ is $\Sp_{p-\comp}^{\Z/2}$.
The inclusion $\Z/2$-functor $\ul{\Sp}_{p-\comp}^{\Z/2}\to \ul{\Sp}^{\Z/2}$ admits a left adjoint
\[
(-)_p^\wedge
\colon
\ul{\Sp}^{\Z/2}
\to
\ul{\Sp}_{p-\comp}^{\Z/2},
\]
which is the pointwise $p$-completion.
We have the induced functor
\[
(-)_p^\wedge
\colon
(\Sp^{\Z/2})^{B S^\sigma}
\to
(\Sp_{p-\comp}^{\Z/2})^{B S^\sigma}
:=
\Fun_{\Z/2}(B S^\sigma,\ul{\Sp}_{p-\comp}^{\Z/2}).
\]
Hence for every $\Z/2$-spectrum $X$ with $S^\sigma$-action,
we can regard the completion $X_p^\wedge$ as a $\Z/2$-spectrum with $S^\sigma$-action.
\end{const}

\begin{prop}
\label{comp.3}
Let $X$ be a $p$-complete $\Z/2$-spectrum with $S^\sigma$-action.
Then $X^{hS^\sigma}$ is $p$-complete.
If we assume that $i^*X$ is bounded below,
then $X_{hS^\sigma}$ and $X^{tS^\sigma}$ are $p$-complete too.
\end{prop}
\begin{proof}
We have the induced commutative square
\[
\begin{tikzcd}
\Sp_{p-\comp}^{\Z/2}\ar[d,"(-)_p^\wedge"']\ar[r,"j^*"]&
(\Sp_{p-\comp}^{\Z/2})^{BS^\sigma}\ar[d,"(-)_p^\wedge"]
\\
\Sp^{\Z/2}\ar[r,"j^*"]&
(\Sp^{\Z/2})^{BS^\sigma},
\end{tikzcd}
\]
where $j^*$ imposes the trivial $S^\sigma$-action.
By adjunction,
we see that $(-)^{hS^\sigma}$ preserves $p$-complete objects.
In particular,
$X^{hS^\sigma}$ is $p$-complete.
Due to the norm cofiber sequence $\Sigma^\sigma X_{hS^\sigma}\to X^{hS^\sigma}\to X^{tS^\sigma}$,
it suffices to show that $X_{hS^\sigma}$ is $p$-complete if $i^*X$ is bounded below.

By \cite[Lemma 3.16]{QS22},
$i^*X$ and $\Phi^{\Z/2}X$ are $p$-complete.
We have an equivalence of spectra
\(
i^*(X_{hS^\sigma})\simeq (i^*X)_{hS^1},
\)
which is $p$-complete as noted in \cite[p.\ 216]{BMS19}.
We also have an equivalence of spectra
\(
\Phi^{\Z/2} (X_{hS^\sigma})\simeq (\Phi^{\Z/2}X)_{hC_2}
\)
as noted in Definition \ref{remind.1},
which is $p$-complete by \cite[Lemma I.2.9]{NS}, the norm cofiber sequence for $C_2$, and the fact $(-)^{hC_2}$ preserves $p$-complete objects.
Since $(i^*X)_{hS^1}$ is bounded below,
we see that $X_{hS^\sigma}$ is $p$-complete by \cite[Lemma 3.16]{QS22} again.
\end{proof}

\begin{prop}
\label{comp.4}
Let $X$ be a $\Z/2$-spectrum with $S^\sigma$-action such that $i^*X$ is bounded below.
Then there are natural equivalences of $\Z/2$-spectra
\[
(X_{hS^\sigma})_p^\wedge \simeq (X_p^\wedge)_{h S^\sigma},
\;
(X^{hS^\sigma})_p^\wedge \simeq (X_p^\wedge)^{h S^\sigma},
\;
(X^{tS^\sigma})_p^\wedge \simeq (X_p^\wedge)^{t S^\sigma}.
\]
\end{prop}
\begin{proof}
Using Propositions \ref{comp.2}(2) and \ref{comp.3},
we reduce to the case when $X$ is $p$-local.
By Propositions \ref{comp.2}(5),
the multiplication $X\xrightarrow{p} X$ is an equivalence.
Hence the multiplications $p$ for $X_{hS^\sigma}$, $X^{hS^\sigma}$, and $X^{tS^\sigma}$ are equivalences.
By Propositions \ref{comp.2}(5) again,
we see that $X_{hS^\sigma}$, $X^{hS^\sigma}$, and $X^{tS^\sigma}$ are $p$-local.
Propositions \ref{comp.2}(3) finishes the proof.
\end{proof}

\begin{df}
\label{comp.5}
For a normed $\Z/2$-spectrum $A$, let $\THR(A;\Z_p)$, $\TCR^-(A;\Z_p)$, and $\TPR(A;\Z_p)$ are the $p$-completions of $\THR(A)$, $\TCR^-(A)$, and $\TPR(A)$.
\end{df}

\begin{prop}
\label{comp.6}
Let $A$ be a normed $\Z/2$-spectrum.
Then there are equivalences of $\Z/2$-spectra
\begin{gather*}
\TCR^-(A;\Z_p)
\simeq
\THR(A;\Z_p)^{hS^\sigma},
\\
\TPR(A;\Z_p)
\simeq
\THR(A;\Z_p)^{tS^\sigma}.
\end{gather*}
\end{prop}
\begin{proof}
This is an immediate consequence of Proposition \ref{comp.4}.
\end{proof}

The real $p$-cyclotomic Frobenius $\varphi_p$ induces
\(
\varphi_p\colon \THR(A)^{hS^\sigma} \to (\THR(A)^{tC_p^\sigma})^{hS^\sigma}.
\)
Apply $p$-completions on both sides and use \cite[Proposition 4.4]{QS22} to obtain
\[
\varphi_p\colon \TCR^-(A;\Z_p)
\to
\TPR(A;\Z_p).
\]
On the other hand,
we have the canonical map
\[
\can \colon \TCR^-(A;\Z_p)
\to
\TPR(A;\Z_p)
\]
induced by the canonical natural transformation $(-)^{h S^\sigma}\to (-)^{tS^\sigma}$.

\begin{df}
\label{comp.8}
For a normed $\Z/2$-spectrum $A$,
we define
\[
\TCR(A;\Z_p)
:=
\fib(\TCR^-(A;\Z_p)\xrightarrow{\varphi_p-\can}\TPR(A;\Z_p)).
\]
For the non $p$-completed $\TCR$, we refer to \cite[Remark 4.3]{QS22}.

If $R$ is a commutative ring, then we set $\TCR(R;\Z_p):=\TCR(\rH \ul{R};\Z_p)$.
\end{df}

\section{strongly even \texorpdfstring{$\Z/2$}{Z/2}-spectra}

Let us review the regular slice filtration in \cite{Ull13},
which is a slight modification of the slice filtration in \cite{HHR}.
See also \cite[\S 11.1]{HHR21}.
For every integer $n$,
let $\Sp_{\geq n}^{\Z/2}$ be the smallest full subcategory of $\Sp^{\Z/2}$ closed under colimits and extensions and containing
\[
\cS_n
:=
\{
\Sigma^{m+n\sigma} \Sphere:2m\geq n
\}
\cup
\{\Sigma^m \Sigma^\infty (\Z/2)_+ : m\geq n\}.
\]
Let $\Sp_{\leq n}^{\Z/2}$ be the full subcategory of $\Sp^{\Z/2}$ spanned by those objects $X$ such that $\Hom_{\Sp^{\Z/2}}(Y,X)=0$ for $Y\in \cS_n$.
The inclusion functor $\Sp_{\geq n}^{\Z/2}\to \Sp^{\Z/2}$ preserves colimits, and hence it admits a right adjoint.
Let $P_n\to \id\colon \Sp^{\Z/2}\to \Sp^{\Z/2}$ denote the counit of the adjunction pair.
We have the induced natural transformations
\[
\cdots \to P_n\to P_{n-1}\to \cdots,
\]
which we call the \emph{regular slice filtration}.
We set $P^n:=\cofib(P_{n+1}\to \id)$ and $P_n^n:=\cofib(P_{n+1}\to P_n)$.
There is a natural isomorphism $P_n^n\simeq \fib(P^n\to P^{n-1})$.

For $X\in \Sp^{\Z/2}$,
we say that $X$ is $n$-connected (resp.\ $n$-coconnected) if $\ul{\pi}_i(X)=0$ for $i\leq n$ (resp.\ $i\geq n$).

Next, we collect several properties of the regular slice filtration.

\begin{prop}
\label{slice.1}
For $X\in \Sp^{\Z/2}$,
we have the following properties:
\begin{enumerate}
\item[\textup{(1)}]
$X\in \Sp_{\geq 0}^{\Z/2}$ if and only if $X$ is $(-1)$-connected.
\item[\textup{(2)}]
$X\in \Sp_{\geq 1}^{\Z/2}$ if and only if $X$ is $0$-connected.
\item[\textup{(3)}]
$X\in \Sp_{\leq -1}^{\Z/2}$ if and only if $X$ is $0$-coconnected.
\item[\textup{(4)}]
$X\in \Sp_{\leq 0}^{\Z/2}$ if and only if $X$ is $1$-coconnected.
\end{enumerate}
\end{prop}
\begin{proof}
We refer to \cite[Proposition 11.1.18]{HHR21}.
\end{proof}

\begin{prop}
\label{slice.2}
Let $n$ be an integer.
Then we have equivalences $\Sp_{\geq n+2}^{\Z/2}\simeq \Sigma^{1+\sigma}\Sp_{\geq n}^{\Z/2}$ and $\Sp_{\leq n+2}^{\Z/2}\simeq \Sigma^{1+\sigma}\Sp_{\leq n}^{\Z/2}$.
\end{prop}
\begin{proof}
This is a consequence of $\cS_{n+2}=\Sigma^{1+\sigma}\cS_n$.
\end{proof}

For an integer $n$,
let $\tau_{\geq n}$ and $\tau_{\leq n}$ be the truncation functors for the homotopy $t$-structures on $\Sp$ and $\Sp^{\Z/2}$,
see \cite[\S 6]{BGS20} (also \cite[Definition 3.6]{QS22}) for the homotopy $t$-structure on $\Sp^{\Z/2}$.
For a full subcategory $\cC$ of $\Sp^{\Z/2}$ and integers $m$ and $n$
let $\Sigma^{m+n\sigma}\cC$ denote the essential image of the composite functor $\cC\hookrightarrow \Sp^{\Z/2}\xrightarrow{\Sigma^{m+n\sigma}}\Sp^{\Z/2}$.
Do not confuse $\Sp_{\geq n}^{\Z/2}$ with $\Sigma^n \Sp_{\geq 0}^{\Z/2}$.

Recall from \cite[Definition 3.11]{QS22} that $\ul{\Sp}_{\geq n}^{\Z/2}$ denotes the full $\Z/2$-subcategory of $\ul{\Sp}^{\Z/2}$ such that the fiber at $\pt\in \cO_{\Z/2}^{op}$ is $\Sp_{\geq n}^{\Z/2}$ and the fiber at $\Z/2\in \cO_{\Z/2}^{op}$ is $\Sp_{\geq n}$.
Likewise,
$\ul{\Sp}_{\leq n}^{\Z/2}$ denotes the full $\Z/2$-subcategory of $\ul{\Sp}^{\Z/2}$ such that the fiber at $\pt\in \cO_{\Z/2}^{op}$ is $\Sp_{\leq n}^{\Z/2}$ and the fiber at $\Z/2\in \cO_{\Z/2}^{op}$ is $\Sp_{\leq n}$.

\begin{prop}
\label{slice.3}
For every integer $n$,
we have natural equivalences
\begin{gather*}
P_{2n}\simeq \Sigma^{n+n\sigma}\tau_{\geq 0}\Sigma^{-n-n\sigma},
\\
P_{2n+1}\simeq
\Sigma^{n+n\sigma}\tau_{\geq 1}\Sigma^{-n-n\sigma},
\\
P^{2n-1}
\simeq
\Sigma^{n+n\sigma}\tau_{\leq -1}\Sigma^{-n-n\sigma}.
\\
P^{2n}
\simeq
\Sigma^{n+n\sigma}\tau_{\leq 0}\Sigma^{-n-n\sigma},
\end{gather*}
\end{prop}
\begin{proof}
This is an immediate consequence of Propositions \ref{slice.1} and \ref{slice.2}.
\end{proof}

For a $\Z/2$-spectrum $X$ and integers $m$ and $n$,
we set
\begin{gather*}
\ul{\pi}_{m+n\sigma}(X)
:=
\ul{\pi}_{m}(\Sigma^{-n\sigma} X),
\\
\pi_{m+n\sigma}^{\Z/2}(X)
:=
\pi_{m}^{ \Z/2}(\Sigma^{-n\sigma} X).
\end{gather*}

\begin{prop}
\label{slice.4}
For a $\Z/2$-spectrum $X$ and an integer $n$,
we have natural isomorphisms
\begin{gather*}
P_{2n}^{2n} X
\simeq
\Sigma^{n+n\sigma} \ul{\pi}_{n+n\sigma} X,
\\
P_{2n+1}^{2n+1} X
\simeq
\Sigma^{n+1+n\sigma} \cP^0 \ul{\pi}_{n+1+n\sigma} X,
\end{gather*}
where $\cP^0$ is the endofunctor of the category of Mackey functors killing the kernel of the restriction.
\end{prop}
\begin{proof}
We refer to \cite[Theorem 17.5.25]{Hil20}.
\end{proof}

For a $\Z/2$-spectrum $X$,
We use the convenient notation
\begin{equation}
\label{slice.4.1}
\rho_n(X)
:=
\left\{
\begin{array}{ll}
\ul{\pi}_{m+m\sigma} X &\text{if $n=2m$,}
\\
\cP^0\ul{\pi}_{m+1+m\sigma} X & \text{if $n=2m+1$}
\end{array}
\right.
\end{equation}
so that we have $P_{2m}^{2m}  X \simeq \Sigma^{m+m\sigma} \rho_{2m}(X)$ and $P_{2m+1}^{2m+1}X\simeq \Sigma^{m+1+m\sigma}\rho_{2m+1}(X)$.

For an integer $n$,
let $\Sp_{\geq n}$ (resp.\ $\Sp_{\leq n}$) be the full subcategory of $\Sp$ spanned by $(n-1)$-connected (resp.\ $(n+1)$-coconnected) spectra.
Unlike the case of $\Sp^{\Z/2}$,
we have $\Sp_{\geq n}= \Sigma^n \Sp_{\geq 0}$ and $\Sp_{\leq n}= \Sigma^n \Sp_{\leq 0}$.

\begin{prop}
\label{slice.5}
For every integer $n$,
we have natural equivalences $i^*P_n\simeq \tau_{\geq n}i^*$ and $i^*P^n\simeq \tau_{\leq n}i^*$,
and we have inclusions $i^*\Sp_{\geq n}^{\Z/2}\subset \Sp_{\geq n}$ and $i^*\Sp_{\leq n}^{\Z/2}\subset \Sp_{\leq n}$.
\end{prop}
\begin{proof}
We have natural equivalences $i^*\tau_{\geq n}\simeq \tau_{\geq n} i^*$ and $i^*\Sigma^{n\sigma}\simeq \Sigma^n i^*$.
We finish the proof using Proposition \ref{slice.3}.
\end{proof}

For an abelian group $M$,
let $\ul{M}$ denote the constant Mackey functor associated with $M$.
The following definition is also used in \cite{HP}.

\begin{df}\label{def:even}
A $\Z/2$-spectrum $X$ is \emph{even} if $\rho_{2n+1}(X)=0$ for every integer $n$.
Note that this definition is different from the definition due to Hill--Meier \cite[Definition 3.1]{HM17},
where the evenness was defined as the vanishing of odd Hill--Hopkins--Ravenel slices instead of odd regular slices,
which is equivalent to $\ul{\pi}_{n-1+n\sigma} X=0$ for every integer $n$.

For a Mackey functor $M$,
by \cite[Proposition 3.10]{HP},
$X$ is even if and only if $i^*X\in \Sp$ is even in the sense that $\pi_{2n+1}(i^*X)=0$ for every integer $n$.

Hill and Meier \cite[Definition 3.1]{HM17} also defined strongly even $\Z/2$-spectra.
Here is an equivalent definition:
A $\Z/2$-spectrum $X$ is \emph{strongly even}\footnote{This is called \emph{very even} in \cite{HP}.} if it is even and $\rho_{2n}(X)$ is a constant Mackey functor for every integer $n$.
In this case,
we have a natural isomorphism
\begin{equation}
\label{eq:very even}
\rho_{2n}(X)\simeq \ul{\pi_{2n}(i^* X)}.
\end{equation}
\end{df}

\begin{lem}
\label{THR.2}
Let $M$ be an abelian group.
Then we have $\pi_\sigma^{\Z/2}(\rH \ul{M})= 0$ and $\pi_{-\sigma}^{\Z/2}(\rH \ul{M})\simeq \coker(M\xrightarrow{2} M)$.
\end{lem}
\begin{proof}
Consider the abelian group $M\oplus M$ with the involution switching the two coordinates,
and consider the associated Mackey functor $\ul{M\oplus M}$.
We have the cofiber sequence of $\Z/2$-spectra $\Sigma^\infty (\Z/2)_+\to \Sphere \to \Sigma^\sigma \Sphere$,
which induces the cofiber sequence
\[
\rH \ul{M\oplus M}\to \rH \ul{M}\xrightarrow{\rH \ul{f}} \Sigma^\sigma \rH{\ul{M}}
\]
such that $f\colon M\oplus M\to M$ is the summation map.
We have the associated long exact sequence
\[
\cdots \to \pi_0^{\Z/2}(\rH \ul{M\oplus M})
\to
\pi_0^{\Z/2}(\rH \ul{M})
\to
\pi_{-\sigma}^{\Z/2} (\rH \ul{M})
\to
\pi_1^{\Z/2}(\rH \ul{M\oplus M})
\to
\cdots.
\]
From this, we obtain $\pi_{-\sigma}^{\Z/2} (\rH \ul{M}) \simeq \coker(M\xrightarrow{2} M)$.

We have the fiber sequence of $\Z/2$-spectra $\Sigma^{-\sigma}\Sphere \to \Sphere \to \Sigma^\infty (\Z/2)_+$,
which induces the fiber sequence
\[
\Sigma^{-\sigma} \rH \ul{M}
\to
\rH \ul{M}
\xrightarrow{\rH \ul{g}}
\rH \ul{M\oplus M}
\]
such that $g\colon M\to M\oplus M$ is the diagonal map.
We have the associated long exact sequence
\[
\cdots
\to
\pi_1^{\Z/2}(\rH \ul{M\oplus M})
\to
\pi_\sigma^{\Z/2} (\rH \ul{M})
\to
\pi_0^{\Z/2}(\rH \ul{M})
\to
\pi_0^{\Z/2}(\rH \ul{M\oplus M}).
\]
From this,
we obtain $\pi_\sigma^{ \Z/2}(\rH \ul{M})=0$.
\end{proof}

\begin{lem}
\label{THR.3}
Let $X$ be a strongly even $\Z/2$-spectrum.
Then for every integer $n$,
we have $\ul{\pi}_{n-1+n\sigma}(X)=0$ and
\(
\pi_{n+1+n\sigma}^{\Z/2}(X)\simeq \coker(\pi_{2n+2}(i^*X)\xrightarrow{2} \pi_{2n+2}(i^*X)).
\)
\end{lem}
\begin{proof}
We have $\Sigma^{-n+1-n\sigma}P_{2n}X\in \Sigma^1 \Sp_{\geq 0}^{\Z/2}$ and $\Sigma^{-n+1-n\sigma}P^{2n-4}X\in \Sigma^{-1}\Sp_{\leq 0}^{\Z/2}$ using Proposition \ref{slice.2}.
Since $X$ is even,
we have $P_{2n-1}^{2n-1}X\simeq P_{2n-3}^{2n-3}X\simeq 0$ in $\Sp^{\Z/2}$.
Hence $\ul{\pi}_{n-1+n\sigma}(P_{2n-1}X)$ and $\ul{\pi}_{n-1+n\sigma}(P^{2n-3}X)$ vanish,
so we have isomorphisms
\[
\ul{\pi}_{n-1+n\sigma}(X)
\simeq
\ul{\pi}_{n-1+n\sigma}(P_{2n-2}^{2n-2}X)
\simeq
\ul{\pi}_{0}(\Sigma^{-\sigma}\rho_{2n-2}(X)).
\]
Since $\rho_{2n-2}(X)$ is a constant Mackey functor,
we have $\ul{\pi}_{0}(\Sigma^{-\sigma}\rho_{2n-2}(X))\simeq 0$ by Lemma \ref{THR.2}.

We have $\Sigma^{-n-1-n\sigma}P_{2n+4}X\in \Sigma^1 \Sp_{\geq 0}^{\Z/2}$ and $\Sigma^{-n-1-n\sigma}P^{2n}X\in \Sigma^{-1}\Sp_{\leq 0}^{\Z/2}$
using Proposition \ref{slice.2}.
Since $X$ is even,
we have $P_{2n+1}^{2n+1}X\simeq P_{2n+3}^{2n+3}X\simeq 0$ in $\Sp^{\Z/2}$.
Hence we have isomorphisms
\[
\pi_{n+1+n\sigma}^{\Z/2}(X)
\simeq
\pi_{n+1+n\sigma}^{\Z/2}(P_{2n+2}^{2n+2}X)
\simeq
\pi_{0}^{\Z/2}(\Sigma^\sigma\rho_{2n+2}(X)).
\]
Since $\rho_{2n+2}(X)$ is isomorphic to the constant Mackey functor $\ul{\pi_{2n+2}(i^*X)}$,
Lemma \ref{THR.2} finishes the proof.
\end{proof}

A morphism $M\to M'$ of Mackey functors is a monomorphism (resp.\ epimorphism) if $M(\Z/2)\to M'(\Z/2)$ and $M(\pt)\to M'(\pt)$ are monomorphisms (resp.\ epimorphisms).

\begin{lem}
\label{THR.4}
Let $X\to Y\to Z$ be a fiber sequence in $\Sp^{\Z/2}$.
If $X$ and $Z$ are strongly even,
then $Y$ is strongly even,
and $\rho_{2n}(Y)\to \rho_{2n}(Z)$ is an epimorphism of Mackey functors for every integer $n$.
\end{lem}
\begin{proof}
Let $n$ be an integer.
We have $\pi_{2n+1}(i^*X)=\pi_{2n+1}(i^*Z)=0$ since $X$ and $Z$ are even,
so we have $\pi_{2n+1}(i^*Y)=0$ using the induced exact sequence $\pi_{2n+1}(i^*X)\to \pi_{2n+1}(i^*Y)\to \pi_{2n+1}(i^*Z)$.
This implies that $Y$ is even.

We have the induced commutative diagram with exact rows
\[
\begin{tikzcd}[column sep=tiny]
\pi_{n+1+n\sigma}^{\Z/2}(Z)\ar[r]\ar[d]&
\pi_{n+n\sigma}^{\Z/2}(X)\ar[r]\ar[d,"a"]&
\pi_{n+n\sigma}^{\Z/2}(Y)\ar[r]\ar[d,"b"]&
\pi_{n+n\sigma}^{\Z/2}(Z)\ar[r]\ar[d,"c"]&
\pi_{n-1+n\sigma}^{\Z/2}(X)\ar[d]
\\
\pi_{2n+1}(i^*Z)\ar[r]&
\pi_{2n}(i^*X)\ar[r]&
\pi_{2n}(i^*Y)\ar[r]&
\pi_{2n}(i^*Z)\ar[r]&
\pi_{2n-1}(i^*X),
\end{tikzcd}
\]
where the vertical arrows are the restriction maps of the Mackey functors $\ul{\pi}_{n+i+n\sigma}(-)$ with $i=-1,0,1$.
Since $X$ and $Z$ are strongly even,
$a$ and $c$ are isomorphisms,
and we have $\pi_{2n+1}(i^*Z)=0$.
By Lemma \ref{THR.3},
we also have $\pi_{n-1+n\sigma}^{\Z/2}(X)=0$.
Therefore by the snake lemma, $b$ is an isomorphism as well, and thus $\rho_{2n}(Y)$ is a constant Mackey functor.
Since $X$ is even,
we have $\pi_{2n-1}(i^*X)=0$.
Hence $\rho_{2n}(Y)\to \rho_{2n}(Z)$ is an epimorphism.
\end{proof}

We keep using the convention of a filtration in \cite[Definition 4.18]{HP}:
A \emph{filtered $\Z/2$-spectrum} is an infinite sequence
\[
\cdots \to \Fil_n X\to \Fil_{n-1} X\to \cdots
\]
of $\Z/2$-spectra.
It is \emph{complete} if $\lim_n \Fil_n X\simeq 0$,
and its \emph{underlying $\Z/2$-spectrum} is $\colim_n \Fil_n X$.
For every integer $n$,
the \emph{$n$th graded piece} is
\[
\gr^nX:=\cofib(\Fil_{n+1} X\to \Fil_{n} X).
\]
The \emph{$\infty$-category of filtered $\Z/2$-spectrum} is $\Fun(\Z^{op},\Sp^{\Z/2})$,
where $\Z^{op}$ denotes the category whose objects are the integers and whose morphism set $\Hom_{\Z^{op}}(m,n)$ has exactly one element (resp.\ no elements) if $m\geq n$ (resp.\ $m<n$).

\begin{lem}
\label{THR.5}
Let $\cdots \to \Fil_1X\to \Fil_0X=X$ be a complete filtration on $X\in \Sp^{\Z/2}$.
If the graded piece $\gr^n X$ is strongly even for every integer $n$,
then $X$ is strongly even.
\end{lem}
\begin{proof}
We set $\Fil^nX:=\cofib(\Fil_nX\to X)$.
Apply Lemma \ref{THR.4} to the fiber sequence $\gr^{n+1}X\to \Fil^{n+1}X\to \Fil^nX$,
and use induction to see that $\Fil^{n}X$ is strongly even and $\rho_{2i}\Fil^{n+1} X\to \rho_{2i}\Fil^{n} X$ is an epimorphism for all integers $i$ and $n$.

The filtration $\Fil_\bullet X$ is complete,
so we have an equivalence of $\Z/2$-spectra $\lim_n \Fil^n X\simeq X$.
Since $\rho_{2i+2}(\Fil^{n+1}X)\to \rho_{2i+2}(\Fil^n X)$ is an epimorphism,
the induced map
\begin{align*}
&\coker(\pi_{2i+2}(i^* \Fil^{n+1} X)\xrightarrow{2} \pi_{2i+2}(i^* \Fil^{n+1} X))
\\
\to &
\coker(\pi_{2i+2}(i^* \Fil^n X)\xrightarrow{2} \pi_{2i+2}(i^* \Fil^n X))
\end{align*}
is an epimorphism too.
Together with Lemma \ref{THR.3},
we see that $\pi_{i+1+i\sigma}^{\Z/2}(\Fil^{n+1}X)\to \pi_{i+1+i\sigma}^{\Z/2}(\Fil^n X)$ is an epimorphism.
On the other hand,
we have
$\pi_{2i+1}(i^*\Fil^n X)=0$ since $\Fil^n X$ is even,
so $\pi_{2i+1}(i^*\Fil^{n+1}X)\to \pi_{2i+1}(i^*\Fil^n X)$ is an epimorphism.
It follows that $\ul{\pi}_{i+1+i\sigma}(\Fil^{n+1}X)\to \ul{\pi}_{i+1+i\sigma}(\Fil^n X)$ is an epimorphism.
In particular, the tower $\ul{\pi}_{i+1+i\sigma}(\Fil^\bullet X)$ satisfies the Mittag-Leffler condition, so we have
\[
\mathop{\mathrm{lim}^1}_n \ul{\pi}_{i+1+i\sigma} (\Fil^n X)
\simeq
0.
\]
Together with the Milnor exact sequence
\[
0
\to
\mathop{\mathrm{lim}^1}_n \ul{\pi}_{i+1+i\sigma} (\Fil^n X)
\to
\ul{\pi}_{i+i\sigma} (X)
\to
\lim_n \ul{\pi}_{i+i\sigma} (\Fil^n X)
\to
0,
\]
we have an isomorphism $\ul{\pi}_{i+i\sigma} (X)\simeq \lim_n \ul{\pi}_{i+i\sigma} (\Fil^n X)$.
This implies that $\ul{\pi}_{i+i\sigma} (X)$ is a constant Mackey functor.

Since $\rho_{2i+2}(\Fil^{n+1}X)\to \rho_{2i+2}(\Fil^n X)$ is an epimorphism,
$\pi_{2i+2}(i^*\Fil^{n+1}X)\to \pi_{2i+2}(i^*\Fil^n X)$ is an epimorphism.
Together with the Milnor exact sequence
\[
0
\to
\mathop{\mathrm{lim}^1}_n \pi_{2i+2} (i^*\Fil^n X)
\to
\pi_{2i+1} (i^*X)
\to
\lim_n \pi_{2i+1} (i^*\Fil^n X)
\to
0,
\]
we have an isomorphism $\pi_{2i+1} (i^*X)\simeq \lim_n \pi_{2i+1} (i^*\Fil^n X)$.
Hence $\pi_{2i+1}(i^*X)$ vanishes,
so $X$ is strongly even.
\end{proof}

\section{Completeness of the real Hochschild--Kostant--Rosenberg filtration}

The purpose of this section is to show that the natural filtration on real Hochschild homology in \cite[Theorem 4.32]{HP} is complete under a certain condition, see Theorem \ref{geometric.6}.

\begin{df}
Let $R\to A$ be a map normed $\Z/2$-spectra.
The \emph{Hochschild homology of $A$ over $R$} is
\[
\HR(A/R)
:=
\THR(A)\wedge_{\THR(R)}R
\in
(\NAlg^{\Z/2})^{B S^\sigma}.
\]
This imposes a natural $S^\sigma$-action on the Hochschild homology in \cite[Definition 4.1]{HP}.
After forgetting the normed structure on $\HR(A/R)$,
we set
\[
\HR(A/R;\Z_p):=\HR(A/R)_p^\wedge
\in
(\Sp^{\Z/2})^{BS^\sigma}.
\]
\end{df}

\begin{df}
If $R\to A$ is a map of commutative rings,
then we set
\[
\HR(A/R)
:=
\HR(\rH \ul{A}/\rH \ul{R}),
\;
\HR(A/R;\Z_p)
:=
\HR(\rH \ul{A}/\rH \ul{R};\Z_p).
\]
\end{df}

\begin{prop}
\label{geometric.1}
Let $A$ be a commutative ring.
Then we have a natural isomorphism
\[
\pi_i \Phi^{\Z/2}\rH \ul{A}
\cong
\left\{
\begin{array}{ll}
A/2 & \text{if $i$ is even and $i\geq 0$},\\
\ker(A\xrightarrow{2} A) & \text{if $i$ is odd and $i\geq 0$},
\\
0 &\text{if $i<0$}.
\end{array}
\right.
\]
\end{prop}
\begin{proof}
We have the isotropy separation cofiber sequence
\[
(\rH \ul{A})_{h\Z/2}
\to
(\rH \ul{A})^{\Z/2}
\to
\Phi^{\Z/2} \rH \ul{A}.
\]
The induced map $\pi_0 (\rH \ul{A})_{h\Z/2}\to \pi_0 (\rH \ul{A})^{\Z/2}$ is the multiplication by $2$ since this is equal to the transfer map for the constant Mackey functor $\ul{A}$.
It follows that we have $\pi_0\Phi^{\Z/2}\rH \ul{A}\cong A/2$ and $\pi_i\Phi^{\Z/2}\rH \ul{A}\cong \pi_{i-1}(\rH \ul{A})_{h\Z/2}$ for $i>0$.
We have natural isomorphisms
\[
\pi_i (\rH \ul{A})_{h\Z/2}
\cong
H_i (B (\Z/2),A)
\cong
H_i(\R \P^\infty, A)
\]
for $i\geq 0$.
Hence it suffices to compute the last one, which can be done using the universal coefficient theorem and the well-known computation of $H_i(\R\P^\infty,\Z)$.
\end{proof}

\begin{prop}
\label{geometric.2}
Let $R\to A$ be a map of normed $\Z/2$-spectra.
Then there is a natural equivalence of $\E_\infty$-rings
\[
\Phi^{\Z/2}\HR(A/R)
\simeq
\Phi^{\Z/2}A\wedge_{i^*A\wedge_{i^*R} \Phi^{\Z/2}}\Phi^{\Z/2}A.
\]
\end{prop}
\begin{proof}
We have
\begin{align*}
\HR(A/R)
= &
\THR(A)\wedge_{\THR(R)} R
\\
= &
(A\wedge_{N_e^{\Z/2}i^*A} A)\wedge_{R\wedge_{N_e^{\Z/2}i^*A}R} R
\\
\simeq &
A\wedge_{N_e^{\Z/2}i^*R\wedge_{N_e^{\Z/2}i^*R}R} A.
\end{align*}
Apply $\Phi^{\Z/2}$ to this,
and use $\Phi^{\Z/2}N_e^{\Z/2}\simeq \id$ and the fact that $\Phi^{\Z/2}$ is symmetric monoidal (see e.g.\ \cite[Remark 9.10]{BH21}) to conclude.
\end{proof}

\begin{prop}
\label{geometric.3}
Let $A$ be a commutative ring.
Then $\pi_0$ of the induced composite map
\[
\rH A
\xrightarrow{\simeq}
i^*\rH \ul{A}
\xrightarrow{\simeq}
\Phi^{\Z/2}
N_e^{\Z/2}i^*
\rH \ul{A}
\to
\Phi^{\Z/2}
\rH \ul{A}
\]
is the map $A\to A/2$ given by $a\mapsto a^2$ for $a\in A$.
\end{prop}
\begin{proof}
Let $(A\otimes A)_{\Z/2}$ be the abelian group of coinvariants with the action $w$ on $A\otimes A$ given by $w(a\otimes b)=b\otimes a$ for $a,b\in A$.
Consider the abelian group $W_2^\otimes(A)$ in \cite[Definition 5.3]{DMPR21},
whose underlying set is $A\otimes (A\otimes A)_{\Z/2}$ with the group operation given by
\[
(a,x)+(b,y):=(a+b,x+y-a\otimes b)
\]
for $a,b\in A$ and $x,y\in (A\otimes A)_{\Z/2}$.
Using \cite[Proposition 5.5]{DMPR21},
we can describe $\ul{\pi}_0$ of the induced map $N_e^{\Z/2}i^*\rH \ul{A}\to \rH \ul{A}$ as the map of Mackey functors
\[
\begin{tikzcd}
W_2^\otimes(A)\ar[r,"g"]
\ar[d,"\res"',shift right=0.5ex]
\ar[d,"\tr",shift left=0.5ex,leftarrow]
&
A\ar[d,"\id"',shift right=0.5ex]
\ar[d,"2",shift left=0.5ex,leftarrow]
\\
A\otimes A\ar[loop left,"w"]\ar[r,"f"]&
A,\ar[loop right,"\id"]
\end{tikzcd}
\]
where $f(a\otimes b):=ab$, $\res(a,x):=a\otimes a+x+w(x)$, and $\tr(a\otimes b):=(0,a\otimes b)$ for $a,b\in A$ and $x\in (A\otimes A)_{\Z/2}$.
Since $f\circ \res=g$,
we have $g(a,0)=a^2$ for $a\in A$.
The induced map $\pi_0^{\Z/2}(\rH A)\to \pi_0(\Phi^{\Z/2}\rH \ul{A})$ is
\[
\coker(A\otimes A\xrightarrow{\tr} W_2^{\otimes}(A))
\to
\coker(A\xrightarrow{2} A).
\]
The above descriptions of $g(a,0)$ and $\res$ finish the proof.
\end{proof}

For a commutative $\F_2$-algebra $R$,
let $\varphi\colon R\to R$ be the Frobenius map.
For
an $R$-algebra $A$,
we have the induced commutative diagram
\begin{equation}
\label{geometric.3.1}
\begin{tikzcd}
R\ar[r,"\varphi"]\ar[d]&
R\ar[d]
\\
A\ar[r]\ar[rr,bend right,"\varphi"]&
A^{(1)}\ar[r]&
A,
\end{tikzcd}
\end{equation}
where the left vertical map is the structural map, and
$A^{(1)}:=R\otimes_{\varphi,R}A$.
Recall that the map $A^{(1)}\to A$ is the \emph{relative Frobenius of $R\to A$}.

Recall from \cite[Definition 7.2.2.10]{HA} that for a ring spectra $A$ and an $A$-module $M$,
$M$ is a \emph{flat $A$-module} if $\pi_0(M)$ is a flat $\pi_0(A)$-module and the induced map $\pi_n(R)\otimes_{\pi_0(R)}\pi_0(A)\to \pi_n(A)$ is an isomorphism for every integer $n$.

\begin{prop}
\label{geometric.7}
Let $R$ be a commutative ring,
and let $A$ be a flat $R$-algebra.
Then the induced map
\[
\Phi^{\Z/2}\rH \ul{R}\to \Phi^{\Z/2}\rH \ul{A}
\]
is flat.
If the relative Frobenius $B^{(1)}\to B$ of $R/2\to B:=A/2$ is flat too,
then the induced maps
\[
\rH A\wedge_{\rH R} \Phi^{\Z/2}\rH \ul{R}\to \Phi^{\Z/2}\rH \ul{A},
\;
\Phi^{\Z/2}\rH \ul{R}\to  \Phi^{\Z/2}\HR(A/R)
\]
are flat.
\end{prop}
\begin{proof}
By Propositions \ref{geometric.1} and \ref{geometric.3},
we have
\begin{equation}
\label{geometric.7.1}
\pi_0 \Phi^{\Z/2}\rH \ul{R}\cong R/2,
\;
\pi_0 (\rH A\wedge_{\rH R}\Phi^{\Z/2}\rH \ul{R})\cong B^{(1)},
\;
\pi_0 \Phi^{\Z/2}\rH \ul{A}\cong B.
\end{equation}
Since $R\to A$ is flat,
we have an isomorphism
\[
A\otimes_R \ker(R\xrightarrow{2} R)
\cong
\ker(A\xrightarrow{2} A).
\]
Together with Proposition \ref{geometric.1},
we see that $\Phi^{\Z/2}\rH \ul{R}\to \Phi^{\Z/2}\rH \ul{A}$ is flat.
This implies that the induced map
\[
\pi_n(\rH A\wedge_{\rH R} \Phi^{\Z/2}\rH \ul{R})
\otimes_{\pi_0(\rH A\wedge_{\rH R} \Phi^{\Z/2}\rH \ul{R})}
\pi_0(\Phi^{\Z/2}\rH \ul{A})
\to
\pi_n(\Phi^{\Z/2}\rH \ul{A})
\]
is an isomorphism for every integer $n$ since the flatness of $R\to A$ implies
\[
\pi_n(\rH A\wedge_{\rH R} \Phi^{\Z/2}\rH \ul{R})
\simeq
A\otimes_R \pi_n(\Phi^{\Z/2}\rH \ul{R}).
\]
Using Proposition \ref{geometric.3} and the assumption that $B^{(1)}\to B$ is flat,
we see that
\[
\pi_0(\rH A\wedge_{\rH R} \Phi^{\Z/2}\rH \ul{R})\to \pi_0(\Phi^{\Z/2}\rH \ul{A})
\]
is flat.
Hence $\rH A\wedge_{\rH R} \Phi^{\Z/2}\rH \ul{R}\to \Phi^{\Z/2}\rH \ul{A}$ is flat.
Since the tensor product of flat algebras is flat,
Proposition \ref{geometric.2} implies that $\Phi^{\Z/2}\rH \ul{R}\to \HR(A/R)$ is flat.
\end{proof}

\begin{prop}
\label{geometric.4}
Let $R$ be a commutative ring,
and let $A$ be a flat $R$-algebra.
If the relative Frobenius $B^{(1)}\to B$ of $R/2\to B:=A/2$ is flat too,
then we have a natural isomorphism
\[
\pi_0 \Phi^{\Z/2}\HR(A/R)
\cong
B\otimes_{B^{(1)}}B.
\]
\end{prop}
\begin{proof}
By Propositions \ref{geometric.2} and \ref{geometric.7},
we have a natural isomorphism
\[
\pi_0 \Phi^{\Z/2} \HR(A/R)
\cong
\pi_0 \Phi^{\Z/2} \rH \ul{A}
\otimes_{\pi_0(\rH A\wedge_{\rH R} \Phi^{\Z/2} \rH \ul{R})}
\pi_0 \Phi^{\Z/2} \rH \ul{A}.
\]
Combine this with \eqref{geometric.7.1} to obtain the desired isomorphism.
\end{proof}

For a commutative ring $R$,
let $\Poly_R$ be the category of finitely generated polynomial $R$-algebras,
and let $\Ani(\Ring_R)$ be the $\infty$-category of animated commutative $R$-algebras.
Recall from \cite[Example 5.1.3, 5.1.4]{CS24} that $\Ani(\Ring_R)$ is the $\infty$-category freely generated under sifted colimits by the category of commutative $R$-algebras that are retracts of finitely generated polynomial $R$-algebras.
Recall also from \cite[Construction 2.1]{BMS19} and \cite[Proposition 5.5.8.15]{HTT} that for every $\infty$-category $\cD$ admitting sifted colimits,
there exists a natural equivalence of $\infty$-categories
\begin{equation}
\label{geometric.5.2}
\Fun_\Sigma(\Ani(\Ring_R),\cD)
\xrightarrow{\simeq}
\Fun(\Poly_R,\cD),
\end{equation}
where the left-hand side denotes the full subcategory spanned by those functors $\Ani(\Ring_R)\to \cD$ preserving sifted colimits.
If a functor $f\colon \Poly_R\to \cD$ corresponds to a sifted colimit preserving functor $F\colon \Ani(\Ring_R)\to \cD$ via \eqref{geometric.5.2},
then we say that $F$ is a \emph{left Kan extension of $f$}.
Observe that $F$ is uniquely determined by $f$ in the $\infty$-categorical sense.

For a commutative ring $R$ and $A\in \Poly_R$,
let us recall the construction of the natural filtration on $\HR(A/R)$ in \cite[Theorem 4.32]{HP}.
If $A=R[x]$,
then let $\Fil_1\HR(A/R)$ be the fiber of the canonical map $\HR(A/R)\to A$.

If $A=R[x_1,\ldots,x_n]$,
then we have an equivalence of $\Z/2$-spectra
\[
\HR(A/R)
\simeq
\HR(R[x_1]/R)\wedge_{\rH \ul{R}} 
\cdots \wedge_{\rH \ul{R}}  \HR(R[x_n]/R),
\]
or we can write this as an equivalence
\begin{equation}
\label{geometric.5.3}
\HR(A/R)
\simeq
\HR(R[x_1]/R)\sotimes_{\ul{R}}^\L
\cdots
\sotimes_{\ul{R}}^\L \HR(R[x_n]/R)
\end{equation}
in the derived $\infty$-category $\rD(\ul{R})$,
where $\sotimes_{\ul{R}}^\L$ denotes the monoidal product of $\rD(\ul{R})$ \cite[Recollection 2.2]{HP}.
By taking $\sotimes_{\ul{R}}^\L$ in the right-hand side as the monoidal product of the derived $\infty$-category $\DF(\ul{R})$,
we obtain the filtration on $\HR(A/R)$.
By \cite{GP18} (see \cite[Lemma 5.2(5)]{BMS19} for a non-equivariant one),
we have
\begin{equation}
\label{geometric.5.1}
\gr^n\HR(A/R)
\simeq
\bigoplus_{q_1+\cdots+q_d=n}
\gr^{q_1}\HR(R[x_1]/R)\sotimes_{\ul{R}}^\L\cdots \sotimes_{\ul{R}}^\L\gr^{q_d}\HR(R[x_d]/R).
\end{equation}

\begin{prop}
\label{geometric.5}
Let $R$ be a commutative ring.
For $A\in \Poly_R$ and integer $n$,
$\Phi^{\Z/2} \Fil_n \HR(A/R)$ is a flat $\Phi^{\Z/2}\rH \ul{R}$-module.
Furthermore,
we have a natural isomorphism
\[
\pi_i\Phi^{\Z/2}\Fil_n \HR(A/R)
\cong
\left\{
\begin{array}{ll}
J^n & \text{if $i$ is even and $i\geq 0$},\\
\ker(R\xrightarrow{2} R)\otimes_R J^n & \text{if $i$ is odd and $i\geq 0$},
\\
0 &\text{if $i<0$},
\end{array}
\right.
\]
where $B:=A/2$, and $J$ is the kernel of the multiplication map $B\otimes_{B^{(1)}}B\to B$.
\end{prop}
\begin{proof}
We can deduce the claim for $\pi_i \Phi^{\Z/2}\Fil_n\HR(A/R)$ from Proposition \ref{geometric.1} and the claim for $\pi_0 \Phi^{\Z/2}\Fil_n \HR(A/R)$ if we know that $\Phi^{\Z/2}\Fil_n \HR(A/R)$ is a flat $\Phi^{\Z/2}\rH \ul{R}$.

Assume $A=R[x]$.
Then $\Phi^{\Z/2} \gr^0 \HR(A/R)\simeq \Phi^{\Z/2} \rH \ul{A}$ and $\Phi^{\Z/2} \HR(A/R)$ are flat $\Phi^{\Z/2} \rH \ul{R}$-modules by Proposition \ref{geometric.7}.
Hence $\Phi^{\Z/2} \Fil_n \HR(A/R)$ is a flat $\Phi^{\Z/2}\rH \ul{R}$-module for every integer $n$.
The map
$\HR(A/R)\to \rH \ul{A}$ can be identified with the map
\[
\rH \ul{A}\wedge_{N_e^{\Z/2}i^*\rH \ul{A}}\rH \ul{A}
\to
\rH \ul{A}\wedge_{\rH \ul{A}}\rH \ul{A}
\]
induced by the counit $N_e^{\Z/2}i^* \rH \ul{A}\to \rH \ul{A}$.
Together with Propositions \ref{geometric.3} and \ref{geometric.4},
we see that $\pi_0\Phi^{\Z/2}\HR(A/R)\to \pi_0\Phi^{\Z/2}\rH \ul{A}$ is the multiplication map $B\otimes_{B^{(1)}}B\to B$.
Hence we have $\pi_0\Phi^{\Z/2}\Fil_1\HR(A/R)\cong J$.
This finishes the proof when $A=R[x]$.

For general $A=R[x_1,\ldots,x_d]$,
we have
\[
B\otimes_{B^{(1)}}B
\cong
B[\epsilon_1,\ldots,\epsilon_d]/(\epsilon_1^2,\ldots,\epsilon_d^2),
\;
J=(\epsilon_1,\ldots,\epsilon_d)
\]
with $\epsilon_r:=x_r\otimes 1-1\otimes x_r$ for $1\leq r\leq d$.
Using \eqref{geometric.5.1} and the monoidality of $\Phi^{\Z/2}$,
we have
\begin{align*}
& \Phi^{\Z/2} \gr^n \HR(A/R)
\\
\simeq &
\bigoplus_{q_1+\cdots+q_d=n}
\Phi^{\Z/2}\gr^{q_1}\HR(R[x_1]/R)\wedge_{\Phi^{\Z/2}\rH R} \cdots \wedge_{\Phi^{\Z/2} \rH R} \Phi^{\Z/2}\gr^{q_d}\HR(R[x_d]/R).
\end{align*}
Hence $\Phi^{\Z/2} \gr^n \HR(A/R)$ is a flat $\Phi^{\Z/2}\rH \ul{R}$-module,
so $\Phi^{\Z/2}\Fil_n \HR(A/R)$ is a flat $\Phi^{\Z/2}\rH \ul{R}$-module.
We already know
\[
\pi_0\Phi^{\Z/2}\gr^{q_r} \HR(R[x_r]/R)
\cong
B_r\epsilon_r^{q_r}
\]
with $B_r:=R/2[x_r]$.
On the other hand,
we have $\pi_0\Phi^{\Z/2}\rH \ul{R}\cong R/2$ by Proposition \ref{geometric.1}.
Combine what we have discussed above to obtain an isomorphism
\[
\pi_0 \Phi^{\Z/2} \gr^n \HR(A/R)
\cong
\bigoplus_{q_1+\cdots+q_d=n}
B \epsilon_1^{q_1} \cdots \epsilon_n^{q_n},
\]
which can be identified with $J^n/J^{n+1}$.
This implies $\pi_0\Phi^{\Z/2}\Fil_n \HR(A/R)
\cong J^n$.
\end{proof}

By left Kan extension,
we can define $J^n$ and $\Fil_i \Phi^{\Z/2}\Fil_n\HR(A/R)$ for $A\in \Ani(\Ring_R)$ from $J^n$ and $\pi_i \Phi^{\Z/2}(\Fil_n \HR(A/R))$ for $A\in \Poly_R$.
We will show in Proposition \ref{conv.4} that $J^n$ is $(n-1)$-connected if $J$ is $0$-connected.
This will lead to the completeness of the filtration on $\HR(A/R)$ under a certain assumption,
see Theorem \ref{geometric.6}.

Let $R$ be a commutative ring.
For $A\in \Ani(\Ring_R)$,
recall from \cite[\S 2]{BMS19} that $\HH(A/R)$ admits the Hochschild--Kostant--Rosenberg filtration whose graded pieces are
\[
\gr^n \HH(A/R)
\simeq
\wedge_A^i \L_{A/R}[i].
\]

\begin{prop}
\label{geometric.8}
Let $R$ be a commutative ring.
For $A\in \Ani(\Ring_R)$ and integer $n$,
there exists a natural isomorphism
\[
i^*\Fil_n \HR(A/R)
\simeq
\Fil_n \HH(A/R).
\]
\end{prop}
\begin{proof}
By left Kan extension,
we reduce to the case when $A\in \Poly_R$.
To conclude,
check that $i^*\Fil_n \HR(A/R)$ is the Postnikov filtration on $\HH(A/R)$ using \cite[Lemma 4.28]{HP} for $A=R[x]$ and using \eqref{geometric.5.3} for general $A$.
\end{proof}

Our next interest is to show that the filtration on real Hochschild homology is $S^\sigma$-equivariant,
see Proposition \ref{TPR.11}.

\begin{const}
\label{TPR.3}
For all integers $m$ and $n$,
we have a natural equivalence $i^* \Sigma^{m+n\sigma}\simeq \Sigma^{m+n}$.
Use this to construct the $\Z/2$-functor
\[
\Sigma^{m+n\sigma}
\colon
\ul{\Sp}^{\Z/2}
\to
\ul{\Sp}^{\Z/2}
\]
that is $\Sigma^{m+n\sigma}$ on the fiber of $*\in \cO_{\Z/2}^{op}$ and $\Sigma^{m+n}$ on the fiber of $\Z/2\in \cO_{\Z/2}^{op}$.
\end{const}

Do not confuse $\ul{\Sp}_{\geq n}^{\Z/2}$ with $\Sigma^n \ul{\Sp}_{\geq 0}^{\Z/2}$.

\begin{const}
\label{TPR.10}
By Proposition \ref{slice.5},
we have the $\Z/2$-functor $P_n\colon \ul{\Sp}^{\Z/2}\to \ul{\Sp}^{\Z/2}$ for every integer $n$ such that this is $P_n\colon \Sp^{\Z/2}\to \Sp^{\Z/2}$ on the fiber at $\Z/2\in \cO_{\Z/2}^{op}$ and $\tau_{\geq n}\colon \Sp\to \Sp$ on the fiber at $*\in \cO_{\Z/2}^{op}$.
We also have natural $\Z/2$-transformations
\(
\cdots \to P_n\to P_{n-1}\to \cdots
\)
of $\Z/2$-endofunctors on $\ul{\Sp}^{\Z/2}$.

Applying $\Fun_{\Z/2}(BS^\sigma,-)$ to this,
we obtain the induced natural transformations
\(
\cdots \to P_n\to P_{n-1}\to \cdots
\)
of endofunctors on $(\Sp^{\Z/2})^{BS^\sigma}$,
which we call the \emph{$S^\sigma$-equivariant slice filtration}.
\end{const}

We refer to \cite[Corollary 4.39]{HP23} for a natural filtration on $p$-completed Hochschild homology.

\begin{prop}
\label{TPR.11}
Let $R$ be a commutative ring,
and let $A$ be an animated commutative $R$-algebra.
Then the natural filtration on $\HR(A/R)$ (resp.\ $\HR(A/R;\Z_p)$) can be made into a natural $S^\sigma$-equivariant filtration whose $n$th graded piece has the trivial $S^\sigma$-action for every integer $n$.
\end{prop}
\begin{proof}
The filtration in the proof of \cite[Lemma 4.27]{HP} is $S^\sigma$-equivariant using the $S^\sigma$-equivariant slice filtration in Construction \ref{TPR.10}.
Hence the natural filtration on $\HR(A/R)$ can be made into a natural $S^\sigma$-equivariant.
Together with the $S^\sigma$-equivariant $p$-completion in Construction \ref{comp.7},
we see that the natural filtration on $\HR(A/R;\Z_p)$ can be made into a natural $S^\sigma$-equivariant filtration.

The graded pieces of the natural filtrations on $\HR(A/R)$ and $\HR(A/R;\Z_p)$ are
\begin{gather*}
\gr^n \HR(A/R)
\simeq
\Sigma^{n\sigma}\iota \wedge_A^n \L_{A/R},
\\
\gr^n \HR(A/R;\Z_p)
\simeq
\Sigma^{n\sigma} \iota (\wedge_A^n \L_{A/R})_p^\wedge,
\end{gather*}
see \cite[Theorem 4.32, Corollary 4.40]{HP}.
To show that the $S^\sigma$-actions on these are trivial,
we reduce to the case when $A\in \Poly_R$ by left Kan extension.
In this case,
$\wedge_A^n \L_{A/R}$ is concentrated in degree $0$.
We finish the proof by \cite[Proposition 6.10]{QS22}.
\end{proof}

For a commutative ring $A$ and an $A$-module $M$,
let $\Gamma_A (M)$ denote the divided power algebra.
Recall that a \emph{(commutative) differential graded algebra $(A,d)$} is a (commutative) graded algebra $A$ equipped with a map $d\colon A\to A$ of degree $-1$ such that $d\circ d=0$ and $d(a\cdot b)=(da)\cdot b +(-1)^k a\cdot (db)$ for $a,b\in A$ such that $a$ is homogeneous of degree $k$.

\begin{lem}
\label{conv.1}
Let $R$ be a commutative ring of characteristic $2$.
For $A\in \Poly_R$,
there is a natural equivalence
\[
\HH(A/A^{(1)})\simeq
(\Gamma_A(\Omega_{A/R}^1),0).
\]
\end{lem}
\begin{proof}
Assume $A=R[x_1,\ldots,x_n]$.
Since $A^{(1)}\to A$ is flat,
we have
\[
B:=A\otimes_{A^{(1)}}A\simeq A\otimes_{A^{(1)}}^\L A.
\]
We can write the divided power algebra $\Gamma_B(\Omega_{A/R}^1\otimes_A B)$ as the free $B$-module
\[
\Gamma_B(\Omega_{A/R}^1\otimes_A B)
\cong
\bigoplus_{i_1,\ldots,i_n\geq 0} B \gamma_{i_1}(x_1)\cdots \gamma_{i_n}(x_n).
\]
The multiplication on  $\Gamma_B(\Omega_{A/R}^1\otimes_A B)$ is given by the formula
\[
\gamma_i(x_r)\gamma_j(x_r)
:=
\binom{i+j}{i}\gamma_{i+j}(x_r)
\]
for $1\leq r\leq n$.
Consider the map $d\colon \Gamma_B(\Omega_{A/R}^1\otimes_A B)\to \Gamma_B(\Omega_{A/R}^1\otimes_A B)$ given by the formula
\begin{align*}
&d(y\gamma_{i_1}(x_1)\cdots \gamma_{i_n}(x_n)) 
\\
:= &
\epsilon_1 y\gamma_{i_1-1}(x_1)\cdots \gamma_{i_n}(x_n)
+
\cdots
+
\epsilon_n y\gamma_{i_1}(x_1)\cdots \gamma_{i_n-1}(x_n)
\end{align*}
for $y\in B$,
where $\epsilon_r:=x_r\otimes 1-1\otimes x_r$ for $1\leq r\leq n$,
and $\gamma_{-1}=0$.
Observe that $d$ is a derivation.
We can form the commutative differential graded algebra $(\Gamma_B(\Omega_{A/R}^1\otimes_A B),d)$.

We claim that the truncation map
\begin{equation}
\label{conv.1.1}
(\Gamma_B(\Omega_{A/R}^1\otimes_A B),d)
\to
A
\end{equation}
is a quasi-isomorphism of commutative differential graded algebras.
If $n=1$,
then we can check this using $\epsilon_1^2=0$.
For general $n$,
we set $A_i:=R[x_i]$ and $B_i:=A_i\otimes_{A_i^{(1)}}A_i$.
Then there is an isomorphism of commutative differential graded algebras
\[
(\Gamma_B(\Omega_{A/R}^1\otimes_A B),d)
\cong
(\Gamma_{B_1}(\Omega_{A_1/R}^1\otimes_{A_1} B_1),d)
\otimes_R
\cdots
\otimes_R
(\Gamma_{B_n}(\Omega_{A_n/R}^1\otimes_{A_n} B_n),d).
\]
Using the case of $n=1$,
we see that the right-hand side is quasi-isomorphic to $A_1\otimes_R \cdots \otimes_R A_n$,
which is isomorphic to $A$.

Hence we have a quasi-isomorphism
\[
(\Gamma_B(\Omega_{A/R}^1\otimes_A B),d)\otimes_B A
\simeq
A\otimes_B^\L A.
\]
The left hand-side is isomorphic to $(\Gamma_A(\Omega_{A/R}^1),0)$ since the image of $\epsilon_r$ in $A$ is $0$ for $1\leq r\leq n$.
Hence we obtain the desired equivalence,
which is natural in $A$ since the truncation map \eqref{conv.1.1} is natural in $A$.
\end{proof}

\begin{lem}
\label{conv.2}
Let $R$ be a commutative ring of characteristic $2$.
For $A\in \Poly_R$ and integers $q\geq 1$ and $n$,
there exists a short exact sequence of $A$-modules
\[
0
\to
\Tor_{q+1}^B(A,B/J^n)
\to
\Gamma_A^qN \otimes_A \wedge_A^n N
\to
\Tor_q^B(A,B/J^{n+1})
\to
0,
\]
where $B:=A\otimes_{A^{(1)}}A$, $J$ is the kernel of the multiplication map $B\to A$,
and $N:=\Omega_{A/R}^1$.
\end{lem}
\begin{proof}
Compare this proof with the proof of \cite[Proposition 8.7]{Qui}.
Assume $A=R[x_1,\ldots,x_d]$.
Then we have 
\[
B\cong R[x_1,\ldots,x_d,\epsilon_1,\ldots,\epsilon_d]/(\epsilon_1^2,\ldots,\epsilon_d^2)
\]
with $\epsilon_i:=x_i\otimes 1-1\otimes x_i\in B$.
The $A$-module $J^n/J^{n+1}$ is isomorphic to the free module with the basis $\{\epsilon_{i_1}\cdots \epsilon_{i_n}:1\leq i_1<\cdots<i_n\leq d\}$.
Using this,
we have the natural isomorphism
\begin{equation}
\label{conv.2.1}
J^n/J^{n+1}
\cong
\wedge_A^n N.
\end{equation}
Together with Lemma \ref{conv.1},
we have the natural isomorphisms
\begin{equation}
\label{conv.2.2}
\Tor_q^B(A,J^n/J^{n+1})
\cong
\Tor_q^B(A,A)\otimes_A J^n/J^{n+1}
\cong
\Gamma_A^q N
\otimes_A
\wedge_A^n N.
\end{equation}

The short exact sequence
\[
0\to J^{n}/J^{n+1}\to J^{n-1}/J^{n+1}\to J^{n-1}/J^n\to 0
\]
induces the boundary map
\[
\delta\colon \Tor_q^B(A,J^{n-1}/J^n)\to \Tor_q^B(A,J^n/J^{n+1}).
\]
Since \eqref{conv.1.1} is a quasi-isomorphism,
the cofiber sequence
\[
A\otimes_B^\L J^n/J^{n+1}
\to
A\otimes_B^\L J^{n-1}/J^{n+1}
\to
A\otimes_B^\L J^{n-1}/J^n
\]
admits the model
\[
(\Gamma_A N,0)\otimes_A J^n/J^{n+1}
\xrightarrow{u}
(\Gamma_B(N\otimes_A B),d)\otimes_B J^{n-1}/J^{n+1}
\xrightarrow{v}
(\Gamma_A N,0)\otimes_A J^{n-1}/J^{n}.
\]
Let $a$ be an element of $\Gamma_A^q N\otimes_A J^{n-1}/J^n$.
Choose $b$ such that $v(b)=a$, and choose $c$ such that $u(c)=d(b)$.
Then we have $\delta(a)=c$.
Consider the specific element
\[
a:=\gamma_{i_1}(x_1)\cdots \gamma_{i_n}(x_n)\otimes y
\]
with $i_1+\cdots+i_n=q$ and $y\in J^{n-1}/J^n$,
We can take
\[
b:=\gamma_{i_1}(x_1)\cdots \gamma_{i_n}(x_n)\otimes y'
\]
with $y'\in J^{n-1}/J^{n+1}$ such that the image of $y'$ in $J^{n-1}/J^n$ is $y$.
Then we have
\[
d(b)=
(\epsilon_1\gamma_{i_1-1}(x_1) \cdots \gamma_{i_n}(x_n)
+
\cdots
+
\epsilon_n\gamma_{i_1}(x_1) \cdots \gamma_{i_n-1}(x_n)) \otimes y'.
\]
We can take
\begin{equation}
\label{conv.2.3}
c:=\gamma_{i_1-1}(x_1) \cdots \gamma_{i_n}(x_n)\otimes \epsilon_1 y'
+
\cdots
+
\gamma_{i_1}(x_1) \cdots \gamma_{i_n-1}(x_n)) \otimes \epsilon_n y',
\end{equation}
which is equal to $\delta(a)$.

By the discussion around \cite[(7.39)]{Qui},
there exists a unique endomorphism
\[
\delta\colon \Gamma_A N \otimes_A \wedge_A N
\to
\Gamma_A N \otimes_A \wedge_A N
\]
such that $\delta$ is a skew-derivation (i.e., a derivation since $R$ has characteristic $2$), $\delta(\gamma_i(a)\otimes 1)=\gamma_{i-1}(a)\otimes a$, and $\delta(1\otimes a)=0$ for $a\in N$.
Compare this with \eqref{conv.2.3} to see that the square
\begin{equation}
\label{conv.2.4}
\begin{tikzcd}
\Gamma^{q+1}_A N\otimes_A \wedge_A^{n-1}N\ar[d,"\simeq"']\ar[r,"\delta"]&
\Gamma^q_A N\otimes_A \wedge_A^n N\ar[d,"\simeq"]
\\
\Tor_{q+1}^B(A,J^{n-1}/J^n)\ar[r,"\delta"]&
\Tor_q^B(A,J^n/J^{n+1})
\end{tikzcd}
\end{equation}
commutes.

The short exact sequence
\[
0
\to
J^{n-1}/J^n
\to
B/J^n
\to
B/J^{n-1}
\to
0
\]
yields the long exact sequence
\begin{align*}
\cdots
&\xrightarrow{\partial_{n}}
\Tor_{q}^B(A,J^{n-1}/J^n)
\xrightarrow{i_{n}}
\Tor_{q}^B(A,B/J^n)
\xrightarrow{j_{n}}
\Tor_{q}^B(A,B/J^{n-1})
\\
&\xrightarrow{\partial_{n}}
\Tor_{q-1}^B(A,J^{n-1}/J^n)
\xrightarrow{i_{n}}
\cdots.
\end{align*}
We claim that $j_{n}\colon \Tor_{q}^B(A,B/J^n)\to \Tor_{q}^B(A,B/J^{n-1})$ is $0$ for $q\geq 1$.
We proceed by induction on $n$.
The claim is clear if $n=0$.
Assume that the claim holds for $n$.
Apply \cite[\S 4.3.1.7]{MR0491680} to the short exact sequence $0\to 0\to N\to N\to 0$ to see that the sequence
\begin{equation}
\label{conv.2.5}
\cdots
\xrightarrow{\delta}
\Gamma_A^{q+1} N \otimes_A \wedge_A^{n-1} N
\xrightarrow{\delta}
\Gamma_A^q N \otimes_A \wedge_A^n N
\xrightarrow{\delta}
\Gamma_A^{q-1}N\otimes_A \wedge_A^{n+1} N
\xrightarrow{\delta}
\cdots
\end{equation}
is exact.
Together with the commutativity of \eqref{conv.2.4},
we see that the bottom row of the commutative diagram
\[
\begin{tikzcd}
\Tor_{q+2}^B(A,B/J^{n-1})\ar[d,leftarrow,"i_{n-1}"']\ar[rd,"\partial_n"]&
\Tor_{q+1}^B(A,B/J^{n})\ar[d,leftarrow,"i_{n}"']\ar[rd,"\partial_{n+1}"]
\\
\Tor_{q+2}^B(A,J^{n-2}/J^{n-1})\ar[r,"\delta"]&
\Tor_{q+1}^B(A,J^{n-1}/J^n)\ar[r,"\delta"]&
\Tor_{q}^B(A,J^n/J^{n+1})
\end{tikzcd}
\]
is exact.
For $x\in \Tor_{q+1}^B(A,B/J^n)$ such that $x\in \im j_{n+1}$,
we have $\partial_{n+1}x=0$.
By induction, we have $j_{n}=0$,
which implies that $i_{n}$ is surjective.
Hence $x=i_{n}y$ for some $y$.
This implies $\delta y=0$,
so $y=\delta z$ for some $z$.
Then we have $x=i_{n}\partial_n i_{n-1}z$, which is $0$ since $i_n\partial_n=0$.
Hence we have $j_{n+1}=0$.
This completes the induction argument.
It follows that $i_n$ is surjective and $\partial_n$ is injective.

Now we have
\[
\im (\Tor_{q+1}^B(A,J^{n-1}/J^n)\xrightarrow{\delta} \Tor_q^B(A,J^n/J^{n+1}))
\simeq
\Tor_{q+1}^B(A,B/J^n).
\]
Combine this with \eqref{conv.2.2} and \eqref{conv.2.5} to obtain the desired exact sequence.
\end{proof}

\begin{lem}
\label{conv.3}
Let $X\to Y\to Z$ be a fiber sequence of chain complexes.
If $Y$ is $n$-connected, then $X$ is $(n-1)$-connected if and only if $Z$ is $n$-connected.
\end{lem}
\begin{proof}
Use the long exact sequence for the homology groups.
\end{proof}

\begin{const}
\label{conv.5}
Let $R$ be a commutative ring.
For $A\in \Poly_R$,
consider $B:=A\otimes_R R/2$ and the kernel $J$ of the multiplication map $\mu\colon C:=B\otimes_{B^{(1)}}B\to B$.
We have the ideal $J^n$ of $C$.
We also have the $B$-modules
$\Tor_q^C(B,J^n)$ and $\Tor_q^C(B,C/J^n)$ and the $C$-module $\Tor_q^C(J,J^n)$.
The constructions of $B$, $B^{(1)}$, $C$, $J^n$, $\Tor_q^C(B,J^n)$, $\Tor_q^C(B,C/J^n)$, and $\Tor_q^C(J,J^n)$ are natural in $A$.
Hence we can define these for $A\in \Ani(\Ring_R)$ too by left Kan extension.
We can also define many maps between these e.g.\ $B^{(1)}\to B$ by left Kan extension since $\Fun(X,\cC)$ admits sifted colimits for every simplicial set $X$ and $\infty$-category $\cC$ admitting sifted colimits by \cite[Corollary 5.1.2.3]{HTT}.
\end{const}

\begin{prop}
\label{conv.4}
With the above notation,
if $J$ is $0$-connected,
then $J^n$ is $(n-1)$-connected for every integer $n\geq 1$.
In particular,
the filtration $\cdots \to J^2 \to J\to B$ on $B$ is complete.
\end{prop}
\begin{proof}
Compare this proof with the proof of \cite[Convergence theorem 8.8]{Qui}.
Since $J^2$ is $(-1)$-connected,
Lemma \ref{conv.3} implies that $N:=J/J^2$ is $0$-connected.

We claim that $\Tor_q^C(B,C/J^n)$ is $(n-1)$-connected for $q>0$.
We proceed by induction on $n$.
The claim is clear if $n=0$.
Assume $n>0$.
By left Kan extension,
Lemma \ref{conv.2} yields a fiber sequence
\[
\Tor_{q+1}^C(B,C/J^{n-1})
\to
\Gamma_B^q N \otimes_B^\L \wedge_B^{n-1} N
\to
\Tor_q^C(B,C/J^{n})
\]
for $q\geq 1$.
The induction hypothesis says that $\Tor_{q+1}^C(B,C/J^{n-1})$ is $(n-2)$-connected.
By \cite[Corollary 7.40]{Qui},
$\Gamma_B^q N$ is $0$-connected,
and $\wedge_B^{n-1}N$ is $(n-2)$-connected.
Hence $\Gamma_B^q N \otimes_B^\L \wedge_B^{n-1} N$ is $(n-1)$-connected.
Lemma \ref{conv.3} implies that $\Tor_q^C(B,C/J^n)$ is $(n-1)$-connected.

Hence
$
\Tor_q^C(B,J^n)\simeq \Tor_{q+1}^C(B,C/J^n)$ and $\Tor_q^C(J,J^n)\simeq \Tor_{q+2}^C(B,C/J^n)$
are $(n-1)$-connected for $q>0$.
It follows that in the K\"unneth spectral sequence
\[
E_{pq}^2
:=
H_p\Tor_q^C(J,J^n)
\Rightarrow
H_{p+q}(J\otimes_C^\L J^n),
\]
we have $E_{pq}^2=0$ for $p<n$ and $q>0$.
In particular,
we have an isomorphism
\[
H_p \Tor_0^C(J, J^n)
\to
H_p(J\otimes_C^\L J^n)
\]
for $p\leq n$,
so $\Tor_0^C(J,J^n)$ is $n$-connected if and only if $J\otimes_C^\L J^n$ is $n$-connected.

Now, we proceed by induction on $n$ to show that $J^n$ is $(n-1)$-connected.
By assumption,
the claim holds if $n=1$.
Assume that $J^n$ is $(n-1)$-connected.
This implies that $J\otimes_C^\L J^n$ is $n$-connected,
so $\Tor_0^C(J,J^n)$ is $n$-connected too.
We have the fiber sequence
\[
\Tor_1^C(B,J^n)
\to
\Tor_0^C(J, J^n)
\to
J^{n+1},
\]
which can be first constructed when $A\in \Poly_R$ and extended to the case when $A\in \Ani(\Ring_R)$ by left Kan extension.
Together with Lemma \ref{conv.3},
we see that $J^{n+1}$ is $n$-connected.
\end{proof}

\begin{thm}
\label{geometric.6}
Let $R$ be a commutative ring.
Then for $A\in \Ani(\Ring_R)$ such that the relative Frobenius $(\pi_0(A)/2)^{(1)}\to \pi_0(A)/2$ of $R/2\to \pi_0(A)/2$ is surjective,
there exist natural complete $S^\sigma$-equivariant filtrations on $\HR(A/R)$ and $\HR(A/R;\Z_p)$ whose graded pieces are
\begin{gather*}
\gr^n \HR(A/R)
\simeq
\Sigma^{n\sigma}\iota \wedge_A^n \L_{A/R},
\\
\gr^n \HR(A/R;\Z_p)
\simeq
\Sigma^{n\sigma}\iota (\wedge_A^n \L_{A/R})_p^\wedge
\end{gather*}
with trivial $S^\sigma$-action.
\end{thm}
\begin{proof}
The part for $S^\sigma$-actions is due to Proposition \ref{TPR.11}.
We need to show that the filtrations on $\HR(A/R)$ and $\HR(A/R;\Z_p)$ in \cite[Theorem 4.32, Corollary 4.40]{HP} are complete.

We have the cofiber sequence
\(
J\to B\otimes_{B^{(1)}}^\L B \xrightarrow{\mu} B
\)
using the notation in Construction \ref{conv.5}.
Observe that we have $B\simeq A\otimes_R^\L R/2$.
Since $\mu$ has a section,
$\pi_n(\mu)$ is surjective for every integer $n$.
Hence we have the short exact sequence
\[
0\to \pi_0(J)\to \pi_0(B\otimes_{B^{(1)}}^\L B) \xrightarrow{\pi_0(\mu)} \pi_0(B)\to 0.
\]
Since $\pi_0(B)\cong \pi_0(A)/2$ and $\pi_0(B^{(1)})\cong (\pi_0(A)/2)^{(1)}$, the assumption that the relative Frobenius $(\pi_0(A)/2)^{(1)}\to \pi_0(A)/2$ is surjective implies that $\pi_0(B^{(1)})\to \pi_0(B)$ is surjective.
Using $\pi_0(B\otimes_{B^{(1)}}^\L B)\simeq \pi_0(B)\otimes_{\pi_0(B^{(1)})}\pi_0(B)$,
we see that $\pi_0(\mu)$ is an isomorphism.
Hence we have $\pi_0(J)=0$, i.e., $J$ is $0$-connected.

For integers $i$ and $n$,
let
\[
\Fil_i \Phi^{\Z/2}\Fil_n \HR(-/R)
\colon
\Ani(\Ring_R)
\to
\Sp
\]
be a left Kan extension of the functor
\[
\tau_{\geq i}\Phi^{\Z/2}\Fil_n \HR(-/R)
\colon
\Poly_R
\to
\Sp.
\]
Since $\Phi^{\Z/2}\Fil_n\HR(A/R)$ is $(-1)$-connected for $A\in \Poly_R$,
we have
\[
\Fil_i \Phi^{\Z/2}\Fil_n\HR(-/R)\simeq \Fil_0 \Phi^{\Z/2}\Fil_n\HR(-/R)
\]
for $i\leq 0$.
Observe that $\Fil_i\Phi^{\Z/2}\Fil_n \HR(-/R)$ is $(i-1)$-connected.
By Propositions \ref{geometric.5} and \ref{conv.4},
$\gr^i \Phi^{\Z/2}\Fil_n \HR(A/R)$ is $(n+i-1)$-connected.
It follows that $\Phi^{\Z/2}\Fil_n \HR(A/R)$ is $(n-1)$-connected.
On the other hand,
$i^*\Fil_n \HR(A/R)$ is $(n-1)$-connected by Proposition \ref{geometric.8},
which implies that the homotopy orbit $(\Fil_n \HR(A/R))_{h\Z/2}$ is $(n-1)$-connected too.
Using the isotropy separation cofiber sequence
\[
(\Fil_n \HR(A/R))_{h\Z/2}
\to
(\Fil_n \HR(A/R))^{\Z/2}
\to
\Phi^{\Z/2}\Fil_n \HR(A/R),
\]
we see that $(\Fil_n \HR(A/R))^{\Z/2}$ is $(n-1)$-connected.
Hence $\Fil_n \HR(A/R)$ is $(n-1)$-connected.
This implies $\lim_n \Fil_n \HR(A/R)\simeq 0$,
i.e.,
the filtration on $\HR(A/R)$ is complete.
\end{proof}

\begin{rmk}
For a map of commutative rings $R\to A$,
the Frobenius $\varphi\colon A/2\to A/2$ factors through the relative Frobenius $(A/2)^{(1)}\to A/2$ of $R/2\to A/2$,
see \eqref{geometric.3.1}.
Hence the condition in Theorem \ref{geometric.6} is satisfied if $A$ is a commutative ring such that the Frobenius $\varphi\colon A/2\to A/2$ is surjective.
This completes the proof of \cite[Theorem 4.41]{HP},
so from now on we can use the results in \cite[\S 5]{HP}.
\end{rmk}

\section{Quasisyntomic sheaves}

We refer to \cite[\S 4]{BMS19} for the notions of quasisyntomic site, quasisyntomic maps, quasisyntomic rings, perfectoid rings, and quasiregular semiperfectoid rings.

Theorem \ref{TPR.6} below is a real refinement of \cite[Corollary 3.4]{BMS19}.
In loc.\ cit., the proof is done by reducing to the case of the cotangent complex using the Hochschild--Kostant--Rosenberg filtration.
We instead directly check the sheaf conditions arguing as in the proof of \cite[Theorem 3.1]{BMS19}.

Recall from \cite[Remark 3.13]{QS22} that we have the $t$-structure on $(\Sp^{\Z/2})^{BS^\sigma}$ given by
\begin{gather*}
(\Sp^{\Z/2}_{\geq 0})^{BS^\sigma}
:=
\Fun_{\Z/2}(B S^\sigma,\ul{\Sp}^{\Z/2}_{\geq 0}),
\\
(\Sp^{\Z/2}_{\leq 0})^{BS^\sigma}
:=
\Fun_{\Z/2}(B S^\sigma,\ul{\Sp}^{\Z/2}_{\leq 0}).
\end{gather*}

\begin{df}
\label{TPR.1}
An infinite tower $\cdots \to X_1\to X_0\to X_{-1}:=0$ of $\Z/2$-spectra is a \emph{weak Postnikov tower} if $\fib(X_n\to X_{n-1})$ is in $\Sigma^n \Sp_{\geq 0}^{\Z/2}$ for $n\geq 0$.
\end{df}

\begin{lem}
\label{TPR.2}
Let $\cdots \to X_1\to X_0$ be a tower of $\Z/2$-spectra with $S^\sigma$-action.
If this is a weak Postnikov tower after forgetting the $S^\sigma$-action,
then the tower $\cdots \to (X_1)_{h S^\sigma}\to (X_0)_{h S^\sigma}$ is a weak Postnikov tower,
and there is a natural equivalence of $\Z/2$-spectra
\[
(\lim_n X_n)_{h S^\sigma}
\simeq
\lim_n (X_n)_{h S^\sigma}.
\]
\end{lem}
\begin{proof}
The forgetful functor $(\Sp^{\Z/2})^{BS^\sigma}\to \Sp^{\Z/2}$ is $t$-exact,
so its left adjoint $(-)_{hS^\sigma}$ is right $t$-exact.
This implies that $\{(X_i)_{h S^\sigma}\}_{i\geq 0}$ is a weak Postnikov tower.

We set $X:=\lim_n X_n$.
We have $\fib(X\to X_n)\in \Sigma^n (\Sp_{\geq 0}^{\Z/2})^{BS^\sigma}$.
Since $(-)_{hS^\sigma}$ is right $t$-exact,
we have $\fib(X_{hS^\sigma}\to (X_n)_{h S^\sigma})\in \Sigma^n \Sp_{\geq 0}^{\Z/2}$.
Apply $\lim_n$ to conclude.
\end{proof}

\begin{lem}
\label{descent.2}
Let $R\to A,B$ be maps of normed $\Z/2$-spectra.
Then there is a natural equivalence of normed $\Z/2$-spectra with $S^\sigma$-action
\[
\HR(A/R)\wedge_{R} B
\simeq
\HR(A\wedge_R B/B).
\]
\end{lem}
\begin{proof}
We have an equivalence of normed $\Z/2$-spectra with $S^\sigma$-action
\[
\THR(A)\wedge_{\THR(R)}B
\simeq
\THR(A)\wedge_{\THR(R)}\THR(B)\wedge_{\THR(B)}B.
\]
By \cite[Proposition 2.1.5]{HP23} and Proposition \ref{G-cat.3},
we have an equivalence of normed $\Z/2$-spectra with $S^\sigma$-action
\[
\THR(A)\wedge_{\THR(R)}\THR(B)
\simeq
\THR(A\wedge_R B).
\]
Combine these two to obtain the desired equivalence.
\end{proof}

\begin{lem}
\label{descent.1}
Let $A$ be a commutative algebra object of $\Sp^{\Z/2}$.
If $M$ is an $A$-module with $S^\sigma$-action and $N$ is an $A$-module with trivial $S^\sigma$-action,
then we have a natural equivalence of $\Z/2$-spectra
\[
(M\wedge_A N)_{hS^\sigma}\simeq M_{hS^\sigma} \wedge_A N,
\]
where $A$ is equipped with with trivial $S^\sigma$-action.
\end{lem}
\begin{proof}
By \cite[\S 4.4.1, Theorem 4.5.2.1(2)]{HA},
it suffices to show that the induced map
\[
\colim
\big(
\cdots
\,
\substack{\rightarrow\\[-1em] \rightarrow \\[-1em] \rightarrow}
\,
M\wedge A\wedge N
\,
\substack{\rightarrow\\[-1em] \rightarrow}
\,
M\wedge N
\big)_{hS^\sigma}
\to
\colim
\big(
\cdots
\,
\substack{\rightarrow\\[-1em] \rightarrow \\[-1em] \rightarrow}
\,
M_{hS^\sigma}\wedge A\wedge N
\,
\substack{\rightarrow\\[-1em] \rightarrow}
\,
M_{hS^\sigma}\wedge N
\big)
\]
is an equivalence,
where the colimits run over the two-sided bar constructions.
We can show this using the projection formula \cite[Example 5.19, Lemma 5.45]{QS21}.
\end{proof}

Let $\CRing$ denote the category of commutative rings.
For a commutative ring $R$,
let $\CRing_R$ denote the category of $R$-algebras.
Following \cite[Variant 4.33]{BMS19},
let $\QSyn_R$ denote the category of $R$-algebras $A$ such that $A$ is quasisyntomic.

In the proof of Theorem \ref{TPR.6} below,
we will use the following property \cite[Corollary 6.8.1]{BGS20}:
For a $(-1)$-connected normed $\Z/2$-spectrum $A$ and $(-1)$-connected $A$-modules $M$ and $N$,
the monoidal product $M\wedge_A N$ is $(-1)$-connected,
and we have a natural isomorphism
\[
\ul{\pi}_0(M\wedge_A N)
\simeq
\ul{\pi}_0 M \sotimes_{\ul{\pi}_0 A}\ul{\pi}_0 M.
\]

\begin{thm}
\label{TPR.6}
Let $R$ be a commutative ring.
The presheaves
\[
\HR(-/R),
\;
\HCR^-(-/R),
\;
\HR(-/R)_{hS^\sigma}
\;
\HPR(-/R)
\]
on $\CRing_R^{op}$ are fpqc-sheaves.
The presheaves
\[
\THR(-),
\;
\TCR^-(-),
\;
\THR(-)_{hS^\sigma},
\;
\TPR(-)
\]
on $\CRing^{op}$ are fpqc-sheaves.
The presheaves
\[
\HR(-/R;\Z_p),
\;
\HCR^-(-/R;\Z_p),
\;
\HR(-/R;\Z_p)_{hS^\sigma},
\;
\HPR(-/R;\Z_p)
\]
on $\QSyn_R^{op}$ are quasisyntomic sheaves.
The presheaves
\[
\THR(-;\Z_p),
\;
\TCR^-(-;\Z_p),
\;
\THR(-;\Z_p)_{hS^\sigma},
\;
\TPR(-;\Z_p)
\]
on $\QSyn^{op}$ are quasisyntomic sheaves.
\end{thm}
\begin{proof}
Let us first show that $\HR(-/R)_{h S^\sigma}$ is a sheaf.
Let $A\to B$ be a faithfully flat map of rings,
and let $B^\bullet$ denote its \v{C}ech nerve.
We need to show that the induced map of $\Z/2$-spectra
\[
\HR(A/R)_{hS^\sigma}
\to
\lim_m \HR(B^m/R)_{hS^\sigma}.
\]
is an equivalence.
Assume that the induced map
\begin{equation}
\label{TPR.6.1}
(M_i)_{hS^\sigma}
\to
\lim_m (\HR(B^m/R)\wedge_{\HR(A/R)}M_i)_{hS^\sigma}
\end{equation}
is an equivalence with $M_i:=\rH  \ul{\pi}_i \HR(A/R)$ for every $i$.
By induction,
the induced map
\[
(\tau_{\leq i} \HR(A/R))_{hS^\sigma}
\to
\lim_m (\HR(B^m/R)\wedge_{\HR(A/R)} \tau_{\leq i}\HR(A/R))_{hS^\sigma}
\]
is an equivalence.
It follows that the induced square
\[
\begin{tikzcd}
\HR(A/R)_{hS^\sigma}\ar[d]\ar[r]&
\lim_m \HR(B^m/R)_{hS^\sigma}\ar[d]
\\
(\tau_{\geq i+1} \HR(A/R))_{hS^\sigma}
\ar[r]&
\lim_m (\HR(B^m/R)\wedge_{\HR(A/R)} \tau_{\geq i+1}\HR(A/R))_{hS^\sigma}
\end{tikzcd}
\]
is cocartesian.
If we apply $\lim_i$ to the lower horizontal arrow,
then we get $0\to 0$ using \cite[Corollary 6.8.1]{BGS20} for the right-hand side since $(-)_{hS^\sigma}$ is right $t$-exact.
Hence the upper horizontal arrow is an equivalence,
so it suffices to show that \eqref{TPR.6.1} is an equivalence.

Assume that the induced map
\begin{equation}
\label{TPR.6.2}
\rH \ul{\pi}_j (M_i)_{hS^\sigma}
\to
\lim_m
\rH \ul{\pi}_j (\HR(B^m/A)\wedge_{\rH \ul{A}}M_i)_{hS^\sigma}
\end{equation}
is an equivalence for every integer $j$.
By induction, the induced map
\[
\tau_{\leq j} (M_i)_{hS^\sigma}
\simeq
\lim_m \tau_{\leq j} (\HR(B^m/A)\wedge_{\rH \ul{A}}M_i)_{hS^\sigma}
\]
is an equivalence.
It follows that the induced square
\[
\begin{tikzcd}
(M_i)_{hS^\sigma}\ar[d]\ar[r]&
\lim_m (\HR(B^m/A)\wedge_{\rH \ul{A}}M_i)_{hS^\sigma}\ar[d]
\\
\tau_{\geq j+1} (M_i)_{hS^\sigma}\ar[r]&
\lim_m \tau_{\geq j+1} (\HR(B^m/A)\wedge_{\rH \ul{A}}M_i)_{hS^\sigma}
\end{tikzcd}
\]
is cocartesian.
If we apply $\lim_j$ to the lower horizontal arrow,
then we get $0\to 0$ using \cite[Corollary 6.8.1]{BGS20} for the right-hand side since $(-)_{hS^\sigma}$ is right $t$-exact.
Hence the upper horizontal arrow is an equivalence,
so it suffices to show that \eqref{TPR.6.2} is an equivalence.

Using the Dold--Kan correspondence for abelian categories \cite[Tag 019G]{stacks},
it suffices to show that the induced map of cochain complexes
\[
\ul{\pi}_j (M_i)_{hS^\sigma}
\to
\ul{\pi}_j (\HR(B^\bullet/A)\wedge_{\rH \ul{A}} M_i)_{hS^\sigma}
\]
is a quasi-isomorphism.
By fpqc-descent and \cite[Lemma 2.17]{HP},
it suffices to show that the induced map of cochain complexes
\[
(\ul{\pi}_j (M_i)_{hS^\sigma})\sotimes_{\ul{A}}\ul{B}
\to
(\ul{\pi}_j (\HR(B^\bullet/A)\wedge_{\rH \ul{A}}M_i)_{hS^\sigma})\sotimes_{\ul{A}}\ul{B}
\]
is a quasi-isomorphism.
Since $A\to B$ is flat,
$B$ is a filtered colimit of free $A$-modules by Lazard's theorem \cite[Tag 058G]{stacks}.
Hence $\rH \ul{A}\to \rH \ul{B}$ is flat in the sense of \cite[Definition A.5.4]{HP23} by \cite[Proposition A.5.11]{HP23},
so it suffices to show that the induced map of cochain complexes
\[
\ul{\pi}_j ((M_i)_{hS^\sigma}\wedge_{\rH \ul{A}}\rH \ul{B})
\to
\ul{\pi}_j ((\HR(B^\bullet/A)\wedge_{\rH \ul{A}}M_i)_{hS^\sigma}\wedge_{\rH \ul{A}}\rH \ul{B})
\]
is a quasi-isomorphism.
By Lemma \ref{descent.2} and \ref{descent.1},
it suffices to show that the induced map of cochain complexes
\[
\ul{\pi}_j (M_i\wedge_{\rH \ul{A}}\rH \ul{B})_{hS^\sigma}
\to
\ul{\pi}_j (M_i\wedge_{\rH \ul{A}}\HR(C^\bullet/B))_{hS^\sigma}
\]
is a quasi-isomorphism with $C^\bullet:=B^\bullet\otimes_{A}B$.

The diagonal map $B\to B\otimes_A B$ has a section,
so its \v{C}ech nerve $B\to C^\bullet$ is a cosimplicial homotopy equivalence of $B$-algebras.
This means that for every Mackey functor valued functor $F$ on $B$-algebras,
the induced map of cochain complexes $F(B)\to F(C^\bullet)$ is a quasi-isomorphism.
Apply this to $F:=\ul{\pi}_j (M_i\wedge_{\rH \ul{A}} \HR(-/B))_{hS^\sigma}$ to finish the proof that $\HR(-/A)_{hS^\sigma}$ is a sheaf.

Let us spell out the extra argument for showing that $\HR(-/R;\Z_p)_{hS^\sigma}$ is a sheaf,
see also \cite[Remark 4.9]{BMS19}.
Let $A\to B$ be a quasisyntomic cover in $\QSyn_R$.
Arguing as above,
we reduce to showing that the induced map of pro-cochain complexes 
\[
\{
\ul{\pi}_j (M_i\wedge_{\rH \ul{\Z}}\rH \ul{\Z/p^n})_{hS^\sigma}
\}_n
\to
\{
\ul{\pi}_j (\HR(B^\bullet/A)\wedge_{\rH \ul{A}} M_i\wedge_{\rH \ul{\Z}}\rH \ul{\Z/p^n})_{hS^\sigma}
\}_n
\]
is a pro-quasi-isomorphism since this will imply that the map becomes a quasi-isomorphism after the taking homotopy limits over $n$.
Since $\{A/p^n\}_n$ and $\{A\otimes_{\Z}^\L\Z/p^n\}_n$ are pro-quasi-isomorphic by the assumption that $A$ has bounded $p^\infty$-torsion,
it suffices to show that the induced map of pro-cochain complexes 
\begin{align*}
&\{
\ul{\pi}_j ((M_i\wedge_{\rH \ul{\Z}}\rH \ul{\Z/p^n})\wedge_{A_n} \rH \ul{A/p^n})_{hS^\sigma}
\}_n
\\
\to &
\{
\ul{\pi}_j ((\HR(B^\bullet/A)\wedge_{\rH \ul{A}} M_i\wedge_{\rH \ul{\Z}}\rH \ul{\Z/p^n})\wedge_{A_n} \rH \ul{A/p^n})_{hS^\sigma}
\}_n
\end{align*}
is a pro-quasi-isomorphism with $A_n:=\rH \ul{A}\wedge_{\rH \ul{\Z}}\rH \ul{\Z/p^n}$.
By fpqc-descent and \cite[Lemma 2.17]{HP},
it suffices to show that the induced map of pro-cochain complexes
\begin{align*}
&\{
\ul{\pi}_j ((M_i\wedge_{\rH \ul{\Z}}\rH \ul{\Z/p^n})\wedge_{A_n} \rH \ul{A/p^n})_{hS^\sigma}\sotimes_{\ul{A/p^n}}\ul{B/p^n}
\}_n
\\
\to &
\{
\ul{\pi}_j ((\HR(B^\bullet/A)\wedge_{\rH \ul{A}} M_i\wedge_{\rH \ul{\Z}}\rH \ul{\Z/p^n})\wedge_{A_n} \rH \ul{A/p^n})_{hS^\sigma}\sotimes_{\ul{A/p^n}}\ul{B/p^n}
\}_n
\end{align*}
is a pro-quasi-isomorphism since $A/p^n\to B/p^n$ is faithfully flat.
By Lemma \ref{descent.1} and the flatness of $A/p^n\to B/p^n$,
it suffices to show that the induced map of cochain complexes
\begin{align*}
&\{
\ul{\pi}_j ((M_i\wedge_{\rH \ul{\Z}}\rH \ul{\Z/p^n})\wedge_{A_n} \rH \ul{B/p^n})_{hS^\sigma}
\}_n
\\
\to &
\{
\ul{\pi}_j ((\HR(B^\bullet/A)\wedge_{\rH \ul{A}} M_i\wedge_{\rH \ul{\Z}}\rH \ul{\Z/p^n})\wedge_{A_n} \rH \ul{B/p^n})_{hS^\sigma}
\}_n
\end{align*}
is a pro-quasi-isomorphism.
Since $\{B/p^n\}_n$ and $\{B\otimes_{\Z}^\L\Z/p^n\}_n$ are pro-quasi-isomorphic by the assumption that $B$ has bounded $p^\infty$-torsion,
by Lemma \ref{descent.2},
it suffices to show that the induced map of cochain complexes
\[
\{
\ul{\pi}_j (M_i\wedge_{\rH \ul{A}}\rH \ul{B}\wedge_{\rH \ul{\Z}}\rH \ul{\Z/p^n})_{hS^\sigma}
\}_n
\to 
\{
\ul{\pi}_j (M_i\wedge_{\rH \ul{A}} \HR(C^\bullet/B)\wedge_{\rH \ul{\Z}}\rH \ul{\Z/p^n})_{hS^\sigma}
\}_n
\]
is a pro-quasi-isomorphism  with $C^\bullet:=B^\bullet\otimes_{A}B$.
Arguing as in the remaining part of the proof that $\HR_{hS^\sigma}$ is a sheaf,
we can show that $\HR(-;\Z_p)_{hS^\sigma}$ is a sheaf.

We can similarly show that $\HR(-/R)$ and $\HR(-/R;\Z_p)$ are sheaves.

We claim that $\THR(-;\Z_p)_{hS^\sigma}$ is a sheaf.
Since $\THR(A;\Z_p)$ is $(-1)$-connected for every commutative ring $A$ by \cite[Proposition 4.13]{HP},
we see that
$\THR(A;\Z_p)\wedge_{\THR(\Z)}\tau_{\leq i} \THR(\Z)$ is a weak Postnikov tower.
By Lemma \ref{TPR.2},
it suffices to show that $(\THR(-;\Z_p)\wedge_{\THR(\Z)}\tau_{\leq i}\THR(\Z))_{hS^\sigma}$ is a sheaf for every integer $i$.
Using induction,
it suffices to show that $(\THR(-;\Z_p)\wedge_{\THR(\Z)}\ul{\pi}_i\THR(\Z))_{hS^\sigma}$ is a sheaf.
Together with \cite[Lemma 4.14, Construction 4.16, Proposition 4.25]{HP},
it suffices to show that
$\HR(-/\Z;\Z_p)_{hS^\sigma}\simeq (\THR(-;\Z_p)\wedge_{\THR(\Z)}\rH \ul{\Z})_{hS^\sigma}$ is a sheaf,
which is already proven above.

We can similarly show that $\THR$, $\THR(-;\Z_p)$, and $\THR_{hS^\sigma}$ are sheaves.

The remaining presheaves are sheaves since $(-)^{hS^\sigma}$ preserves limits and we have the cofiber sequence $\Sigma^\sigma (-)_{hS^\sigma}\to (-)^{hS^\sigma}\to (-)^{tS^\sigma}$.
\end{proof}

Using that real Hochschild homology is a sheaf,
we can show that the real Hochschild--Kostant--Rosenberg filtration is complete for the quasisyntomic case as follows.

\begin{thm}
\label{descent.3}
Let $R\to A$ be a map of commutative rings.
If $R$ is perfectoid and $A$ is quasisyntomic,
then there exists a natural complete $S^\sigma$-equivariant filtration on $\HR(A/R;\Z_p)$ whose graded pieces are
\[
\gr^n \HR(A/R;\Z_p)
\simeq
\Sigma^{n\sigma}\iota (\wedge_A^n \L_{A/R})_p^\wedge
\]
with trivial $S^\sigma$-action.
\end{thm}
\begin{proof}
By \cite[Remark 4.9]{BMS19} and Theorem \ref{TPR.6},
$\gr^n \HR(-/R;\Z_p)$ and $\HR(-/R;\Z_p)$ are quasisyntomic sheaves.
Hence $\Fil_n\HR(-/R;\Z_p)$ is a quasisyntomic sheaf.
By \cite[Proposition 4.31, Variant 4.33]{BMS19},
it suffices to show that the filtration on $\HR(S/R;\Z_p)$ is complete for every semiperfectoid quasiregular $R$-algebra $S$.
This follows from \cite[Lemma 4.25]{BMS19} and Theorem \ref{geometric.6}.
\end{proof}

\section{\texorpdfstring{$\THR$}{THR} of quasiregular semiperfectoid rings}

In this section,
we show that $\THR(S;\Z_p)$ is strongly even for every quasiregular semiperfectoid ring $S$.

\begin{prop}
\label{THR.7}
Let $R\to S$ be a map of commutative rings,
where $R$ is perfectoid,
and $S$ is quasiregular semiperfectoid.
Then $\HR(S/R;\Z_p)$ is strongly even.
\end{prop}
\begin{proof}
By Theorem \ref{descent.3},
we have a natural complete filtration on $\HR(S/R;\Z_p)$ whose $n$th graded piece is
\(
\gr^n \HR(S/R;\Z_p)
\simeq
\Sigma^{n\sigma} \iota (\wedge_{S}^n  \L_{S/R})_p^\wedge
\)
for every integer $n$.
Using \cite[Lemma 5.14(1)]{BMS19},
$(\wedge_{A}^n  \L_{S/R})_p^\wedge$ is equivalent to $\Sigma^n \rH M_n$ in $\Sp$ for some $S$-module $M_n$.
Then $\gr^n \HR(S/R;\Z_p)$ is equivalent to $\Sigma^{n+n\sigma} \rH \ul{M_n}$ in $\Sp^{\Z/2}$,
which is strongly even.
Lemma \ref{THR.5} finishes the proof.
\end{proof}

The next result is an equivariant refinement of \cite[Theorem 7.1(1)]{BMS19}.

\begin{thm}
\label{THR.6}
Let $S$ be a quasiregular semiperfectoid ring.
Then $\THR(S;\Z_p)$ is strongly even.
Hence for every integer $n$,
there is a natural isomorphism
\[
\rho_{2n}\THR(S;\Z_p)
\simeq
\ul{\pi_{2n}\THH(S;\Z_p)}.
\]
\end{thm}
\begin{proof}
There exists a perfectoid ring $R$ with a map of commutative rings $R\to S$ by definition.
Consider the filtration
\[
\Fil_n \THR(S;\Z_p):=\THR(S;\Z_p)\wedge_{\THR(R;\Z_p)}P_{2n} \THR(R;\Z_p)
\]
on $\THR(S;\Z_p)$,
where $n$ is an integer.
We have equivalences of $\Z/2$-spectra
\[
\gr^n \THR(S;\Z_p)
\simeq
\THR(S;\Z_p)\wedge_{\THR(R;\Z_p)}\Sigma^{n+n\sigma} \rH \ul{R}
\simeq
\Sigma^{n+n\sigma}\HR(S/R;\Z_p),
\]
where the first one follows from \cite[Theorem 5.15]{HP}, and the second one follows from \cite[Theorem 5.16]{HP}.
Proposition \ref{THR.7} implies that $\gr^n \THR(S;\Z_p)$ is strongly even.
By \cite[Lemma 4.22]{HP},
the filtration $\Fil_n\THR(S;\Z_p)$ is complete.
Lemma \ref{THR.5} implies that $\THR(S;\Z_p)$ is strongly even.
Now the desired isomorphism follows from \eqref{eq:very even}.
\end{proof}

\begin{cor}\label{moreonqrsperfd}
Let $R\to S$ be a map of commutative rings,
where $R$ is perfectoid,
and $S$ is quasiregular semiperfectoid.
\begin{enumerate}
\item[\textup{(1)}]
For every integer $i$,
the morphism
\[
\rho_{2i-2}\THR(S;\Z_p)
\to
\rho_{2i}\THR(S;\Z_p)
\]
given by multiplying a generator $\tilde{u}\in \rho_2\THR(R;\Z_p)$ is a monomorphism of Mackey functors.
\item[\textup{(2)}]
Let $\rho_{\infty}\THR(S;\Z_p):=\colim_i \rho_{2i}\THR(S;\Z_p)$ be the colimit of multiplication by $\tilde{u}$,
which we view as an increasingly filtered commutative $\ul{R}$-algebra.
We set $M:=\pi_1(\L_{S/R})_p^\wedge$.
Then there is a natural isomorphism of graded Green functors
\[
\ul{(\Gamma_S M)_p^\wedge}
\cong
\gr_* \rho_{\infty}\THR(S;\Z_p).
\]
In particular,
$\rho_{2i}\THR(S;\Z_p)$ admits a finite filtration for every integer $i$ whose $j$th graded pieces are given by $\ul{(\Gamma_S^* M)_p^\wedge}$ for $0\leq j\leq i$.
\end{enumerate}
\end{cor}
\begin{proof}
Combine \cite[Theorem 7.1(2),(3)]{BMS19} with Theorem \ref{THR.6}.
\end{proof}

\begin{cor}
Let $R\to A$ be a map of commutative rings,
where $R$ is perfectoid,
and $A$ is quasisyntomic.
Then $\THR(A;\Z_p)$ admits a natural complete filtration such that $\gr^i\THR(A;\Z_p)$ admits a finite filtration with graded pieces given by \[
\Sigma^{i+i\sigma-j}\iota(\wedge_A^j \L_{A/R})^{\wedge}_p
\]
for $0 \leq j \leq i$.
\end{cor}    
\begin{proof}
If $A$ is quasiregular semiperfectoid,
then consider the double speed Postnikov filtration on $\THR(A;\Z_p)$,
and use Corollary \ref{moreonqrsperfd}(2).
Together with \cite[Remark 4.9, Proposition 4.31, Variant 4.33]{BMS19} and Theorem \ref{TPR.6},
we finish the proof for quasisyntomic $A$ by quasisyntomic descent.
\end{proof}

\section{\texorpdfstring{$\TCR^-$}{TCR-} and \texorpdfstring{$\TPR$}{TPR} of quasiregular semiperfectoid rings}

So far,
we have shown that $\THR(S;\Z_p)$ is strongly even for every quasiregular semiperfectoid ring $S$.
Based on this,
we show that $\TCR^-(S;\Z_p)$ and $\TPR(S;\Z_p)$ are strongly even too in this section.

\begin{prop}
\label{TPR.5}
Let $M$ be an abelian group.
Then $\rH \ul{M}_{hS^\sigma}$, $\rH \ul{M}^{hS^\sigma}$, and $\rH \ul{M}^{tS^\sigma}$ are strongly even.
\end{prop}
\begin{proof}
Note that every $S^\sigma$-action on $\rH \ul{M}$ is trivial by \cite[Proposition 6.10]{QS22}.
Consider the commutative ring $R:=\Z\oplus M$ whose multiplication is given by $(a,x)(b,y):=(ab,ay+bx)$.
The claim for $R$ implies the claim for $M$ since $\rH \ul{R}\simeq \rH \ul{\Z}\oplus \rH\ul{M}$.
Hence we reduce to the case when $M$ is a commutative ring.

As in the proofs of \cite[Lemma 4.7, Proposition 6.14]{QS22},
there are equivalences of $\Z/2$-spectra
\begin{gather*}
\rH \ul{M}_{hS^\sigma}
\simeq
\coprod_{i=0}^\infty \Sigma^{i+i\sigma} \rH \ul{M},
\\
\rH \ul{M}^{hS^\sigma}
\simeq
\prod_{i=0}^\infty \Sigma^{-i-i\sigma} \rH \ul{M}.
\end{gather*}
Hence $\rH \ul{M}_{hS^\sigma}$ and $\rH \ul{M}^{hS^\sigma}$ are strongly even.

Using the cofiber sequence $(-)^{hS^\sigma}\to (-)^{tS^\sigma}\to \Sigma^{1+\sigma}(-)_{hS^\sigma}$ and Lemma \ref{THR.4},
we deduce that $\rH \ul{M}^{tS^\sigma}$ is strongly even.
\end{proof}

\begin{thm}
\label{TPR.7}
Let $S$ be a quasiregular semiperfectoid ring.
Then $\TCR^-(S;\Z_p)$ and $\TPR(S;\Z_p)$ are strongly even.
Hence for every integer $n$,
there are natural isomorphisms
\begin{gather*}
\rho_{2n}\TCR^-(S;\Z_p)
\simeq
\ul{\pi_{2n}\TCR^-(S;\Z_p)},
\\
\rho_{2n}\TPR(S;\Z_p)
\simeq
\ul{\pi_{2n}\TPR(S;\Z_p)}.
\end{gather*}
\end{thm}
\begin{proof}
Since $(-)^{hS^\sigma}$ preserves limits,
$(P_{2\bullet} \THR(S;\Z_p))^{hS^\sigma}$ is a complete filtration.
Its graded pieces are strongly even by Theorem \ref{THR.6} and Proposition \ref{TPR.5}.
Lemma \ref{THR.5} implies that $\TCR^-(S;\Z_p)$ is strongly even.

We claim that $\THR(S;\Z_p)_{hS^\sigma}$ is strongly even.
Since $\THR(S;\Z_p)$ is strongly even by Theorem \ref{THR.6},
we have a fiber sequence
\[
P_{2n}^{2n}\THR(S;\Z_p)
\to
P^{2n}\THR(S;\Z_p)
\to
P^{2n-2}\THR(S;\Z_p)
\]
for every integer $n$.
Using Proposition \ref{slice.2},
we see that $P^{2\bullet}\THR(S;\Z_p)$ is a weak Postnikov tower.
By Lemma \ref{TPR.2},
we see that $(P_{2\bullet}\THR(S;\Z_p))_{hS^\sigma}$ is a complete filtration.
Since its $n$th graded piece $(P_{2n}^{2n}\THR(S;\Z_p))_{hS^\sigma}$ is strongly even by Proposition \ref{TPR.5},
Lemma \ref{THR.5} implies that $\THR(S;\Z_p)_{hS^\sigma}$ is strongly even.

Using the cofiber sequence $(-)^{hS^\sigma}\to (-)^{tS^\sigma}\to \Sigma^{1+\sigma}(-)_{hS^\sigma}$ and Lemma \ref{THR.4},
we deduce that $\TPR(S;\Z_p)$ is strongly even.
\end{proof}

\section{\texorpdfstring{$\TCR$}{TCR} of quasiregular semiperfectoid rings}\label{TCR}

Recall from \cite[\S 7.4]{BMS19} that for a quasiregular semiperfectoid ring $S$ and an integer $n$,
the \emph{syntomic cohomology} is
\[
\Z_p(n)(S)
:=
\fib(\rH \pi_{2n}\TC^-(S;\Z_p)
\xrightarrow{\varphi-\can}
\rH \pi_{2n}\TP(S;\Z_p))
\in \Sp.
\]
Furthermore,
we obtain 
$\Z_p(n)(A)$ for quasisyntomic ring $A$ by quasisyntomic descent \cite[Proposition 4.31]{BMS19}.

\begin{thm}
\label{TCR.1}
Let $A$ be a quasisyntomic ring.
Then there exist natural complete exhaustive multiplicative filtrations on $\THR(A;\Z_p)$, $\TCR^-(A;\Z_p)$, $\TPR(A;\Z_p)$, and $\TCR(A;\Z_p)$ whose $n$th graded pieces are
\begin{gather*}
\gr^n \THR(A;\Z_p)
\simeq
\Sigma^{n+n\sigma}
\iota \cN^n  \widehat{\prism}_A\{n\},
\\
\gr^n \TCR^-(A;\Z_p)
\simeq
\Sigma^{n+n\sigma}
\iota \cN^{\geq n} \widehat{\prism}_A\{n\},
\\
\gr^n \TPR(A;\Z_p)
\simeq
\Sigma^{n+n\sigma}
\iota \widehat{\prism}_A\{n\},
\\
\gr^n \TCR(A;\Z_p)
\simeq
\Sigma^{n+n\sigma}
\iota \Z_p(n)(A).
\end{gather*}
\end{thm}
\begin{proof}
By quasisyntomic descent \cite[Proposition 4.31]{BMS19} and Theorem \ref{TPR.6},
we reduce to the case when $A$ is quasiregular semiperfectoid.
Then the claim is a consequence of \cite[Theorem 1.12]{BMS19},
Theorems \ref{THR.6} and \ref{TPR.7},
and the formulation of $\TCR(A;\Z_p)$ in Definition \ref{comp.8}.
\end{proof}

\begin{exm}
\label{TCR.5}
Quigley and Shah \cite[Theorem 6.30]{QS22} proved that there exists an equivalence of $\Z/2$-spectra
\[
\TCR(\rH \ul{k})
\simeq
\rH \ul{\Z_p} \oplus \Sigma^{-1} \rH \ul{\coker(1-F)}
\]
for every perfect field $k$ of characteristic $p>0$,
where $F$ denotes the Frobenius on $W(k)$.
We can recover the computation of $\TCR(\rH\ul{k};\Z_p)$ by combining Theorem \ref{TCR.1} and the computation Hesselholt and Madsen \cite[Theorem B]{MR1410465}
\[
\TC(\rH \ul{k})
\simeq
\rH \Z_p \oplus \Sigma^{-1} \rH \coker(1-F).
\]
as follows.
Since $k$ is quasiregular semiperfectoid,
$\Z_p(0)(k)$ agrees with this, and $\Z_p(i)(k)$ vanishes for every integer $i>0$.
Use
\[
\iota (\rH \Z_p \oplus \Sigma^{-1} \rH \coker(1-F))
\simeq
\rH \ul{\Z_p} \oplus \Sigma^{-1} \rH \ul{\coker(1-F)}
\]
to conclude.
\end{exm}

From now on,
for a $\Z/2$-spectrum $X$,
we will also use the grading convention
\[
\pi_{s,w}^{\Z/2}(X)
:=
\pi_{s-w+w\sigma}^{\Z/2}(X).
\]

\begin{thm}
\label{TCR.6}
Let $k$ be a perfect field of characteristic $2$.
Then for every integer $e\geq 1$, we have an isomorphism of abelian groups
\[
\pi_{s,w}^{\Z/2}\TCR(k[x]/x^e;\Z_2) 
\cong
\bigoplus_{i=0}^s E_{i,s-i,w}^2,
\]
where
\begin{align*}
&E_{*,*,*}^2
\cong 
(C[\tau^2]
\oplus
\rho D[\tau^2,\rho]
\oplus
\tau D'[\tau^2,\rho]
\oplus
\tfrac{\gamma}{\tau} C[\tfrac{1}{\tau^2}]
\oplus
\tfrac{\gamma}{\tau^2}D [\tfrac{1}{\tau^2},\tfrac{1}{\rho}] \oplus \tfrac{\gamma}{\tau \rho} D'[\tfrac{1}{\tau^2},\tfrac{1}{\rho}])\{y\}
\\
&\oplus 
\Z_2[\tau^2] \oplus \rho \F_2[\tau^2,\rho]
\oplus
\tfrac{\gamma}{\tau} \Z_2[\tfrac{1}{\tau^2}] \oplus \tfrac{\gamma}{\tau^2}\F_2 [\tfrac{1}{\tau^2},\tfrac{1}{\rho}]
\\
\oplus
&
\bigoplus_{t=1}^\infty
(
A_t[\tau^2]
\oplus
\rho B_t[\tau^2,\rho]
\oplus
\tau B_t'[\tau^2,\rho]
\oplus
\tfrac{\gamma}{\tau} A_t[\tfrac{1}{\tau^2}]
\oplus
\tfrac{\gamma}{\tau^2}B_t [\tfrac{1}{\tau^2},\tfrac{1}{\rho}] \oplus \tfrac{\gamma}{\tau \rho} B_t'[\tfrac{1}{\tau^2},\tfrac{1}{\rho}]
)
\{x_t\},
\end{align*}
$\lvert \tau\rvert =(0,0,-1)$,
$\lvert \rho \rvert = (-1,0,-1)$,
$\lvert \gamma \rvert = (0,0,1)$,
$\lvert x_t\rvert = (t-1,t,t)$,
$\lvert y \rvert = (-1,0,0)$,
$A_t:=\W_{et}(k)/V_e\W_{t}(k)$, $\W_{m}$ denotes the ring of $m$-truncated big Witt vectors for an integer $m$.
$V_e$ is the $e$th Verschiebung operator,
$B_t:=\coker(2\colon A_t\to A_t)$,
$B_t':=\ker(2\colon A_t\to A_t)$,
$C:=\coker(1-F)$ with the Frobenius $F$ on $W(k)$,
$D:=\coker(2\colon C\to C)$,
and $D':=\ker(2\colon C\to C)$.
\end{thm}
\begin{proof}
Step 1. \emph{Spectral sequence.}
The filtration in Theorem \ref{TCR.1} yields an $RO(\Z/2)$-graded multiplicative spectral sequence
\[
E_{s,t,w}^2
=
\pi_{s+t,w}^{\Z/2}\gr^t \TCR(k[x]/x^e;\Z_2)
\Rightarrow
\pi_{s+t,w}^{\Z/2}\TCR(k[x]/x^e;\Z_2)
\]
with the differentials $d_r\colon E_{s,t,w}^r \to E_{s-r,t+r-1,w}^r$.
Let us compute $E_{*,*,*}^2$.

Using the unit and augmentation maps $k\to k[x]/x^e\to k$,
we see that $\TCR(k;\Z_2)$ is a direct summand of $\TCR(k[x]/x^e;\Z_2)$.
Also, the degree $0$ part of $\widehat{\prism}_{k[x]/x^e}$ in \cite[\S 3]{MR4632370} shows that $\Z_2(0)(k[x]/x^e)$ is independent of $e$.
Together with Example \ref{TCR.5},
we can write
\[
E_{*,*,*}^*
\cong
E_{*,0,*}^*
\oplus
\overline{E}_{*,*,*}^*
\]
with
\[
\overline{E}_{s,t,w}^2 \Rightarrow \pi_{s+t,w}^{\Z/2}(\TCR(k[x]/x^e;\Z_2)/\TCR(k;\Z_2)).
\]

Sulyma's result \cite[Theorem 1.1]{MR4632370},
generalizing Mathew's result \cite[Theorem 10.4]{M21},
yields an equivalence of spectra
\[
\Z_2(t)(k[x]/x^e)
\simeq
\Sigma^{-1}\rH(\W_{et}(k)/V_e\W_{t}(k))
\]
for every integer $t\geq 1$.
If we set $A_t:=\W_{et}(k)/V_e\W_{t}(k)$ for simplicity of notation,
then
\[
\overline{E}_{s,t,w}^2\cong
\left\{
\begin{array}{ll}
\pi_{s-t+1,w}^{\Z/2}(\rH \ul{A_t}) & \text{if }t\geq 1,
\\
0 & \text{otherwise}.
\end{array}
\right.
\]

After applying $\otimes_{\Z_2}W(k)$ to the computation in \cite[p.\ 642]{MR4669153},
we obtain an isomorphism of graded abelian groups
\begin{equation}
\label{TCR.6.1}
\pi_{*,*}^{\Z/2}(\rH \ul{W(k)})
\cong
W(k)[\tau^2] \oplus \rho k[\tau^2,\rho]
\oplus
\tfrac{\gamma}{\tau} W(k)[\tfrac{1}{\tau^2}] \oplus \tfrac{\gamma}{\tau^2}k [\tfrac{1}{\tau^2},\tfrac{1}{\rho}],
\end{equation}
where
$\lvert \tau\rvert =(0,-1)$,
$\lvert \rho \rvert = (-1,-1)$,
and $\lvert \gamma \rvert = (0,1)$.
Since $k\otimes_{W(k)}^L A_t\simeq [A_t\xrightarrow{2} A_t]$ with the right $A_t$ in degree $0$,
we obtain an isomorphism of graded abelian groups
\[
\pi_{*,*}^{\Z/2}(\rH \ul{A_t})
\cong
A_t[\tau^2]
\oplus
\rho B_t[\tau^2,\rho]
\oplus
\tau B_t'[\tau^2,\rho]
\oplus
\tfrac{\gamma}{\tau} A_t[\tfrac{1}{\tau^2}]
\oplus
\tfrac{\gamma}{\tau^2}B_t [\tfrac{1}{\tau^2},\tfrac{1}{\rho}] \oplus \tfrac{\gamma}{\tau \rho} B_t'[\tfrac{1}{\tau^2},\tfrac{1}{\rho}],
\]
where $B_t:=\coker(2\colon A_t\to A_t)$,
and $B_t':=\ker(2\colon A_t\to A_t)$.
See Figure \ref{figure1} for a description of $\pi_{*,*}^{\Z/2}(\rH \ul{A_t})$.
\begin{figure}
\begin{tikzpicture}
\draw[red,->] (0,0)--(-1.5,-0.5);
\draw[shift = {(0,-0.5)},red,->] (0,0)--(-1.5,-0.5);
\draw[shift = {(0,-1)},red,->] (0,0)--(-1.5,-0.5);
\draw[shift = {(0,-1.5)},red,->] (0,0)--(-1.5,-0.5);
\draw[shift = {(0,-2)},red,->] (0,0)--(-1.5,-0.5);
\draw[shift = {(0,1)},red,->] (0,0)--(-1.5,-0.5);
\draw[shift = {(0,1.5)},red,->] (0,0)--(-1.5,-0.5);
\draw[shift = {(0,2)},red,->] (0,0)--(-1.5,-0.5);
\draw[gray, thick,->] (-2.5,0) -- (2.5,0);
\draw[gray, thick,->] (0,-2.5) -- (0,2.5);
\draw node at (2.4,-0.3) {$s$};
\draw node at (-0.35,2.4) {$w$};
\draw node at (0,0) {$\Box$};
\draw node at (0,1) {$\Box$};
\draw node at (0,2) {$\Box$};
\draw node at (0,-1) {$\Box$};
\draw node at (0,-2) {$\Box$};
\draw node at (0,-2) {$\Box$};
\draw node at (0,-0.5) {$\ocircle$};
\draw node at (-0.5,-1) {$\ocircle$};
\draw node at (0,-1.5) {$\ocircle$};
\draw node at (-1,-1.5) {$\ocircle$};
\draw node at (-0.5,-2) {$\ocircle$};
\draw node at (-1.5,-2) {$\ocircle$};
\draw node at (0.5,1.5) {$\ocircle$};
\draw node at (1,2) {$\ocircle$};
\draw node at (-0.5,-0.5) {$\bullet$};
\draw node at (-1,-1) {$\bullet$};
\draw node at (-0.5,-1.5) {$\bullet$};
\draw node at (-1.5,-1.5) {$\bullet$};
\draw node at (-1,-2) {$\bullet$};
\draw node at (-2,-2) {$\bullet$};
\draw node at (0,1.5) {$\bullet$};
\draw node at (0.5,2) {$\bullet$};
\end{tikzpicture}
\caption{\label{figure1}
Descriptions of $\pi_{s,w}^{\Z/2}(\rH \ul{A_t})$ and $d_2$, where $\Box=A_t$, $\bullet=B_t$, and $\ocircle = B_t'$.}
\end{figure}
Hence as graded abelian groups,
$\ol{E}_{*,*,*}^2$ is isomorphic to
\[
\bigoplus_{t=1}^\infty (A_t[\tau^2] \oplus   \rho B_t [\tau^2,\rho]
\oplus \tau B_t' [\tau^2,\rho]
\oplus \tfrac{\gamma}{\tau} A_t[\tfrac{1}{\tau^2}] \oplus \tfrac{\gamma}{\tau^2} B_t [\tfrac{1}{\tau^2},\tfrac{1}{\rho}] \oplus \tfrac{\gamma}{\tau \rho} B_t'[\tfrac{1}{\tau^2},\tfrac{1}{\rho}])\{x_t\},
\]
where $\lvert \tau\rvert =(0,0,-1)$,
$\lvert \rho \rvert = (-1,0,-1)$,
$\lvert \gamma \rvert = (0,0,1)$,
and $\lvert x_t\rvert = (t-1,t,t)$.

By a similar argument,
we also obtain an isomorphism of graded abelian groups
\begin{align*}
E_{*,0,*}^2
\cong &
(C[\tau^2]
\oplus
\rho D[\tau^2,\rho]
\oplus
\tau D'[\tau^2,\rho]
\oplus
\tfrac{\gamma}{\tau} C[\tfrac{1}{\tau^2}]
\oplus
\tfrac{\gamma}{\tau^2}D [\tfrac{1}{\tau^2},\tfrac{1}{\rho}] \oplus \tfrac{\gamma}{\tau \rho} D'[\tfrac{1}{\tau^2},\tfrac{1}{\rho}])\{y\}
\\
\oplus &
\Z_2[\tau^2] \oplus \rho \F_2[\tau^2,\rho]
\oplus
\tfrac{\gamma}{\tau} \Z_2[\tfrac{1}{\tau^2}] \oplus \tfrac{\gamma}{\tau^2}\F_2 [\tfrac{1}{\tau^2},\tfrac{1}{\rho}],
\end{align*}
where $\lvert y\rvert = (-1,0,0)$,
$C:=\coker(1-F)$ with the Frobenius $F$ on $W(k)$,
$D:=\coker(2\colon C\to C)$,
and $D':=\ker(2\colon C\to C)$.

Step 2. \emph{Degeneracy.}
For the case of $W(k)=\Z_2$, \eqref{TCR.6.1} becomes
\[
\pi_{*.*}^{\Z/2}(\rH \ul{\Z_2})
\cong
\Z_2[\tau^2] \oplus \rho \F_2[\tau^2,\rho] \oplus \tfrac{\gamma}{\tau} \Z_2[\tfrac{1}{\tau^2}] \oplus \tfrac{\gamma}{\tau^2} \F_2 [\tfrac{1}{\tau^2},\tfrac{1}{\rho}].
\]
See \cite[Definition 3.3, Proposition 3.4]{MR4669153} for the multiplicative relations on the right-hand side.
Among them, we will only need the relations generated by
\begin{gather*}
\tau^2 (\tau^{m} \rho^n) = \tau^{m+2}\rho^n,
\;
\rho (\tau^{m} \rho^{n}) = \tau^{m} \rho^{n+1},
\\
\tau^2 (\tfrac{\gamma}{\tau^{m+3}\rho^n}) = \tfrac{\gamma}{\tau^{m+1}\rho^n},
\;
\rho (\tfrac{\gamma}{\tau^{m+1}\rho^{n+1}}) = \tfrac{\gamma}{\tau^{m+1}\rho^n}
\end{gather*}
for all $m,n\geq 0$ with even $m$.
By Example \ref{TCR.5},
$\TCR(k;\Z_2)$ is an $\rH \ul{\Z_2}$-algebra.
It follows that 
$\TCR(k[x]/x^e;\Z_2)$ is an $\rH \ul{\Z_2}$-algebra too.
In particular,
$\tau^2$ and $\rho$ act on the computation of $\ol{E}_{*,*,*}^2$ in Step 1 with the above relations allowing odd $m$.

We claim that $\overline{E}$ degenerates at $\overline{E}^2$.
Let $r\geq 2$, and assume $r=2$ or $d_s=0$ for all $s<r$.
Note that $\overline{E}_{s,t,w}^r$ does not vanish if and only if $(a,0,c):=(s,t,w)-\lvert x_t\rvert = (s-t+1,0,w-t)$ satisfies
\[
\text{
($a,c\geq 0$ and $c\geq a+2$)
or
($a,c\leq 0$ and $a\geq c$).
}
\]
See Figure \ref{figure1} for a description of $d_2$.
We have
\[
\lvert d_r(x_t)\rvert -\lvert x_{t+r-1}\rvert= (-2r+1,0,-r+1),
\]
which implies $d_r(x_t)=0$ and $d_r(\tau x_t)=0$ by the degree consideration.
It follows that $d_r(\tau^m\rho^n x_t)=0$ for all integers $m,n\geq 0$ by the Leibniz rule.

For all integers $m,n\geq 1$,
we have
\[
\lvert d_r(\tfrac{\gamma}{\tau^m \rho^n}x_t)\rvert = \lvert \tfrac{\gamma}{\tau^{m+r}\rho^{n-2r-1}} x_{t+r-1}\rvert.
\]
If $n<2r+1$, then interpret this as $d_r(\tfrac{\gamma}{\tau^m \rho^n}x_t)=0$.
Assume $n\geq 2r+1$.
If $m$ is even,
then
\[
0 = d_r(0) = d_r(\tau^m \tfrac{\gamma}{\tau^m \rho^n} x_t)
=
d_r(\tau^m) \tfrac{\gamma}{\tau^m \rho^n} x_t + \tau^m d_r(\tfrac{\gamma}{\tau^m \rho^n} x_t)
=
\tau^m d_r(\tfrac{\gamma}{\tau^m \rho^n} x_t).
\]
Since the multiplication $\tau^m$ from the coefficient of $\tfrac{\gamma}{\tau^{m+r}\rho^{n-2r-1}} x_{t+r-1}$ to the coefficient of $\tfrac{\gamma}{\tau^{r}\rho^{n-2r-1}} x_{t+r-1}$ is an isomorphism,
we have $d_r(\tfrac{\gamma}{\tau^m \rho^n} x_t)=0$.
If $m$ is odd,
then
\[
0 = d_r(\tau^{m+1} \tfrac{\gamma}{\tau^m \rho^n} x_t)
=
\tau^{m+1} d_r(\tfrac{\gamma}{\tau^m \rho^n} x_t).
\]
We similarly deduce $d_r(\tfrac{\gamma}{\tau^m \rho^n} x_t)=0$.

We have shown $d_r=0$.
Hence
$\ol{E}$ degenerates at $\overline{E}^2$ by induction.

Step 3. \emph{Convergence.}
The spectral sequence $E$ is a half-plane spectral sequence.
Since the connectivity of the filtration on $\TCR(k[x]/x^e;\Z_2)$ goes to $\infty$,
$E$ conditionally converges in the sense of \cite[Definition 5.10]{MR1718076}.
Since $E$ degenerates at $E_2$,
$E$ strongly converges by \cite[Theorem 7.1 and the following Remark]{MR1718076}.

Step 4. \emph{$\F_2$-spectral sequence}.
Let us discuss the $\F_2$-version of $\ol{E}$,
which will be useful later when solving extension problems.
Consider the induced spectral sequence
\[
\ol{E}_{s,t,w}'^2
=
\pi_{s+t,w}^{\Z/2}(\gr^t \TCR(k[x]/x^e;\F_2))
\Rightarrow
\pi_{s+t,w}^{\Z/2}(\TCR(k[x]/x^e;\F_2)/\TCR(k;\F_2)),
\]
where $\TCR(R;\F_2):=\TCR(R;\Z_2)/2$ for every commutative ring $R$. Note that
$\TCR(k[x]/x^e;\F_2)\simeq \TCR(k[x]/x^e;\Z_2)\otimes_{\Sphere}\Sphere/2$ is an $\rH\ul{\Z_2}\otimes_{\Sphere} \Sphere/2\simeq \rH\ul{\F_2}$-algebra.
Hence for every integer $w$,
\[
(\Sigma^{-w\sigma} (\TCR(k[x]/x^e;\F_2)/\TCR(k;\F_2)))^{\Z/2}
\]
is an $\rH \F_2$-module,
so
the induced filtration on $\pi_{s+t,w}^{\Z/2}(\TCR(k[x]/x^e;\F_2)/\TCR(k;\F_2))$ splits.
Also, note that $\ol{E}_{s,t,w}'^2$ is multiplicative.

By \cite[Proposition 6.2]{zbMATH01587571},
we have an isomorphism of graded abelian groups
\[
\pi_{*,*}^{\Z/2}(\rH \ul{\F_2})
\cong
\F_2[\tau,\rho] \oplus \tfrac{\gamma}{\tau} \F_2[\tfrac{1}{\tau},\tfrac{1}{\rho}].
\]
Hence as graded abelian groups,
$\ol{E}_{*,*,*}'^2$ is isomorphic to
\[
\bigoplus_{t=1}^\infty
((B_t [\tau,\rho] \oplus \tfrac{\gamma}{\tau} B_t[\tfrac{1}{\tau},\tfrac{1}{\rho}])\{x_t\} \oplus (B_t'[\tau,\rho]\oplus \tfrac{\gamma}{\tau} B_t'[\tfrac{1}{\tau},\tfrac{1}{\rho}])\{x_t'\}
),
\]
where $\lvert x_t'\rvert = (t,t,t)$.

We claim that $\ol{E}'$ degenerates at $\ol{E}'^2$.
Let $r\geq 2$, and assume $r=2$ or $d_s=0$ for all $s<r$.
Note that $\overline{E}_{s,t,w}'^r$ does not vanish if and only if $(a,0,c):=(s,t,w)-\lvert x_t\rvert = (s-t+1,0,w-t)$ satisfies
\[
\text{
($a,c\geq 0$ and $c\geq a+1$)
or
($a\leq 1$, $c\leq 0$, and $a\geq c$).
}
\]
See Figure \ref{figure3} for descriptions of $\ol{E}_{s,t,w}'^2$ and $d_2$.
\begin{figure}
\begin{tikzpicture}
\draw[red,->] (0,0)--(-1.5,-0.5);
\draw[shift = {(0,-0.5)},red,->] (0,0)--(-1.5,-0.5);
\draw[shift = {(0,-1)},red,->] (0,0)--(-1.5,-0.5);
\draw[shift = {(0,-1.5)},red,->] (0,0)--(-1.5,-0.5);
\draw[shift = {(0,-2)},red,->] (0,0)--(-1.5,-0.5);
\draw[shift = {(0,1)},red,->] (0,0)--(-1.5,-0.5);
\draw[shift = {(0,1.5)},red,->] (0,0)--(-1.5,-0.5);
\draw[shift = {(0,2)},red,->] (0,0)--(-1.5,-0.5);
\draw[gray, thick,->] (-2.5,0) -- (2.5,0);
\draw[gray, thick,->] (0,-2.5) -- (0,2.5);
\draw node at (2.4,-0.3) {$s$};
\draw node at (-0.35,2.4) {$w$};
\draw node at (0,0) {$\bullet$};
\draw node at (-0.5,-0.5) {$\bullet$};
\draw node at (-1,-1) {$\bullet$};
\draw node at (-1.5,-1.5) {$\bullet$};
\draw node at (-2,-2) {$\bullet$};
\draw node at (0,1) {$\bullet$};
\draw node at (0,1.5) {$\bullet$};
\draw node at (0,2) {$\bullet$};
\draw node at (0,-0.5) {$\odot$};
\draw node at (0,-1) {$\odot$};
\draw node at (-0.5,-1) {$\odot$};
\draw node at (0,-1.5) {$\odot$};
\draw node at (-0.5,-1.5) {$\odot$};
\draw node at (-1,-1.5) {$\odot$};
\draw node at (0,-2) {$\odot$};
\draw node at (-0.5,-2) {$\odot$};
\draw node at (-1,-2) {$\odot$};
\draw node at (-1.5,-2) {$\odot$};
\draw node at (0.5,1.5) {$\odot$};
\draw node at (0.5,2) {$\odot$};
\draw node at (1,2) {$\odot$};
\draw node at (0.5,0) {$\ocircle$};
\draw node at (0.5,-0.5) {$\ocircle$};
\draw node at (0.5,-1) {$\ocircle$};
\draw node at (0.5,-1.5) {$\ocircle$};
\draw node at (0.5,-2) {$\ocircle$};
\draw node at (0.5,1) {$\ocircle$};
\draw node at (1,1.5) {$\ocircle$};
\draw node at (1.5,2) {$\ocircle$};
\end{tikzpicture}
\caption{\label{figure3}
Descriptions of $\ol{E}_{s,t,w}^2$ and $d_2$, where $\bullet=B_t$, $\ocircle = B_t'$, and $\odot = B_t\oplus B_t'$.}
\end{figure}
As in Steps 2 and 3,
using the actions of $\tau^m \rho^n$ for all integers $m,n\geq 0$ and $\tfrac{\gamma}{\tau^m \rho^n}$ for all integers $m\geq 1$ and $n\geq 0$,
we can show that $\ol{E}'$ degenerates at $\ol{E}'^2$ and strongly converges.

Step 5. \emph{Extension problems, part I.}
The spectral sequence $\overline{E}$ yields a filtration on 
\[
H_{u,w}:=\pi_{u,w}^{\Z/2}(\TCR(k[x]/x^e;\Z_2)/\TCR(k;\Z_2))
\]
whose $s$th graded piece is $E_{s,u-s,w}^2$.
To complete the proof,
it remains to show that the filtration on $H_{u,w}$ splits.
The only nontrivial parts are
\begin{gather*}
\Fil^{t-a-1}H_{2t-1,t-a}\to \cdots \to \Fil^{t-1} H_{2t-1,t-a},
\\
\Fil^{t-a-2}H_{2t-2,t-a-1} \to \cdots \to \Fil^{t-2}H_{2t-2,t-a-1},
\\
\Fil^{t-1}H_{2t-1,t+a+2} \to \cdots \to \Fil^{t+a-1}H_{2t-1,t+a+2},
\\
\Fil^{t}H_{2t,t+a+3} \to \cdots \to \Fil^{t+a}H_{2t,t+a+3}
\end{gather*}
with $a\geq 0$,
which mean that the other parts are zero or remain constant.
See Figure \ref{figure2} for a description of the graded pieces of $H_{u,w}$.
\begin{figure}
\begin{tikzpicture}[scale = 2.2]
\draw node at (0,0) {$F_{0,0}^t$};
\draw node at (0,1) {$F_{0,2}^t$};
\draw node at (0,2) {$F_{0,4}^t$};
\draw node at (0,-1) {$F_{0,-2}^t$};
\draw node at (0,-2) {$F_{0,-4}^t$};
\draw node at (0,-0.5) {$F_{0,-1}^t$};
\draw node at (-0.5,-1) {$F_{-1,-2}^t$};
\draw node at (0,-1.5) {$F_{0,-3}^t$};
\draw node at (-1,-1.5) {$F_{-2,-3}^{t+1}$};
\draw node at (-0.5,-2) {$F_{-1,-4}^t$};
\draw node at (-1.5,-2) {$F_{-3,-4}^{t+1}$};
\draw node at (0.5,1.5) {$F_{1,3}^t$};
\draw node at (1,2) {$F_{2,4}^{t-1}$};
\draw node at (-0.5,-0.5) {$F_{-1,-1}^t$};
\draw node at (-1,-1) {$F_{-2,-2}^{t+1}$};
\draw node at (-0.5,-1.5) {$F_{-1,-3}^t$};
\draw node at (-1.5,-1.5) {$F_{-3,-3}^{t+1}$};
\draw node at (-1,-2) {$F_{-2,-4}^{t+1}$};
\draw node at (-2,-2) {$F_{-4,-4}^{t+2}$};
\draw node at (0,1.5) {$F_{0,3}^t$};
\draw node at (0.5,2) {$F_{1,3}^t$};
\draw[blue] (-2.5,-1.25)--(1.5,0.75);
\draw[shift={(0,-0.25)},blue] (-2.5,-1.25)--(1.5,0.75);
\draw[shift={(0,-0.5)},blue] (-2.5,-1.25)--(1.5,0.75);
\draw[shift={(0,-0.75)},blue] (-2.5,-1.25)--(1.5,0.75);
\draw[shift={(0,-1)},blue] (-2.5,-1.25)--(1.5,0.75);
\draw[shift={(0,-1.25)},blue] (-2,-1)--(1.5,0.75);
\draw[shift={(0,-1.5)},blue] (-1.5,-0.75)--(1.5,0.75);
\draw[shift={(0,-1.75)},blue] (-1,-0.5)--(1.5,0.75);
\draw[shift={(0,-2)},blue] (-0.5,-0.25)--(1.5,0.75);
\draw[shift={(0,1)},blue] (-2.5,-1.25)--(1.5,0.75);
\draw[shift={(0,1.25)},blue] (-2.5,-1.25)--(1.5,0.75);
\draw[shift={(0,1.5)},blue] (-2.5,-1.25)--(1.5,0.75);
\draw[shift={(0,1.75)},blue] (-2.5,-1.25)--(1,0.5);
\draw[shift={(0,2)},blue] (-2.5,-1.25)--(0.5,0.25);
\draw node [right] at (1.5,0.75) {$H_{2t-1,t}$};
\draw node [right] at (1.5,0.5) {$H_{2t-2,t-1}$};
\draw node [right] at (1.5,0.25) {$H_{2t-1,t-1}$};
\draw node [right] at (1.5,0) {$H_{2t-2,t-2}$};
\draw node [right] at (1.5,-0.25) {$H_{2t-1,t-2}$};
\draw node [right] at (1.5,-0.5) {$H_{2t-2,t-3}$};
\draw node [right] at (1.5,-0.75) {$H_{2t-1,t-3}$};
\draw node [right] at (1.5,-1) {$H_{2t-2,t-4}$};
\draw node [right] at (1.5,-1.25) {$H_{2t-1,t-4}$};
\draw node [left] at (-2.5,-0.25) {$H_{2t-1,t+2}$};
\draw node [left] at (-2.5,0) {$H_{2t,t+3}$};
\draw node [left] at (-2.5,0.25) {$H_{2t-1,t+3}$};
\draw node [left] at (-2.5,0.5) {$H_{2t,t+4}$};
\draw node [left] at (-2.5,0.75) {$H_{2t-1,t+4}$};
\end{tikzpicture}
\caption{\label{figure2}
The objects passing through the blue lines describe the graded pieces of $H_{u,w}$,
where $F_{a,c}^{t}:=E_{t-1+a,t,t+c}^2$ for abbreviation.
}
\end{figure}

Step 6. \emph{Extension problems, part II.}
We claim that the filtrations on $H_{2t-1,t-a}$ and $H_{2t-2,t-a-1}$ split.
We proceed by induction on $a$.
The claim is trivial if $a=0$.
Assume $a>0$.
Since the multiplications
\[
\Fil^{t-2} H_{2t-2,t-a+1}
\xrightarrow{\rho}
\Fil^{t-2} H_{2t-1,t-a}
\xrightarrow{\rho}
\Fil^{t-3} H_{2t-2,t-a-1},
\]
are isomorphisms,
by induction,
the filtrations on $H_{2t-1,t-a}$ and $H_{2t-2,t-a-1}$ split if we obtain sections
\begin{equation}
\label{TCR.6.2}
\gr^{t-1} H_{2t-1,t-a}\to H_{2t-1,t-a},
\;
\gr^{t-2} H_{2t-2,t-a-1}\to H_{2t-2,t-a-1}.
\end{equation}

Assume $a=1$.
The multiplication $\rho\colon H_{2t-1,t-1}\to H_{2t-2,t-2}$ is an isomorphism,
so it suffices to show that the filtration on $H_{2t-1,t-1}$ splits.
Consider
\(
H_{u,w}':=\pi_{u,w}^{\Z/2}(\TCR(k[x]/x^e;\F_2)/\TCR(k;\F_2)).
\)
Note that the filtration on $H_{u,w}'$ splits as observed in Step 4.
The induced map $H_{2t-1,t-1}\to H_{2t-1,t-1}'$ yields a splitting for $\gr^{t-2}H_{2t-1,t-1}$ since
\[
\gr^{t-2}H_{2t-1,t-1}
\to
\gr^{t-2}H_{2t-1,t-1}'
\]
can be identified with the identity map $B_{t+1}\to B_{t+1}$.

Assume $a>1$.
Then we have sections
\[
\gr^{t-1}H_{2t-1,t-a+2}\to H_{2t-1,t-a+2},
\;
\gr^{t-2}H_{2t-2,t-a+1}\to H_{2t-2,t-a+1}
\]
by induction.
Compose these with
\[
\tau^2\colon H_{2t-1,t-a+2}\to H_{2t-1,t-a},
\;
\tau^2\colon H_{2t-2,t-a+1}\to H_{2t-2,t-a-1}
\]
and use the induced isomorphisms
\[
\gr^{t-1} H_{2t-1,t-a+2}\cong \gr^{t-1} H_{2t-1,t-a},
\;
\gr^{t-2} H_{2t-2,t-a+1}\cong \gr^{t-2} H_{2t-2,t-a-1}
\]
to yield \eqref{TCR.6.2}.

Step 7. \emph{Extension problems, part III.}
We claim that the filtrations on $H_{2t-1,t+a+2}$ and $H_{2t,t+a+3}$ split.
We proceed by induction on $a$.
The claim is trivial if $a=0$.
Assume $a>0$.
Since the multiplications
\[
H_{2t,t+a+3} / \Fil^{t} H_{2t,t+a+3}
\xrightarrow{\rho} 
 H_{2t-1,t+a+2} / \Fil^{t-1} H_{2t-1,t+a+2}
\xrightarrow{\rho} 
 H_{2t,t+a+2}
\]
are isomorphisms,
by induction,
the filtrations on $H_{2t-1,t+a+2}$ and $H_{2t,t+a+3}$ split if we obtain retractions
\begin{equation}
\label{TCR.6.3}
H_{2t-1,t+a+2}\to \gr^{t-1} H_{2t-1,t+a+2},
\;
H_{2t,t+a+3}\to \gr^t H_{2t,t+a+3}.
\end{equation}

Assume $a=1$.
The multiplication $\rho\colon H_{2t,t+4}\to H_{2t-1,t+3}$ is an isomorphism,
so it suffices to show that the filtration on $H_{2t-1,t+3}$ splits.
For this,
observe that the induced map $H_{2t-1,t+3}\to H_{2t-1,t+3}'$ yields a splitting for $\gr^{t-1} H_{2t-1,t+3}$ since
\[
\gr^{t-1} H_{2t-1,t+3}\to \gr^{t-1} H_{2t-1,t+3}'
\]
can be identified with the identity map $B_t\to B_t$.

Assume $a>1$.
Then we have retractions
\[
H_{2t-1,t+a}\to \gr^{t-1} H_{2t-1,t+a},
\;
H_{2t,t+a+1}\to \gr^t H_{2t,t+a+1}
\]
by induction.
Compose these with
\[
\tau^2\colon H_{2t-1,t+a+2}\to H_{2t-1,t+a},
\;
\tau^2\colon H_{2t,t+a+3}\to H_{2t,t+a+1}
\]
and use the induced isomorphisms
\[
\gr^{t-1} H_{2t-1,t+a+2}\cong \gr^{t-1} H_{2t-1,t+a},
\;
\gr^{t} H_{2t,t+a+3}\cong \gr^t  H_{2t,t+a+1}
\]
to yield \eqref{TCR.6.3}.
\end{proof}

\begin{exm}
\label{TCR.3}
Let $K$ be a finite extension of $\Q_p$,
let $\cO_K$ be its ring of integers,
let $\omega$ be a uniformizer,
and we set $f:=[\cO_K/\omega:\F_2]$.
Antieau, Krause, and Nikolaus announced the following result in \cite[\S 2]{2204.03420} and proved it in \cite{AKN}:
For all integers $i,n\geq 1$,
there exists a quasi-isomorphism of chain complexes (using the Dold--Kan correspondence)
\[
\Z_p(i)(\cO_K/\omega^n)
\simeq
(\Z_p^{f(in-1)}\to \Z_p^{2f(in-1)}\to \Z_p^{f(in-1)})
\]
such that the left $\Z_p^{f(in-1)}$ sits in homological degree $0$ satisfying the following properties.
\begin{enumerate}
\item[(1)] There is an algorithm to determine the latter complex.
\item[(2)] $\Z_p(0)(\cO_K/\omega^n)$ is concentrated in homological degrees $0$ and $-1$.
\item[(3)] $\Z_p(i)(\cO_K/\omega^n)$ is concentrated in homological degrees $-1$ and $-2$ for $i\geq 1$.
\item[(4)]
\(
H^1(\Z_p(i)(\cO_K/\omega^n))
\cong
\pi_{2i-1}\TC(\cO_K/\omega^n)
\) for $i\geq 1$.
\item[(5)]
\(
H^2(\Z_p(1)(\cO_K/\omega^n))
=0.
\)
\item[(6)]
\(
H^2(\Z_p(i)(\cO_K/\omega^n))
\cong
\pi_{2i-2}\TC(\cO_K/\omega^n)
\) for $i\geq 2$.
\end{enumerate}

Let us assume this. Then $\Sigma^2\Z_p(i)(\cO_K/\omega^n)$ is concentrated in homological degrees $0$ and $1$,
so we have
\[
\rho_j(\Sigma^{i+2+i\sigma} \Z_p(i)(\cO_K/\omega^n))
\cong
\left\{
\begin{array}{ll}
\underline{\pi_{2i-2}\TC(\cO_K/\omega^n)}&
\text{if $j=2i$ and $i\geq 2$},
\\
\underline{\pi_{2i-1}\TC(\cO_K/\omega^n)}&
\text{if $j=2i+1$ and $i\geq 1$},
\\
0 &
\text{otherwise}
\end{array}
\right.
\]
for $i\geq 1$.
Together with Theorem \ref{TCR.1},
we have
\[
\rho_j (\Sigma^{2}\tau_{\geq 1}\TCR(\cO_K/\omega^n;\Z_p))
\cong
\left\{
\begin{array}{ll}
\ul{\pi_{j-2}\TC(\cO_K/\omega^n;\Z_p)}&
\text{if $j\geq 3$},
\\
0& \text{otherwise}.
\end{array}
\right.
\]
\end{exm}

\begin{exm}
Here, we assume $p\neq 2$.
As a consequence of  \cite[Theorem 2.8(a)]{Dug05},
for every $\Z_p$-module $M$,
we have
\[
\pi_{s,w}^{\Z/2}(\rH \ul{M})
\cong
\left\{
\begin{array}{ll}
M & \text{if $s=0$ and $w$ is even},
\\
0 & \text{otherwise.}
\end{array}
\right.
\]
Hence
we have
\[
\pi_{s,w}^{\Z/2}(\gr^n\TCR(A;\Z_p))
\cong
\left\{
\begin{array}{ll}
\pi_{s-2n} \Z_p(n)(A) & \text{if $w-n$ is even},
\\
0 & \text{otherwise}
\end{array}
\right.
\]
for every quasisyntomic ring $A$.
Now,
assume that $A$ is one of $k[x]/x^e$ with a perfect field $k$, $\cO_K/\omega^n$, or $\cO_K$,
see Example \ref{TCR.3} for the notation $\cO_K/\omega^n$.
By \cite[Theorem 1.1]{MR4632370} and \cite[Corollary 2.13, Remark 2.17]{AKN},
(3), (4), and (6) in Example \ref{TCR.3} hold for $\Z_p(i)(A)$.
Hence for $s\geq 1$,
we obtain
\[
\pi_{s,w}^{\Z/2}\TCR(A;\Z_p)
\cong
\left\{
\begin{array}{ll}
\pi_s \TC(A;\Z_p)& \text{if $s\equiv 2,3$ $(\mathrm{mod}\; 4)$ and $w$ is even},
\\
\pi_s \TC(A;\Z_p)& \text{if $s\equiv 0,1$ $(\mathrm{mod}\; 4)$ and $w$ is odd},
\\
0 & \text{otherwise}.
\end{array}
\right.
\]
\end{exm}

\begin{rmk}
If we have a real refinement of the theorem of Dundas-Goodwillie-McCarthy,
then we would have
\begin{gather*}
\KR(k[x]/x^e;\Z_p)
\simeq
\tau_{\geq 0}\TCR(k[x]/x^e;\Z_p),
\\
\KR(\cO_K/\omega^n;\Z_p)
\simeq
\tau_{\geq 0}\TCR(\cO_K/\omega^n;\Z_p).
\end{gather*}
Assuming this,
Theorem \ref{TCR.6} computes $\pi_{s,w}^{\Z/2}\KR(k[x]/x^e;\Z_2)$,
and Example \ref{TCR.3} computes $\rho_j(\Sigma^2\tau_{\geq 1}\KR(\cO_K/\omega^n;\Z_p))$.
\end{rmk}

\bibliography{SynTCR}

\begin{thebibliography}{10}

\bibitem{AKQ}
{\sc G.~Angelini-Knoll, H.~Jia~Kong, and J.~D. Quigley}, {\em Real syntomic cohomology}.
\newblock Preprint, \href{https://arxiv.org/pdf/2505.24734}{arXiv:2505.24734}, 2025.

\bibitem{2204.03420}
{\sc B.~Antieau, A.~Krause, and T.~Nikolaus}, {\em On the {K}-theory of $\mathbb{Z}/p^n$ -- announcement}.
\newblock Preprint, \href{https://arxiv.org/pdf/2204.03420}{arXiv:2204.03420}, 2022.

\bibitem{AKN}
\leavevmode\vrule height 2pt depth -1.6pt width 23pt, {\em On the {K}-theory of $\mathbb{Z}/p^n$}.
\newblock Preprint, \href{https://arxiv.org/pdf/2405.04329}{arXiv:2405.04329}, 2024.

\bibitem{BH21}
{\sc T.~Bachmann and M.~Hoyois}, {\em Norms in motivic homotopy theory}, Ast\'{e}risque,  (2021), pp.~ix+207.

\bibitem{Bar17}
{\sc C.~Barwick}, {\em Spectral {M}ackey functors and equivariant algebraic {$K$}-theory ({I})}, Adv. Math., 304 (2017), pp.~646--727.

\bibitem{1608.03657}
{\sc C.~Barwick, E.~Dotto, S.~Glasman, D.~Nardin, and J.~Shah}, {\em Parametrized higher category theory and higher algebra: {E}xposé {I} -- {E}lements of parametrized higher category theory}.
\newblock Preprint, \href{https://arxiv.org/pdf/1608.03657}{arXiv:1608.03657}, 2016.

\bibitem{BGS20}
{\sc C.~Barwick, S.~Glasman, and J.~Shah}, {\em Spectral {M}ackey functors and equivariant algebraic {$K$}-theory, {II}}, Tunis. J. Math., 2 (2020), pp.~97--146.

\bibitem{BMS19}
{\sc B.~Bhatt, M.~Morrow, and P.~Scholze}, {\em Topological {H}ochschild homology and integral {$p$}-adic {H}odge theory}, Publ. Math. Inst. Hautes \'{E}tudes Sci., 129 (2019), pp.~199--310.

\bibitem{BS22}
{\sc B.~Bhatt and P.~Scholze}, {\em Prisms and prismatic cohomology}, Ann. of Math. (2), 196 (2022), pp.~1135--1275.

\bibitem{MR1718076}
{\sc J.~M. Boardman}, {\em Conditionally convergent spectral sequences}, in Homotopy invariant algebraic structures ({B}altimore, {MD}, 1998), vol.~239 of Contemp. Math., Amer. Math. Soc., Providence, RI, 1999, pp.~49--84.

\bibitem{BHM93}
{\sc M.~B\"{o}kstedt, W.~C. Hsiang, and I.~Madsen}, {\em The cyclotomic trace and algebraic {$K$}-theory of spaces}, Invent. Math., 111 (1993), pp.~465--539.

\bibitem{MR4514986}
{\sc B.~Calm\`es, E.~Dotto, Y.~Harpaz, F.~Hebestreit, M.~Land, K.~Moi, D.~Nardin, T.~Nikolaus, and W.~Steimle}, {\em Hermitian {K}-theory for stable {$\infty$}-categories {I}: {F}oundations}, Selecta Math. (N.S.), 29 (2023), pp.~Paper No. 10, 269.

\bibitem{2009.07224}
\leavevmode\vrule height 2pt depth -1.6pt width 23pt, {\em Hermitian {K}-theory for stable $\infty$-categories {II}: {C}obordism categories and additivity}.
\newblock Preprint, \href{https://arxiv.org/pdf/2009.07224}{arXiv:2009.07224}, to appear in Acta Math., 2025.

\bibitem{CS24}
{\sc K.~{\v{C}}esnavi{\v{c}}ius and P.~Scholze}, {\em Purity for flat cohomology}, Ann. Math. (2), 199 (2024), pp.~51--180.

\bibitem{DMP}
{\sc E.~Dotto, K.~Moi, and I.~Patchkoria}, {\em On the geometric fixed points of real topological cyclic homology}, J. Lond. Math. Soc. (2), 109 (2024), pp.~Paper No. e12862, 68.

\bibitem{DMPR21}
{\sc E.~Dotto, K.~Moi, I.~Patchkoria, and S.~P. Reeh}, {\em Real topological {H}ochschild homology}, J. Eur. Math. Soc. (JEMS), 23 (2021), pp.~63--152.

\bibitem{Dug05}
{\sc D.~Dugger}, {\em An {A}tiyah-{H}irzebruch spectral sequence for {$KR$}-theory}, $K$-Theory, 35 (2005), pp.~213--256.

\bibitem{DGM13}
{\sc B.~I. Dundas, T.~G. Goodwillie, and R.~McCarthy}, {\em The local structure of algebraic {K}-theory}, vol.~18 of Algebra and Applications, Springer-Verlag London, Ltd., London, 2013.

\bibitem{GP18}
{\sc O.~Gwilliam and D.~Pavlov}, {\em Enhancing the filtered derived category}, J. Pure Appl. Algebra, 222 (2018), pp.~3621--3674.

\bibitem{MR1471886}
{\sc L.~Hesselholt and I.~Madsen}, {\em Cyclic polytopes and the {$K$}-theory of truncated polynomial algebras}, Invent. Math., 130 (1997), pp.~73--97.

\bibitem{MR1410465}
\leavevmode\vrule height 2pt depth -1.6pt width 23pt, {\em On the {$K$}-theory of finite algebras over {W}itt vectors of perfect fields}, Topology, 36 (1997), pp.~29--101.

\bibitem{MR1998478}
\leavevmode\vrule height 2pt depth -1.6pt width 23pt, {\em On the {$K$}-theory of local fields}, Ann. of Math. (2), 158 (2003), pp.~1--113.

\bibitem{Hil20}
{\sc M.~A. Hill}, {\em Equivariant stable homotopy theory}, in Handbook of homotopy theory, CRC Press/Chapman Hall Handb. Math. Ser., CRC Press, Boca Raton, FL, [2020] \copyright 2020, pp.~699--756.

\bibitem{HHR}
{\sc M.~A. Hill, M.~J. Hopkins, and D.~C. Ravenel}, {\em On the nonexistence of elements of {K}ervaire invariant one}, Ann. of Math. (2), 184 (2016), pp.~1--262.

\bibitem{HHR21}
\leavevmode\vrule height 2pt depth -1.6pt width 23pt, {\em Equivariant stable homotopy theory and the {K}ervaire invariant problem}, vol.~40 of New Mathematical Monographs, Cambridge University Press, Cambridge, 2021.

\bibitem{HM17}
{\sc M.~A. Hill and L.~Meier}, {\em The {$C_2$}-spectrum {${\rm Tmf}_1(3)$} and its invertible modules}, Algebr. Geom. Topol., 17 (2017), pp.~1953--2011.

\bibitem{HP23}
{\sc J.~Hornbostel and D.~Park}, {\em Real topological {H}ochschild homology of schemes}, J. Inst. Math. Jussieu, 23 (2024), pp.~1461--1518.

\bibitem{HP}
\leavevmode\vrule height 2pt depth -1.6pt width 23pt, {\em Real topological {Hochschild} homology of perfectoid rings}, J. Topol., 18 (2025).
\newblock Id/No e70032.

\bibitem{zbMATH01587571}
{\sc P.~Hu and I.~Kriz}, {\em Real-oriented homotopy theory and an analogue of the {Adams}-{Novikov} spectral sequence}, Topology, 40 (2001), pp.~317--399.

\bibitem{MR0491680}
{\sc L.~Illusie}, {\em Complexe cotangent et d\'{e}formations. {I}}, vol.~Vol. 239 of Lecture Notes in Mathematics, Springer-Verlag, Berlin-New York, 1971.

\bibitem{MR4669153}
{\sc H.~J. Kong}, {\em The {$C_2$}-effective spectral sequence for {$C_2$}-equivariant connective real {$K$}-theory}, Tunis. J. Math., 5 (2023), pp.~627--662.

\bibitem{MR4493328}
{\sc R.~Liu and G.~Wang}, {\em Topological cyclic homology of local fields}, Invent. Math., 230 (2022), pp.~851--932.

\bibitem{HTT}
{\sc J.~{Lurie}}, {\em {Higher topos theory}}, vol.~170, Princeton, NJ: Princeton University Press, 2009.

\bibitem{HA}
\leavevmode\vrule height 2pt depth -1.6pt width 23pt, {\em {Higher algebra}}.
\newblock \url{https://www.math.ias.edu/~lurie/}, 2017.

\bibitem{SAG}
{\sc J.~Lurie}, {\em Spectral algebraic geometry}.
\newblock \url{https://www.math.ias.edu/~lurie/}, 2018.

\bibitem{M21}
{\sc A.~Mathew}, {\em Some recent advances in topological {H}ochschild homology}, Bull. Lond. Math. Soc., 54 (2022), pp.~1--44.

\bibitem{NS}
{\sc T.~Nikolaus and P.~Scholze}, {\em On topological cyclic homology}, Acta Math., 221 (2018), pp.~203--409.

\bibitem{QS22}
{\sc J.~D. Quigley and J.~Shah}, {\em On the equivalence of two theories of real cyclotomic spectra}.
\newblock Preprint, \href{https://arxiv.org/pdf/2112.07462}{arXiv:2112.07462}, 2022.

\bibitem{QS21}
\leavevmode\vrule height 2pt depth -1.6pt width 23pt, {\em On the parameterized {T}ate construction}, J. Topol., 18 (2025).
\newblock Id/No e70018.

\bibitem{Qui}
{\sc D.~Quillen}, {\em Homology of commutative rings}.
\newblock \url{https://web.archive.org/web/20150420184538/http://chromotopy.org/paste/quillen.djvu}.

\bibitem{zbMATH05690542}
{\sc M.~Schlichting}, {\em Hermitian {{\(K\)}}-theory of exact categories}, J. \(K\)-Theory, 5 (2010), pp.~105--165.

\bibitem{2109.11954}
{\sc J.~Shah}, {\em Parametrized higher category theory {II}: {U}niversal constructions}.
\newblock Preprint, \href{https://arxiv.org/pdf/2109.11954}{arXiv:2109.11954}, 2022.

\bibitem{MR4587313}
\leavevmode\vrule height 2pt depth -1.6pt width 23pt, {\em Parametrized higher category theory}, Algebr. Geom. Topol., 23 (2023), pp.~509--644.

\bibitem{stacks}
{\sc {Stacks project authors}}, {\em The stacks project}.
\newblock \url{https://stacks.math.columbia.edu}.

\bibitem{MR4632370}
{\sc Y.~J.~F. Sulyma}, {\em Floor, ceiling, slopes, and {K}-theory}, Ann. K-Theory, 8 (2023), pp.~331--354.

\bibitem{Ull13}
{\sc J.~Ullman}, {\em On the slice spectral sequence}, Algebr. Geom. Topol., 13 (2013), pp.~1743--1755.

\end{thebibliography}
\bibliographystyle{siam}

\end{document}